\documentclass[a4paper,12pt, reqno]{amsart}
\usepackage{a4wide}
\usepackage{enumerate}

\usepackage[latin1]{inputenc}
\usepackage{subfigure}

\usepackage{dsfont}
\usepackage{amsmath, amssymb}
\usepackage{multirow}

\usepackage{color}
\usepackage{ifpdf}
\ifpdf %if using pdfLaTeX in PDF mode
  \usepackage[pdftex]{graphicx}
  \DeclareGraphicsExtensions{.pdf,.png,.mps}
  \usepackage{pgf}
  \usepackage{tikz}
\else %if using LaTeX or pdfLaTeX in DVI mode
  \usepackage{graphicx}
  \DeclareGraphicsExtensions{.eps,.bmp}
  \usepackage{pgf}
\fi
\usepackage{epic}
\usepackage{wrapfig}% package for wrapfigure environment
\usepackage{adjustbox}

\newtheorem{thm}{Theorem}%[section]
\newtheorem{cor}[thm]{Corollary}
\newtheorem{lem}[thm]{Lemma}

\newtheorem{defn}[thm]{Definition}

\numberwithin{equation}{section}
\newcommand{\rmk}{\par\noindent{\it Remark. }}

\usepackage{amsopn}

\newcommand\set[1]{\left\{#1\right\}}
\newcommand\abs[1]{\left|#1\right|}

\newcommand\floor[1]{\left\lfloor#1\right\rfloor}

\def\Z{\mathbb{Z}}
\def\z{\mathbf{z}}
\def\w{\mathbf{w}}
\def\N{\mathbb{N}}
\def\SS{\mathfrak{S}}
\def\D{\mathfrak{D}}

\def\B{\mathfrak{B}}
\def\M{\mathfrak{M}}

\def\ZZ{\mathcal{Z}}

\def\DD{\mathrm{DD}}
\def\DE{\mathrm{DE}}

\DeclareMathOperator\maj{maj}
\DeclareMathOperator\exc{exc}
\DeclareMathOperator\exca{exc_{A}}
\DeclareMathOperator\fexc{fexc}

\DeclareMathOperator\wex{wex}
\DeclareMathOperator\wexa{wex_{A}}
\DeclareMathOperator\wexc{wex_{c}}

\DeclareMathOperator\drop{drop}
\DeclareMathOperator\dropa{drop_{A}}
\DeclareMathOperator\dropc{drop_{c}}
\DeclareMathOperator\defi{drop}

\DeclareMathOperator\des{des}

\DeclareMathOperator\dd{dd}
\DeclareMathOperator\da{da}
\DeclareMathOperator\peak{peak}
\DeclareMathOperator\valley{valley}
\DeclareMathOperator\fmax{fmax}

\DeclareMathOperator\cros{cros}
\DeclareMathOperator\nest{nest}

\DeclareMathOperator\cpeak{cpeak}
\DeclareMathOperator\cvalley{cvalley}
\DeclareMathOperator\cdd{cdd}
\DeclareMathOperator\cda{cda}
\DeclareMathOperator\fix{fix}
\DeclareMathOperator\inv{inv}

\DeclareMathOperator\cdrop{cdrop}
\mathchardef\mhyphen="2D
\DeclareMathOperator\les{(31\mhyphen 2)}
\DeclareMathOperator\less{(13\mhyphen 2)}
\DeclareMathOperator\res{(2\mhyphen 13)}
\DeclareMathOperator\ress{(2\mhyphen 31)}

\DeclareMathOperator\LES{31\mhyphen 2}

\DeclareMathOperator\RESS{2\mhyphen 31}

\usepackage{cleveref}
\crefname{section}{§}{§§}
\def\312{\les}
\def\132{\less}
\def\213{\res}
\def\231{\ress}
\DeclareMathOperator\fixa{fix}

\DeclareMathOperator\fixc{fix_c}

\DeclareMathOperator\cnum{col}
\DeclareMathOperator\csum{csum}
\DeclareMathOperator\csumw{csum_{w}}
\DeclareMathOperator\csumd{csum_{d}}

\title[Symmetric unimodal expansions]%
{Symmetric unimodal expansions of excedances in colored permutations}
\author{Heesung Shin}
\address[Heesung Shin]{Inha University; 100 Inharo, Namgu, Incheon, 402-751, South Korea}
\email{shin@inha.ac.kr}

\author{Jiang Zeng}
\address[Jiang Zeng]{Universit\'{e} de Lyon;  Universit\'{e} Lyon 1; UMR 5208 du CNRS; Institut Camille Jordan;
43, boulevard du 11 novembre 1918, F-69622 Villeurbanne Cedex, France}
\email{zeng@math.univ-lyon1.fr}

\date{\today}
\begin{document}
\maketitle

\begin{abstract}
We consider several generalizations of the classical $\gamma$-positivity of Eulerian polynomials (and their derangement analogues) using generating functions and combinatorial theory of continued fractions. For the symmetric group, we prove an expansion formula for inversions and excedances as well as a similar expansion for derangements. We also prove the $\gamma$-positivity for Eulerian polynomials for derangements of type $B$. More general expansion formulae are also given for Eulerian polynomials for $r$-colored derangements. Our results  answer and generalize several recent open problems in the literature.
\end{abstract}

\tableofcontents
\section{Introduction and main results in $\SS_{n}$}

Let $\SS_n$ be the set of permutations  of $[n]=\set{1,\dots,n}$. For
any  permutation $\sigma\in \SS_{n}$, written as  the  word $\sigma=\sigma(1)\dots\sigma(n)$,
the entry $i\in [n]$ is called a \emph{descent} of $\sigma$ if $i<n$ and $\sigma(i)>\sigma(i+1)$; the entry $i \in [n]$  is called an \emph{excedance} (resp. \emph{drop}, \emph{fixed-point})  of $\sigma$ if $i < \sigma(i)$ (resp. $i>\sigma(i)$, $i=\sigma(i)$).
Denote the number of descents, excedances, drops, and fixed-points in $\sigma$ by $\des \sigma$, $\exc \sigma$, $\drop\sigma$, and $\fix\sigma$ respectively.
It is well known \cite{FS70} that the  statistics $\des$, $\exc$ and $\drop$  have the same distribution on~$\SS_n$ and their
common enumerative polynomial is
 the Eulerian polynomial:
\begin{align}\label{eq:eulerian}
A_n(t)=\sum_{\sigma \in \SS_n} t^{\des\sigma} =\sum_{\sigma \in \SS_n} t^{\exc\sigma} = \sum_{\sigma \in \SS_n} t^{\defi\sigma}.
\end{align}
Define the set  $\DD_{n,k} := \set{\sigma \in \SS_n : \des \sigma = k \text{ and } \dd (\sigma 0)=0}$, where $\dd(\sigma 0)$ is the number of  \emph{double descents} in the word $\sigma(1)\ldots \sigma(n)0$, i.e.,
indices $i$, $1<i\le n$,  such that $\sigma(i-1)>\sigma(i)>\sigma(i+1)$ with $\sigma(n+1)=0$.
For example,
$$\DD_{4,1} = \set{1324, 1423, 2314, 2413, 3412, 2134, 3124, 4123}.$$
Foata and Sch\"utzenberger~\cite{FS70}  proved the following expansion formula
\begin{align}\label{eq:peaka}
A_n(t) = \sum_{k=0}^{\lfloor (n-1)/2\rfloor}|\DD_{n,k}|  t^{k} (1+t)^{n-1-2k}.
\end{align}
%where $\gamma_{n,k} $ is the cardinality of $\DD_{n,k}$.

A finite sequence of positive integers $ a_0, \ldots ,a_n$ is \emph{unimodal} if there exists
an index $i'$ such that $a_i\leq a_{i+1}$ if $i <i'$
and $a_i>a_{i+1}$ otherwise; it is \emph{log-concave} if $a_{i}^2\geq a_{i-1}a_{i+1}$ for
all $i \in [1,n-1]$. A polynomial $p(x)=a_0+a_1x+\cdots +a_nx^n$ with nonnegative coefficients is unimodal (resp. log-concave) if and
only if the sequence of its coefficients is unimodal (resp. log-concave). It is well known that a polynomial
with nonnegative coefficients and with only \emph{real roots} is log-concave and that log-concavity implies the unimodality.
The set of polynomials whose coefficients are symmetric with center of symmetry $d/2$ is a vector space with a basis given
by $\{t^k(1+t)^{d-2k}\}^{\lfloor d/2\rfloor}_{k=0}$.  If a polynomial $p(t)$ has nonnegative coefficients when expanded in this
basis, then the coefficients in the standard basis $\{t^k\}_{k=0}^\infty$ form a unimodal
sequence (being a nonnegative sum of symmetric and unimodal sequences with the
same center of symmetry). Hence the expansion \eqref{eq:peaka} implies  both the symmetry and the unimodality of the Eulerian numbers.

 Analogues of \eqref{eq:peaka}  for types $B$ and $D$ were  given
 by Petersen~\cite{Pet07}, Stembridge~\cite{Ste08} and Chow~\cite{Cho08}.
In the recent years several $q$-analogues of  this formula have been  established in  \cite{SW10,SZ12,HJZ13,LZ14} by combining one of the Eulerian statistics in \eqref{eq:eulerian}
and one of the two classical Mahonian statistics, i.e., inversion numbers and Major index.
Recall that for $\sigma\in \SS_n$ the inversion number  $\inv\sigma$ is  the number of
pairs $(i,j)$ such that $i<j$ and $\sigma(i)>\sigma(j)$ and the major index $\maj\sigma$ is the sum of the descent positions.
%Using an unpublished result of Gessel, Shareshian and Wachs (See \cite[page 14]{Wac11})
%%(see Unimodality of $q$-Eulerian numbers and $q,p$-Eulerian numbers, AMS special session, Georgia Southern University, March 13, 2011)
%proved the following
%\emph{MAJ} $q$-analogue of \eqref{eq:peaka}:
%%\begin{thm}[Shareshian and Wachs] We have
%\begin{align}\label{maj:exc}
%\sum_{\sigma \in \SS_n} t^{\exc\sigma} q^{\maj\sigma-\exc\sigma} =
%\sum_{k=0}^{\lfloor (n-1)/2\rfloor} \biggl( \sum_{\sigma \in  \DD_{n,k}}q^{\maj\sigma^{-1}} \biggr) t^{k} (1+t)^{n-1-2k}.
%\end{align}
%%\end{thm}
\begin{table}\label{table:gamma_q}
\begin{center}
\begin{tabular}{c|ccc}
  n $\setminus$ k& 0& 1 & 2  \\
  \hline
  1 & 1 &       \\
  2 & 1 &       \\
  3 & 1 & $q(1+q)$ &    \\
  4 & 1 & $q(1+q)(q^2+q+2)$ \\
  5 & 1 & $q(1+q)({q}^{4}+2{q}^{3}+3{q}^{2}+2 q+3)$ & $q^2(1+q)^2(q^4+q^3+q^2+1)$ \\
\end{tabular}
\end{center}
\caption{The values of $\gamma_{n,k}(q)$ for $1\leq k\leq n\leq 5$}
\end{table}
In \cite{SZ12} we have shown that some
symmetric unimodal  expansion formulas of Eulerian polynomials
follow easily from  the combinatorics of the continued fraction expansions of Jacobi type of their ordinary generating functions.
In this paper,    we shall  prove an
\emph{INV} $q$-analogue of \eqref{eq:peaka} as well as some expansions for the Eulerian polynomials of type $B$ by  an analogue study of the correspondence between signed permutations and the weighted lattice paths.

For  $\sigma \in \SS_n$,
the statistic $\les \sigma$ (resp.  $\less \sigma$ )
 is the number of pairs $(i,j)$ such that $2\leq i<j\leq n$ and $ \sigma(i-1)>\sigma(j)>\sigma(i)$
 (resp.  $ \sigma(i-1)<\sigma(j)<\sigma(i)$).
Similarly,
the statistic $\res\sigma $ (resp. $\ress \sigma $)
is the number of pairs $(i,j)$ such that $1\leq i<j\leq n-1$ and $\sigma(j+1)>\sigma(i)>\sigma(j)$ (resp.
$\sigma(j+1)<\sigma(i)<\sigma(j)$).
% of which the first part  confirms a conjecture in \cite[Conj. 3.1]{BP12}, by
%using the combinatorial theory of continued fractions.
%%%%%%%%%%%%%%
\begin{thm}\label{thm:inv-Eulerian}
We have
\begin{align}\label{BPconj}
\sum_{\sigma\in \SS_n}t^{\exc \sigma}q^{\inv\sigma-\exc\sigma}=\sum_{0\leq k\leq (n-1)/2}\gamma_{n,k}(q) t^k(1+t)^{n-1-2k},
\end{align}
where the coefficient $\gamma_{n,k}(q)$
has the following combinatorial interpretation
$$
\gamma_{n,k}(q)=\sum_{\sigma\in \DD_{n,k}} q^{2\res \sigma+\les \sigma}
$$
and is  divisible by $q^k(1+q)^k$ for $0\leq k\leq (n-1)/2$.
\end{thm}

\begin{rmk}
%\begin{enumerate}[(i)]
%\item In \cite[Conj. 3.1]{BP12} Blanco and Petersen  speculated the existence of  polynomials
%$\gamma_{n,k}(q)\in \N[q]$ satisfying \eqref{BPconj}.
%\item
The first values of $\gamma_{n,k}(q)$ in Table~\ref{table:gamma_q} are obtained by replacing $p=q^2$ in
\cite[Table of $a_{n,k}(p,q)$ in Appendix]{SZ12}.
%\end{enumerate}
\end{rmk}
%%%%%%%%%%%%%%%

%Hence, for $n=5$, we have
%\begin{multline}
%A_5(q,t/q)=(1+t)^4+q \left( q+1 \right)  \left( {q}^{4}+2\,{q}^{3}+3\,{q}^{2}+2\,q+3
% \right)t(1+t)^2\\
%+{q}^{2} \left( q+1 \right)^{2} \left( {q}^{4}+{q}^{3}+{q}^{2}+1 \right).\nonumber
%\end{multline}
%

There is a derangement analogue of \eqref{BPconj}.
An element $\sigma\in \SS_n$  is called a \emph{derangement} if it has no fixed point.
Denote by $\D_n$ the set of all derangements in $\SS_n$ and let
$$
\DE_{n,k} := \set{\sigma \in \D_n : \exc \sigma = k \text{ and } \cda (\sigma)=0},
$$
where $\cda(\sigma)$ is
the number of  \emph{cyclic double ascents} (or \emph{double excedances}) of  $\sigma$, i.e.,
 indices $i\in[n]$ such that $i<\sigma(i)<\sigma^2(i)$.
For example, we have $\DE_{4,1} = \set{4123}=\set{(4321)}$ and
$$\DE_{4,2} = \set{2143, 3412, 4321, 4312, 3421}=\set{(21)(43), \, (31)(42),\, (41)(32),\, (4231),\, (4132)}.$$

%2\nest \sigma+\cros\sigma+The next result follows from \cite[Corollary 11]{SZ12}.
\begin{thm}\label{Der-A}
For $n\geq 1$
we have
\begin{align}\label{derangeA}
\sum_{\sigma\in \D_n}q^{\inv\sigma}t^{\exc \sigma}
=\sum_{0\leq k\leq {n/2}} \biggl(\sum_{\sigma\in \DE_{n,k}}q^{\inv\sigma}\biggr) t^k (1+t)^{n-2k}.
\end{align}
%where $\DE_{n,k}$ is the set of permutations in $\D_n$
%with exactly $k$ excedances and without double excedance.
\end{thm}
\begin{table}\label{table:inv_DEnk}
\begin{center}
\begin{tabular}{c|ccc}
  n $\setminus$ k& 0& 1 & 2  \\
  \hline
  0&1\\
  1 & 0 &      \\
  2 & 0 & $q$      \\
  3 & 0 & $q^2$ &    \\
  4 & 0 & $q^3$&$q^2+q^4+2q^5+q^6$ \\
  %5 & 1 & $?$ & $?$ \\
\end{tabular}
\end{center}
\caption{The values of $\sum  q^{\inv\sigma}$ for $\sigma\in \DE_{n,k}$ and $1\leq k\leq n\leq 4$}
\end{table}
\begin{rmk} The $t=-1$ case of Theorems~1 and 2 were first proved in \cite[Theorem~3]{SZ10}
as a $q$-analogue of two well-known  results connecting Eulerian polynomials (resp.
derangment polynomials) to tangent (resp. secant) numbers.
\end{rmk}

Note that Athanasiadis and Savvidou~\cite{AS12} used
the $q=1$ case of \eqref{derangeA}  to prove the $\gamma$-positivity
of  the local $h$-polynomials  of certain simplicial divisions.
More recently,  Athanasiadis and Savvidou~\cite{AS13} and
 Athanasiadis~\cite{Ath13} have extended
their  results to  hyperoctahedral group $B_n$ and colored groups, while
Mongelli~\cite{Mon13} studied the excedances in affine Weyl groups.
Motivated by the aforementioned works, we shall study the
 similar expansion formula for two kinds of  derangement Eulerian polynomials in
 $\B_n$ and wreath product $\Z_r \wr \SS_n$.  We shall introduce
 the necessary  definitions and present  the main results for these derangement
 polynomials  in the next section.

%%%%%%%%%%%%%%%%%%%%%
\section{Excedances and derangements in  colored permutations}
\subsection{Excedances  in  $\B_n$}
Let $\B_n$ be the set of permutations $\sigma$  of $\{\pm 1,\ldots, \pm n\}$ such that $\sigma(-i)=-\sigma(i)$ for every $i\in [n]$, that is, $\B_n = \Z_2 \wr \SS_n$ with $-i = \bar{i}$.
From Steingr\'{i}msson \cite[Definition 3]{Ste94}, an index $i \in [n]$ is called an \emph{excedance} of $\sigma \in \B_n$ if $$i <_{f} \sigma(i)$$
in the \emph{friends order} $<_f$ of $\{\pm 1,\ldots, \pm n\}$:
$$
1<_{f}-1<_{f}2<_{f}-2<_{f} \cdots <_{f} n <_{f} -n
$$
and $\exc(\sigma)$ is defined by the number of excedances of $\sigma \in \B_n$.
Following Brenti~\cite[page 431]{Bre94}, %(see also  \cite{CTZ09})
we say that $i \in [n]$ is a $B$-\emph{excedance} of $\sigma$
if $$\sigma(i) < \sigma(\abs{\sigma(i)})\quad \text{or}\quad \sigma(i)=-i$$
in the \emph{natural order} $<$ of $\{\pm 1,\ldots, \pm n\}$:
$$
-n<\cdots <-2<-1<1<2<\cdots <n.
$$
and the number of $B$-excedances of $\sigma$ will be denoted by $\exc_B(\sigma)$.
We also say that $i\in [0,n-1]$ is a \emph{B-descent} of $\sigma$ if
$$\sigma(i) > \sigma(i-1)$$
in the natural order, where $\sigma(0)=0$, and the number of $B$-descents of $\sigma$ will be denoted by $\des_B(\sigma)$.
From \cite[Theorem 20]{Ste94} and \cite[eq.(14)]{Bre94} it follows that
\begin{align}
\sum_{n\ge 0} \sum_{\sigma\in \B_n}t^{\exc(\sigma)}  \frac{z^n}{n!}
=\sum_{n\ge 0} \sum_{\sigma\in \B_n}t^{\des_B(\sigma)}  \frac{z^n}{n!}
=\sum_{n\ge 0} \sum_{\sigma\in \B_n}t^{\exc_B(\sigma)} \frac{z^n}{n!}
=\frac{(1-t) e^{z(1-t)}}{1-te^{2z(1-t)}},
\label{eq:equiv}
\end{align}
i.e., the two statistics $\exc$ and $\exc_B$ are equidistributed on $\B_n$.
In 2004, Fire \cite[Definition 6.18]{Fir04} introduced the \emph{flag excedance} statistic $\fexc$ of an element $\sigma\in \B_n$ by
\begin{align*}
\fexc(\sigma)
&=\#\set{i\in \{\pm 1,\ldots, \pm n\}\,:\, i <_c \sigma(i)}\\
&=2 \cdot \#\set{i\in [n] \,:\, i < \sigma(i)}+\#\{i\in [n]: \sigma(i)<0\},
%\\&=2\exc(\sigma)+\#\{i\in [n]: \sigma(i)<0\},
\end{align*}
where we use the \emph{color order} $<_{c}$ of $\{\pm 1,\ldots, \pm n\}$:
$$
-1<_{c}-2<_{c}\cdots <_{c}-n<_{c}1<_{c}2<_{c}\cdots <_{c}n.
$$

\subsection{Excedances in wreath product $\Z_r \wr \SS_n$}\label{subsec:wreath}
The wreath product $\Z_r \wr \SS_n$ is the set of \emph{colored permutations} ${\pi \choose \z}$,
where $\pi = \pi_1 \dots \pi_n \in \SS_n$ and $\z = (\z_1, \dots, \z_n) \in [0,r-1]^n$.
The number $z_i$ will be considered as the color assigned to $\pi_i$.
In this generalized symmetric group $\Z_r \wr \SS_n$, the product is defined by
$${\tau \choose \w}{\pi \choose \z} = {\tau \pi \choose \pi(\w) + \z},$$
where the composition $\tau\pi = \tau \circ \pi$ in $\SS_n$ is composed from right to left,
$$\pi(\w) = (\w_{\pi_1},\w_{\pi_2},\dots, \w_{\pi_n}),$$
and the addition is summed modulo $r$ in individual coordinate.

From now, for $\sigma = \sigma(1) \dots \sigma(n) = {\pi_1 \dots \pi_n \choose \z_1 \dots \z_n} \in \Z_r \wr \SS_n$, we write
\begin{align}
\text{$\sigma(i) = {\pi_i \choose  \z_i}$ as ${\pi_i}^{[\z_i]}$,}
\label{eq:sigma_i}
\end{align}
while $\pi_i$ (resp. $\z_i$) is called the \emph{number} (resp. the \emph{color}) of ${\pi_i}^{[\z_i]}$, denote by
$$\pi_i = \abs{\pi_i^{[\z_i]}} \text{ and } \z_i=\cnum({\pi_i}^{[\z_i]}).$$
Conventionally, we use also the window notation $i = i^{[0]}$, $\bar{i} = i^{[1]}$, and $\bar{\bar{i}} = i^{[2]}$.
For an example
$$\sigma = {4~7~2~5~1~6~3 \choose 0~1~0~1~2~0~0} = 4~\bar{7}~2~\bar{5}~\bar{\bar{1}}~6~3 \in \Z_3 \wr \SS_7,$$
we have ${1 \choose 2} = 1^{[2]} = \bar{\bar{1}}$, ${7 \choose 1} = 7^{[1]} = \bar{7}$, $\abs{\bar{5}} = 5$, and $\cnum(\bar{5}) =1$.
So the wreath product $\Z_r \wr \SS_n$ could be considered the set of $r$-colored permutations $\sigma$ of the alphabet $\Sigma$ of $rn$ letters
$$\Sigma := \set{1,\bar{1},\bar{\bar{1}},\dots, 1^{[r-1]},2,\bar{2},\bar{\bar{2}},\dots, 2^{[r-1]}, \dots, n,\bar{n},\bar{\bar{n}},\dots, n^{[r-1]}},$$
such that $\sigma(\bar{i}) = \overline{\sigma(i)}$ for all $i \in \Sigma$.
From Steingr\'{i}msson \cite{Ste94}, an index $i \in [n]$ is called an \emph{excedance} of $\sigma \in \Z_r \wr \SS_n$ if $i <_{f} \sigma(i)$
in the \emph{friends order} $<_f$ of $\Sigma$:
$$
1<_{f}\bar{1}<_{f} \dots <_{f}1^{[r-1]}<_{f}
2<_{f}\bar{2}<_{f} \dots <_{f}2^{[r-1]}<_{f}
\dots<_{f}
n<_{f}\bar{n}<_{f} \dots <_{f}n^{[r-1]}.
$$
Let $\exc(\sigma)$ be the number of excedances of $\sigma \in \Z_r \wr \SS_n$.
Bagno and Garber \cite[Definition 3.4]{BG06} define the \emph{flag excedance} statistic $\fexc$ of an element $\sigma\in \Z_r \wr \SS_n$ by
\begin{align*}
\fexc(\sigma)
&=\#\set{i\in \Sigma \,:\, i <_c \sigma(i)}\\
&=r \cdot \#\set{i\in [n] \,:\, i <_c \sigma(i) } + \sum_{i\in [n]} \cnum({\sigma(i)}),
%\\&=2\exc(\sigma)+\#\{i\in [n]: \sigma(i)<0\},
\end{align*}
where we use the \emph{color order} $<_{c}$ of $\Sigma$:
$$
1^{[r-1]}<_{c}2^{[r-1]}<_{c} \dots <_{c}n^{[r-1]} <_{c}
\dots<_{c}
\bar{1}<_{c}\bar{2}<_{c} \dots <_{c}\bar{n} <_{c}
1<_{c}2<_{c} \dots <_{c}n.
$$

%In the case $r=2$, for $\sigma \in \B_n$, we can recall the definition
%\begin{align*}
%\exc(\sigma) &= \# \set{i \in [n] ~:~ i <_{f} \sigma(i)}\\
%\fexc(\sigma) &= \# \set{i \in \Sigma ~:~ i <_{c} \sigma(i)},
%\end{align*}
%where $\Sigma = \set{\pm 1, \dots, \pm n}$.

\subsection{Derangements in $\Z_r \wr \SS_n$}
An element $\sigma\in \Z_r \wr \SS_n$  is called a \emph{derangement} if it has no fixed point, that is, $\sigma(i)\neq i$ for all $i\in\Sigma$.
Denote by $\D_n^{(r)}$ the set of all derangements in $\Z_r \wr \SS_n$.
The two corresponding derangement Eulerian polynomials on $\D_n^{(r)}$ are defined by
\begin{align}
D_n^{(r)}(t)=\sum_{\sigma\in \D_n^{(r)}}t^{\fexc(\sigma)} = \sum_{k\ge 0} D_{n,k}^{(r)} t^{k} \label{def:BD}
\intertext{and}
d_n^{(r)}(t)=\sum_{\sigma\in \D_n^{(r)}}t^{\exc(\sigma)} = \sum_{k\ge 0} d_{n,k}^{(r)} t^{k}. \label{def:Bd}
\end{align}
%where $D_{n,k}^{(r)}$ and $d_{n,k}^{(r)}$ are coefficients of $t^k$ in two polynomials $D_{n}^{(r)}(t)$ and $d_{n}^{(r)}(t)$.
%An element $\sigma\in \B_n$  is called a \emph{derangement} if it has no fixed point, that is, $\sigma(i)\neq i$ for all $i\in[n]$.
%Denote by $\D_n^{(2)}$ the set of all derangements in $\B_n$.
%It is well-known that two statistics $\exc$ and $\exc_B$ are equidistributed on $\B_n$ and $\D_n^{(2)}$.
%The two corresponding derangement Eulerian polynomials of type $B$ are defined by
%\begin{align}\label{def:BD}
%D_n^{(2)}(t)=\sum_{\sigma\in \D_n^{(2)}}t^{\fexc(\sigma)}
%\quad\textrm{and}\quad
%d_n^{(2)}(t)=\sum_{\sigma\in \D_n^{(2)}}t^{\exc(\sigma)}=\sum_{\sigma\in \D_n^{(2)}}t^{\exc_B(\sigma)}.
%\end{align}
By convention we set $D_0^{(2)}(t)=d_0^{(2)}(t)=1$.
The first values of $D_n^{(2)}(t)$ and $d_n^{(2)}(t)$ are as follows:
\begin{align*}
%D_0^{(2)}(t)&=1,                                   &d_0^{(2)}(t)&=1,\\
D_1^{(2)}(t)&=t,                                   &d_1^{(2)}(t)&=t,\\
D_2^{(2)}(t)&=t+3t^2+t^3,                          &d_2^{(2)}(t)&=4t+t^2,\\
D_3^{(2)}(t)&=t+7t^2+13t^3+7t^4+t^5,               &d_3^{(2)}(t)&=8t+20t^2+t^3,\\
D_4^{(2)}(t)&=t+15t^2+57t^3+87t^4+57t^5+15t^6+t^7, &d_4^{(2)}(t)&=16t+144 t^2+ 72 t^3 +t^4.
\end{align*}

The polynomials $d_n^{(2)}(t)$ were first studied by Chen et al.~\cite{CTZ09} and Chow~\cite{Cho09}. In particular,  the polynomials  $d_n^{(2)}(t)$ have only real roots, see \cite{Cho09}.
Recently Mongelli~\cite{Mon13} has proved the symmetry $D_{n,j}^{(2)} = D_{n,2n-j}^{(2)}$ for all $0\le j \le n $ and conjectured its unimodality. Note  that
 the polynomials $D_3^{(2)}(t)$ has non-real complex roots.
We will prove  similar expansions for the derangement Eulerian polynomials $D_n^{(r)}(t)$ and $d_n^{(r)}(t)$ when  $r\ge 1$. In particular,
we prove the $\gamma$-positivity of $D_n^{(2)}(t)$, which implies both  symmetry and unimodality of its coefficients, see \eqref{derangeB}.

\subsection{Main results in $\Z_r \wr \SS_n$}
%%%%%%%%%%%%%%%
For any  integer $n\geq 0$  define   $[n]_{q}=\frac{1-q^n}{1-q}$ as its  $q$-analogue.
 Let $\gamma_{n,i,j}$ be the number of permutations in $\SS_n$ with exactly $i$ fixed-points, $j$ excedances and without double excedance, i.e.,
 $$
 \gamma_{n,i,j}=|\{\sigma\in \SS_n\,:\,\fix\sigma=i,\; \exc\sigma=j, \; \cda\sigma=0 \}|.
 $$
\begin{thm}\label{thm:DerangeR}
 For $r \ge 1$, we have  $\gamma_{n,i,j} > 0$ for $1\le i+2j \le n$,
\begin{align}
D_n^{(r)}(t)&=\sum_{1\le i+2j\le n} \gamma_{n,i,j} \, t^{i+j} (1+t)^{n-i-2j} ([r-1]_t)^i
([r]_t)^{n-i}, \label{derangeR}
\intertext{and}
d_n^{(r)}(t)&=\sum_{1\le i+2j\le n} \gamma_{n,i,j} \, t^{i+j} (1+t)^{n-i-2j} (r-1)^i r^{n-i}.
\end{align}
\end{thm}
The special $r=1$ case of \eqref{derangeR} corresponds to the $q=1$ case of \eqref{derangeA}. We derive immediately from Theorem~\ref{thm:DerangeR} the following result.\begin{cor}\label{cor:property}
For all $r\ge 1$ and $n\ge 1$, the polynomial $D_n^{(r)}(t)$ is strictly unimodal and symmetric, namely, the coefficients $D_{n,k}^{(r)}$ in \eqref{def:BD} satisfy
\begin{align}
0 \le D_{n,0}^{(r)} < D_{n,1}^{(r)} < \dots <D_{n,\lfloor rn/2 \rfloor}^{(r)}.
\end{align}
and $D_{n,j}^{(r)} = D_{n,rn-j}^{(r)}$ for all $0\le j \le \lfloor rn/2 \rfloor$.
\end{cor}
\begin{proof}
Note that  each summand  at the right-hand side of \eqref{derangeR} is unimodal and symmetric with the
same center of symmetry at
$$(i+j)+\frac{n-i-2j}{2}+\frac{(r-2)i}{2}+\frac{(r-1)(n-i)}{2} = \frac{rn}{2}.\qedhere$$
\end{proof}

Let $\gamma_{n,k}^{(2)}$ be  the number of permutations in $\SS_n$
with exactly $k$ weak excedances and without double excedance, i.e.,
$\gamma_{n,k}^{(2)} = \sum_{i+j=k} \gamma_{n,i,j}.$
For example, the  permutations (written in product of disjoint cycles)
in $\SS_4$ with exactly 2 weak excedances and without double excedance are
as follows:
$$ (1)(432),~ (2)(431),~ (3)(421),~ (4)(432),~ (21)(43),~ (31)(42),~ (32)(41),~ (4132),~ (4231). $$
Thus $\gamma_{4,2}^{(2)}=9$.

For $r=2$, Theorem~\ref{thm:DerangeR} reduces to the following result, of which \eqref{derangeB}
infers  the $\gamma$-positivity of $D_n^{(2)}(t)$.

\begin{cor}\label{thm:DerangeB}
We have  $\gamma_{n,k}^{(2)}>0$ for $1\leq k\leq n$,
\begin{align}
D_n^{(2)}(t)&=\sum_{k=1}^{n} \gamma_{n,k}^{(2)}\, t^{k} (1+t)^{2n-2k},\label{derangeB}
\intertext{and}
d_n^{(2)}(t)&=\sum_{1\le i+2j\le n} \gamma_{n,i,j} \, 2^{n-i}t^{i+j}(1+t)^{n-i-2j}.
\end{align}
\end{cor}

%%%%%%%%%%%%%%%%
%\begin{thm}\label{thm:DerangeB}
%Let $\gamma_{n,k}^{(2)}$ be the number of permutations in $\SS_n$ with exactly $k$ weak excedances and without double excedance. Then $\gamma_{n,k}^{(2)}>0$ for $1\leq k\leq n$ and
%\begin{align}\label{derangeB}
%D_n^{(2)}(t)=\sum_{k=1}^{n} \gamma_{n,k}^{(2)}\, t^{k} (1+t)^{2n-2k}.
%\end{align}
%\end{thm}
%%%%%%%%%%%%%%%%

\begin{table}[t]
\begin{minipage}[l]{.48\textwidth}
\begin{tabular}{c| rrrrrrr}
$n \backslash k$& 0&1&2&3&4&5&6\\
\hline
0&1&&&&&\\
1&0&1&&&&&\\
2&0&1&1&&&&\\
3&0&1&3&1&&&\\
4&0&1&9&6&1&&\\
5&0&1&23&35&10&1&\\
6&0&1&53&184&95&15&1\\
\end{tabular}
\caption{The values of $\gamma_{n,k}^{(2)}$ for $0\leq k\leq n\leq 6$}
\label{table:gamma}
\end{minipage}
\begin{minipage}[r]{.48\textwidth}
\begin{tabular}{c| rrrrrrr}
$n \backslash k$& 0&1&2&3&4&5&6\\
\hline
0&1&&&&&\\
1&0&1&&&&&\\
2&0&1&3&&&&\\
3&0&1&7&11&&&\\
4&0&1&15&54&57&&\\
5&0&1&31&197&458&361&\\
6&0&1&63&648&2551&4379&2763\\
\end{tabular}
\caption{The values of $\hat{\gamma}_{n,k}^{(2)}$ for $0\leq k\leq n\leq 6$}
\label{table:gamma}
\end{minipage}
\end{table}
The first values of $\gamma_{n,k}^{(2)}$ are given in Table~\ref{table:gamma}. For example,
we have
$$D_4^{(2)}(t)=t(1+t)^6 + 9t^2(1+t)^4 + 6t^3(1+t)^2 + t^4.$$

The following result is equivalent to
\cite[Proposition 2.3]{Ath13}, while we will give an alternative
proof  again using the generating functions.
%%%%%%%%%%%%
\begin{thm}\label{thm:fexc}
For all $r \ge 1$, the coefficients $D_{n,k}^{(r)}$ and $d_{n,k}^{(r)}$ of the
polynomials $D_{n}^{(r)}(t)$ and $d_{n}^{(r)}(t)$ (cf. \eqref{def:BD} and \eqref{def:Bd})  are connected by $d_{n,0}^{(r)} = D_{n,0}^{(r)}$ and
\begin{align}
d_{n,k}^{(r)} = \sum_{j=0}^{r-1} D_{n,rk-j}^{(r)}\quad \text{for all}\quad
 1 \le k \le n.\label{eq:fexc}
\end{align}
\end{thm}
For example, when $n=4$ and $r=k=2$ we have $144=57+87$.
In \cite[Theorem~4.6]{CTZ09}, the so-called \emph{spiral property} of the sequence $\{d_{n,k}^{(2)}\}_{0 \leq k\leq n}$ was proven, namely,
$$0\le d_{n,0}^{(2)}<d_{n,n}^{(2)}<d_{n,1}^{(2)}<d_{n,n-1}^{(2)}<d_{n,2}^{(2)}<d_{n,n-2}^{(2)}<\dots<d_{n,\lceil n/2 \rceil}^{(2)}.$$
The following theorem generalizes the above spiral property to the sequence $\{d_{n,k}^{(r)}\}_{0\leq k\leq n}$ for all $r\ge 2$.
\begin{cor} For all $r\ge 2$ and $n\geq 0$, the polynomial $d_{n}^{(r)}(t)$ has the spiral property, namely,
$$
d_{n,k}^{(r)}<d_{n,n-k}^{(r)}<d_{n,k+1}^{(r)}, \quad \text{for all $0\leq k < \lfloor n/2 \rfloor$},
$$
and $d_{n,\lfloor n/2 \rfloor }^{(r)}<d_{n,\lceil n/2 \rceil}^{(r)}$ if $n$ is odd,
where $d_{n,k}^{(r)}$ is the coefficient of $t^k$ in the polynomial $d_{n}^{(r)}(t)$.
%$$
%d_{n,n}<d_{n,1}<d_{n,n-1}<d_{n,2}<d_{n,n-2}<\cdots<
%$$
\end{cor}
\begin{proof}

For $r\ge 1$ and $n\geq 1$,
letting $a^{(r)}_{n,k} = D_{n,rk}^{(r)}$ and $b^{(r)}_{n,k} =\sum_{j=1}^{r-1} D_{n,rk-j}^{(r)}$,
Corollary~\ref{cor:property} gives that
these two sequences $\{ a^{(r)}_{n,k} \}_{0\le k \le n}$ and $\{ b^{(r)}_{n,k} \}_{1\le k \le n}$ are also strict unimodal and symmetric, that is,
%\begin{align}
%\label{eq:ssu}
%\begin{cases}
%0\le a^{(r)}_{n,0} = a^{(r)}_{n,n}<a^{(r)}_{n,1}=a^{(r)}_{n,n-1}<\cdots <a^{(r)}_{n,\lfloor n/2 \rfloor} =a^{(r)}_{n,\lceil n/2 \rceil},\\
%0< b^{(r)}_{n,1} =b^{(r)}_{n,n}<b^{(r)}_{n,2}=b^{(r)}_{n,n-1}<\cdots <b^{(r)}_{n,\lfloor (n+1)/2 \rfloor} =b^{(r)}_{n,\lceil (n+1)/2 \rceil},
%\end{cases}
%\end{align}
\begin{align}\label{eq:ssu}
&0\le a^{(r)}_{n,0}<a^{(r)}_{n,1}<\cdots <a^{(r)}_{n,\lfloor n/2 \rfloor},&
&0< b^{(r)}_{n,1}<b^{(r)}_{n,2}<\cdots <b^{(r)}_{n,\lceil n/2 \rceil},
\end{align}
and $a^{(r)}_{n,k}=a^{(r)}_{n,n-k}$ for $0\leq k\leq \lfloor n/2 \rfloor$ and $b^{(r)}_{n,k}=b^{(r)}_{n,n+1-k}$ for $1\leq k\leq \lceil n/2 \rceil$.
By Theorem~\ref{thm:fexc}, we have $d_{n,0}^{(r)} = a^{(r)}_{n,0}$ and
\begin{align}
d_{n,k}^{(r)} = a^{(r)}_{n,k} + b^{(r)}_{n,k} \quad \text{for all $1 \le k \le n$}. \label{eq:oe}
\end{align}
Combining \eqref{eq:ssu} and \eqref{eq:oe},
the sequence $\{d_{n,k}^{(r)}\}_{1\leq k\leq n}$ has a spiral property for all $r\ge 2$ as follows.

\end{proof}
\rmk Let $a_n^{(r)}(t) = \sum_{k=0}^n a_{n,k}^{(r)} t^{k}$ and $b_n^{(r)}(t) = \sum_{k=1}^n b_{n,k}^{(r)} t^{k}$ such that $$d_n^{(r)}(t) = a_n^{(r)}(t) + b_n^{(r)}(t).$$
Observing the equation \eqref{eq:odd} in proof of Theorem~\ref{thm:fexc}, we have
\begin{align}
\sum_{n\geq 0} a_{n}^{(r)}(t) \frac{z^{n}}{n!}
&=\frac{e^{(r-1)t z}-t e^{(r-1)z}}{e^{rt z}-te^{rz}},\\
\sum_{n\geq 0} b_{n}^{(r)}(t) \frac{z^{n}}{n!}
&=\frac{t e^{(r-1)z} - t e^{(r-1)tz} }{e^{rt z}-te^{rz}}.
\end{align}
%Could you derive $\sum_{n\geq 0} a_{n}^{(r)}(t) z^{n}$ and $\sum_{n\geq 0} b_{n}^{(r)}(t) z^{n}$ as the continued fraction from above two equations using the addition formula of Rogers-Stieltjes?

Athanasiadis and Savvidou~\cite{AS13} proved that there are nonnegative integers
 $\xi^+_{n,k}$ and $\xi^-_{n,k}$ satisfying
\begin{align}\label{eq:as}
d_n^{(2)}(t)=\sum_{k\geq 0}\xi^+_{n,k}t^k(1+t)^{n-2k}+\sum_{k\geq 1}\xi^-_{n,k}t^k(1+t)^{n+1-2k}.
\end{align}

We show that the polynomials $D_n^{(2)}(t)$ have a similar expansion as \eqref{eq:as} with the same coefiicients $\xi_{n,k}^+$ and $\xi_{n,k}^- $  and also derive a  combinatorial interpretation from
Theorem~\ref{thm:DerangeB}.
A permutation $\sigma\in S_n$ is called a \emph{drop-colored permutation} if
some of the drops of $\sigma$ are colored by, say, \emph{bar} or \emph{tilde}.
We denote by $\cdrop(\sigma)$ the number of colored drops in $\sigma$.
Denote by $\hat{\gamma}_{n,k}^{(2)}$ the number of  permutations  $\sigma$ in $\ZZ_n$ whose
total number of weak excedances and colored drops is $k$, namely,
$\wex(\sigma)+\cdrop(\sigma)=k$.
Let $\ZZ_n$ be the set of drop-colored permutations without double excedance in $\SS_n$. For example,
the drop-colored permutations (written in product of disjoint cycles) in $\ZZ_4$ whose
total number of weak excedances and colored drops is 2 are:
\begin{gather*}
(1)(4\,3\,2),~ (2)(4\,3\,1),~ (3)(4\,2\,1),~ (4)(3\,2\,1),~ (2\,1)(4\,3),~ (3\,1)(4\,2),~ (4\,1)(3\,2),\\
(4\,2\,3\,1),~(4\,1\,3\,2),~ (\bar{4}\,3\,2\,1),~ (\tilde{4}\,3\,2\,1),~ (4\,\bar{3}\,2\,1),~ (4\,\tilde{3}\,2\,1),~ (4\,3\,\bar{2}\,1),~ (4\,3\,\tilde{2}\,1).
\end{gather*}
Thus $\hat{\gamma}_{4,2}^{(2)}=15$.
The first values of $ \hat{\gamma}_{n,k}^{(2)}$ are given in Table~\ref{table:gamma}.

\begin{thm}\label{thm:analogue-as}
For $0\leq k\leq n$ we have
\begin{align}\label{eq:bb}
\hat{\gamma}_{n,k}^{(2)} = \sum_{i=0}^k \gamma_{n,i}^{(2)} 2^{k-i} {n-i \choose k-i}.
\end{align}
Furthermore the following expansions hold true:
\begin{align}
D_n^{(2)}(t)
&=\sum_{k\geq 0}\hat{\gamma}^{(2)}_{n,k}t^{k}(1+t^2)^{n-k},\label{eq:SZ}\\
d_n^{(2)}(t)
&=\sum_{k\geq 0}\hat{\gamma}^{(2)}_{n,k}t^{\lceil k/2\rceil}(1+t)^{n-k},  \label{eq:AS}
\end{align}
where $\lceil x\rceil$ denotes the least integer great than or equal to $x$.
\end{thm}
%\begin{conj}\label{conjeutre}
%There are positive integers $\hat{\gamma}_{n,k}^{(r)}$ such that
%\begin{align}
%D_n^{(r)}(t)
%&=\sum_{k\geq 0}\hat{\gamma}^{(r)}_{n,k}t^{k}(1+t^r)^{n-\lceil k/r \rceil - \lfloor k/r \rfloor},\label{conj:eq:SZ}\\
%d_n^{(r)}(t)
%&=\sum_{k\geq 0}\hat{\gamma}^{(r)}_{n,k}t^{\lceil k/r \rceil}(1+t)^{n -\lceil k/r \rceil - \lfloor k/r \rfloor},  \label{conj:eq:AS}
%\end{align}
%where $\lceil x\rceil$ denotes the least integer great than or equal to $x$ and $\lfloor x\rfloor$ denotes the great integer less than or equal to $x$.
%\end{conj}
For example, we have
\begin{alignat*}{4}
%D_1^{(2)}(t)&=t,\\
%D_2^{(2)}(t)&=t(1+t^2)+3t^2,\\
%D_3^{(2)}(t)&=t(1+t^2)^2+7t^2(1+t^2)+11 t^3,\\
D_4^{(2)}(t)&=t(1+t^2)^3&&+15t^2(1+t^2)^2 &&+54t^3(1+t^2) &&+57 t^4,\\
d_4^{(2)}(t)&=t(1+t)^3  &&+15t(1+t)^2     &&+54t^2(1+t)   &&+57 t^2.
\end{alignat*}
Some remarks are in order:
\begin{enumerate}[(i)]
\item The two different polynomials $D_n^{(2)}(t)$ and $d_n^{(2)}(t)$ have the
same coefficients $\hat{\gamma}^{(2)}_{n,k}$, when they are expanded in the
basis $\set{t^{k}(1+t^2)^{n-k}}$ and $\set{t^{\lceil k/2 \rceil}(1+t)^{n-k}}$, respectively.
It would be interesting to have a combinatorial explanation of this coincidence.
\item Comparing  \eqref{eq:as} and \eqref{eq:AS} we derive that
\begin{align}\label{eq:atha}
\hat{\gamma}_{n,2k}^{(2)} = \xi^+_{n,k} \text{ for $k \ge 0$}
\quad\text{and}\quad
\hat{\gamma}_{n,2k-1}^{(2)} = \xi^-_{n,k}\text{ for $k \ge 1$}.
\end{align}
\item Athanisiadis~\cite{Ath13} has given an apparently different interpretation of the numbers in \eqref{eq:atha} using
linear statistics on permutations. Is it easy to relate these two interpretations?
\end{enumerate}

The rest of this paper is organized as follows.
In Section~\ref{sec:proof1}, after recalling the necessary results in our previous paper~\cite{SZ12} we prove Theorems~\ref{thm:inv-Eulerian} and \ref{Der-A}.
In Section~\ref{sec:proof2} we prove Theorems~\ref{thm:DerangeR}, \ref{thm:fexc}, and \ref{thm:analogue-as}.
Finally, in Section~\ref{sec:proof3}, we prove Lemma~\ref{thm:cfb} after recalling some background for the classical combinatorial theory of continued fractions.

%%%%%%%%%%%%%%%%%%%%%%%%%%%%%%%
\section{Proof of Theorems~\ref{thm:inv-Eulerian} and \ref{Der-A}}\label{sec:proof1}
\subsection{Preliminaries}
We need some more definitions.
%\begin{defn}\label{def:1}
For  $\sigma \in \SS_n$,  let   $\sigma(0)=\sigma(n+1)=0$.  Then any entry $\sigma(i)$ ($i\in [n]$)  can be classified according to one of the four cases:
\begin{itemize}
\item a \emph{peak}  if $\sigma(i-1) < \sigma(i)$ and $\sigma(i) > \sigma(i+1)$;
\item a \emph{valley}  if $\sigma(i-1) > \sigma(i)$ and $\sigma(i) < \sigma(i+1)$;
\item a \emph{double ascent}  if $\sigma(i-1) < \sigma(i)$ and $\sigma(i) < \sigma(i+1)$;
\item a \emph{double descent} if $\sigma(i-1) > \sigma(i)$ and $\sigma(i) > \sigma(i+1)$.
\end{itemize}
Let $\peak^* \sigma$ (resp. $\valley^* \sigma$, $\da^* \sigma$, $\dd^* \sigma$) denote  the number of peaks (resp. valleys, double ascents, double descents) in~$\sigma$. Clearly we have $\peak^* \sigma=\valley^* \sigma+1$.
Now,
%where $\SS_{n,k}$  is   the subset of permutations $\sigma\in \SS_{n}$ with exactly
% \red{$k$ valleys} and \red{without double descents}.
 it is easy to see that
 $$
\DD_{n,k}=\{\sigma\in \SS_n: \valley^* \sigma=k \;\; \textrm{and}\;\; \dd^* \sigma=0\}.
 $$
Define the generalized  Eulerian polynomial   by
\begin{align}\label{def:A}
A_n(p,q,t,u,v,w) = \sum_{\sigma\in \SS_n} p^{\res \sigma} q^{\les \sigma}  t^{\des\sigma}u^{\da^*\sigma} v^{\dd^*\sigma}w^{\valley^*\sigma}.
\end{align}

We need the known formula for the so-called Jacobi-Rogers polynomials~\cite{GJ83}.
\begin{lem}\label{lem:JR}
If  the ordinary generating function of the sequence $\{\mu_n\}_n$
has the continued fraction expansion
\begin{align*}
\sum_{n\geq 0}  \mu_n z^n=\cfrac{1}{1-b_0z-\cfrac{\lambda_1 z^2}{1-b_1z-\cfrac{\lambda_2 z^2}{\cdots}}},
\end{align*}
then
\begin{align}\label{JR}
\mu_n&=\sum_{h\geq 0}\sum_{n_0, \ldots, n_{h-1}\geq 1\atop m_0, \ldots, m_h\geq 0} b_0^{m_0}\cdots b_h^{m_h}\lambda_1^{n_{0}}\ldots
\lambda_h^{n_{h-1}} \rho({\mathbf n}, \mathbf{m}),
\end{align}
where
$2(n_0+\cdots +n_{h-1})+(m_0+\cdots +m_h)=n$ and
$$
\rho({\mathbf n}, \mathbf{m})=\prod_{j=0}^{h-1}{n_{j}+n_{j+1}-1\choose n_{j}-1}
\prod_{l=0}^{h}
{m_l+n_l+n_{l-1}-1\choose m_l}
$$
with convention $n_{-1}=1$, $n_{h}=0$, ${p\choose -1}=\delta_{p,-1}$.
\end{lem}

The following result was proved  in
 \cite[Theorem 2]{SZ12}. We provide below a new and neat proof by using the Jacobi-Rogers formula.
\begin{lem}\label{lem:a}
We have the expansion formula
\begin{equation}\label{eq:def}
A_n(p,q,t,u,v,w)= \sum_{k=0}^{\floor{(n-1)/2}} a_{n,k}(p,q) (tw)^k (u+vt)^{n-1-2k},
\end{equation}
where
\begin{equation}\label{eq:ank}
a_{n,k}(p,q)=\sum_{\sigma\in \DD_{n,k}} p^{\res \sigma} q^{\les \sigma}.
\end{equation}
and, for all $0\le k \le \floor{(n-1)/2}$, the polynomial $ a_{n,k}(p,q)$ is divisible by
$(p+q)^k$.
\end{lem}
\begin{proof}
By \cite[Eq. (28)]{SZ12} we have
\begin{align*}
\sum_{n\ge1} A_n(p,q,t,u,v,w)x^{n-1}
&=
\dfrac{1}
{1-(u+tv)[1]_{p,q}x-\dfrac{[1]_{p,q}[2]_{p,q}twx^2}
{1-(u+tv)[2]_{p,q}x-\dfrac{[2]_{p,q}[3]_{p,q}twx^2}
{\cdots}}},
\end{align*}
where $[n]_{p,q}=\frac{p^n-q^n}{p-q}$.
By Lemma~\ref{lem:JR}, there are $a_{n,k}(p,q)\in \N[p,q]$ such that
\begin{align}
A_n(p,q,t,u,v,w) = \sum_{k=0}^{\floor{(n-1)/2}} a_{n,k}(p,q) (u+tv)^{n-1-2k} (tw)^k.
\label{eq:interpret1}
\end{align}
Thus
$$
A_n(p,q,1,1,0,w)=\sum_{k\geq 0}a_{n,k}(p,q)w^k,
$$
and we derive the combinatorial interpretation \eqref{eq:ank} for $a_{n,k}(p,q)$
from \eqref{def:A}.
Finally, as
$$
(p+q) ~|~ [n]_{p,q} [n+1]_{p,q}\qquad  \textrm{for all $n\ge 1$},
$$
each  $w$ appears with  a factor $(p+q)$ in the expansion of the continued fraction, so
the polynomial $a_{n,k}(p,q)$   is divisible by  $(p+q)^k$.
\end{proof}

%%%%%%%%
\begin{defn}
For  $\sigma\in \SS_n$,  a  value $x=\sigma(i)$ ($i\in [n]$) is called
\begin{itemize}
\item a \emph{cyclic peak}  if $i=\sigma^{-1}(x)<x$ and $x >\sigma(x)$;
\item a \emph{cyclic valley}  if $i=\sigma^{-1}(x)>x$ and $x<\sigma(x)$;
\item a \emph{double excedance} if $i=\sigma^{-1}(x)<x$ and $x<\sigma(x)$; % \emph{cyclic double ascent}  or
\item a  \emph{double drop} if $i=\sigma^{-1}(x)>x$ and $x>\sigma(x)$;%\emph{cyclic double descent} or
\item a \emph{fixed point}  if $x=\sigma(x)$.
\end{itemize}
Let $\cpeak \sigma$ (resp. $\cvalley \sigma$, $\cda \sigma$, $\cdd \sigma$, $\fix \sigma$) be the number of cyclic peaks (resp. valleys, double excedances, double drops, fixed points) in $\sigma$.
\end{defn}

For a permutation $\sigma\in \SS_{n}$ the crossing and nesting numbers are defined by
\begin{align}
 \cros\sigma&= \# \set{(i,j)\in [n]\times[n]: (i<j\le\sigma(i)<\sigma(j))\vee(i>j>\sigma(i)>\sigma(j))},\\
\nest\sigma &= \# \set{(i,j)\in [n]\times[n]: (i<j\le\sigma(j)<\sigma(i))\vee(i>j>\sigma(j)>\sigma(i))}.
\end{align}

For example, the diagram of $ \sigma=9 \;3 \;7\;4\;6\;10\;5\;8\;1\;2$ is as follows:
$$
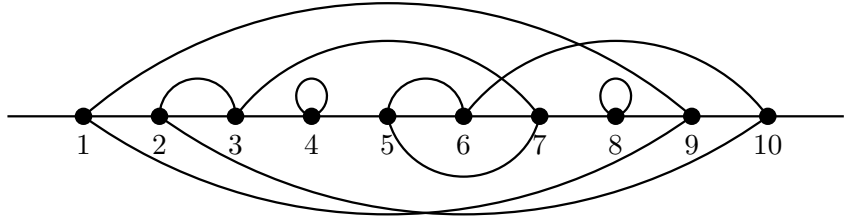
\begin{pgfpicture}{-2.00mm}{-15.00mm}{112.00mm}{17.00mm}
\pgfsetxvec{\pgfpoint{1.00mm}{0mm}}
\pgfsetyvec{\pgfpoint{0mm}{1.00mm}}
\color[rgb]{0,0,0}\pgfsetlinewidth{0.30mm}\pgfsetdash{}{0mm}
\pgfmoveto{\pgfxy(0.00,0.00)}\pgflineto{\pgfxy(110.00,0.00)}\pgfstroke
\pgfcircle[fill]{\pgfxy(10.00,0.00)}{1.00mm}
\pgfcircle[stroke]{\pgfxy(10.00,0.00)}{1.00mm}
\pgfcircle[fill]{\pgfxy(20.00,0.00)}{1.00mm}
\pgfcircle[stroke]{\pgfxy(20.00,0.00)}{1.00mm}
\pgfcircle[fill]{\pgfxy(30.00,0.00)}{1.00mm}
\pgfcircle[stroke]{\pgfxy(30.00,0.00)}{1.00mm}
\pgfcircle[fill]{\pgfxy(40.00,0.00)}{1.00mm}
\pgfcircle[stroke]{\pgfxy(40.00,0.00)}{1.00mm}
\pgfcircle[fill]{\pgfxy(50.00,0.00)}{1.00mm}
\pgfcircle[stroke]{\pgfxy(50.00,0.00)}{1.00mm}
\pgfcircle[fill]{\pgfxy(60.00,0.00)}{1.00mm}
\pgfcircle[stroke]{\pgfxy(60.00,0.00)}{1.00mm}
\pgfcircle[fill]{\pgfxy(70.00,0.00)}{1.00mm}
\pgfcircle[stroke]{\pgfxy(70.00,0.00)}{1.00mm}
\pgfcircle[fill]{\pgfxy(80.00,0.00)}{1.00mm}
\pgfcircle[stroke]{\pgfxy(80.00,0.00)}{1.00mm}
\pgfcircle[fill]{\pgfxy(90.00,0.00)}{1.00mm}
\pgfcircle[stroke]{\pgfxy(90.00,0.00)}{1.00mm}
\pgfcircle[fill]{\pgfxy(100.00,0.00)}{1.00mm}
\pgfcircle[stroke]{\pgfxy(100.00,0.00)}{1.00mm}
\pgfputat{\pgfxy(10.00,-5.00)}{\pgfbox[bottom,left]{\fontsize{11.38}{13.66}\selectfont \makebox[0pt]{1}}}
\pgfputat{\pgfxy(20.00,-5.00)}{\pgfbox[bottom,left]{\fontsize{11.38}{13.66}\selectfont \makebox[0pt]{2}}}
\pgfputat{\pgfxy(30.00,-5.00)}{\pgfbox[bottom,left]{\fontsize{11.38}{13.66}\selectfont \makebox[0pt]{3}}}
\pgfputat{\pgfxy(40.00,-5.00)}{\pgfbox[bottom,left]{\fontsize{11.38}{13.66}\selectfont \makebox[0pt]{4}}}
\pgfputat{\pgfxy(50.00,-5.00)}{\pgfbox[bottom,left]{\fontsize{11.38}{13.66}\selectfont \makebox[0pt]{5}}}
\pgfputat{\pgfxy(60.00,-5.00)}{\pgfbox[bottom,left]{\fontsize{11.38}{13.66}\selectfont \makebox[0pt]{6}}}
\pgfputat{\pgfxy(70.00,-5.00)}{\pgfbox[bottom,left]{\fontsize{11.38}{13.66}\selectfont \makebox[0pt]{7}}}
\pgfputat{\pgfxy(80.00,-5.00)}{\pgfbox[bottom,left]{\fontsize{11.38}{13.66}\selectfont \makebox[0pt]{8}}}
\pgfputat{\pgfxy(90.00,-5.00)}{\pgfbox[bottom,left]{\fontsize{11.38}{13.66}\selectfont \makebox[0pt]{9}}}
\pgfputat{\pgfxy(100.00,-5.00)}{\pgfbox[bottom,left]{\fontsize{11.38}{13.66}\selectfont \makebox[0pt]{10}}}
\pgfmoveto{\pgfxy(10.00,0.00)}\pgfcurveto{\pgfxy(21.08,9.67)}{\pgfxy(35.29,15.00)}{\pgfxy(50.00,15.00)}\pgfcurveto{\pgfxy(64.71,15.00)}{\pgfxy(78.92,9.67)}{\pgfxy(90.00,0.00)}\pgfstroke
\pgfmoveto{\pgfxy(20.00,0.00)}\pgfcurveto{\pgfxy(20.00,2.76)}{\pgfxy(22.24,5.00)}{\pgfxy(25.00,5.00)}\pgfcurveto{\pgfxy(27.76,5.00)}{\pgfxy(30.00,2.76)}{\pgfxy(30.00,0.00)}\pgfstroke
\pgfmoveto{\pgfxy(30.00,0.00)}\pgfcurveto{\pgfxy(34.72,6.30)}{\pgfxy(42.13,10.00)}{\pgfxy(50.00,10.00)}\pgfcurveto{\pgfxy(57.87,10.00)}{\pgfxy(65.28,6.30)}{\pgfxy(70.00,0.00)}\pgfstroke
\pgfmoveto{\pgfxy(50.00,0.00)}\pgfcurveto{\pgfxy(50.00,2.76)}{\pgfxy(52.24,5.00)}{\pgfxy(55.00,5.00)}\pgfcurveto{\pgfxy(57.76,5.00)}{\pgfxy(60.00,2.76)}{\pgfxy(60.00,0.00)}\pgfstroke
\pgfmoveto{\pgfxy(60.00,0.00)}\pgfcurveto{\pgfxy(64.72,6.30)}{\pgfxy(72.13,10.00)}{\pgfxy(80.00,10.00)}\pgfcurveto{\pgfxy(87.87,10.00)}{\pgfxy(95.28,6.30)}{\pgfxy(100.00,0.00)}\pgfstroke
\pgfmoveto{\pgfxy(70.00,0.00)}\pgfcurveto{\pgfxy(68.95,-4.68)}{\pgfxy(64.79,-8.00)}{\pgfxy(60.00,-8.00)}\pgfcurveto{\pgfxy(55.21,-8.00)}{\pgfxy(51.05,-4.68)}{\pgfxy(50.00,0.00)}\pgfstroke
\pgfmoveto{\pgfxy(10.00,0.00)}\pgfcurveto{\pgfxy(21.63,-8.45)}{\pgfxy(35.63,-13.00)}{\pgfxy(50.00,-13.00)}\pgfcurveto{\pgfxy(64.37,-13.00)}{\pgfxy(78.37,-8.45)}{\pgfxy(90.00,0.00)}\pgfstroke
\pgfmoveto{\pgfxy(100.00,0.00)}\pgfcurveto{\pgfxy(88.37,-8.45)}{\pgfxy(74.37,-13.00)}{\pgfxy(60.00,-13.00)}\pgfcurveto{\pgfxy(45.63,-13.00)}{\pgfxy(31.63,-8.45)}{\pgfxy(20.00,0.00)}\pgfstroke
\pgfmoveto{\pgfxy(40.00,0.00)}\pgfcurveto{\pgfxy(38.69,0.37)}{\pgfxy(37.84,1.65)}{\pgfxy(38.00,3.00)}\pgfcurveto{\pgfxy(38.13,4.09)}{\pgfxy(38.94,5.00)}{\pgfxy(40.00,5.00)}\pgfcurveto{\pgfxy(41.06,5.00)}{\pgfxy(41.87,4.09)}{\pgfxy(42.00,3.00)}\pgfcurveto{\pgfxy(42.16,1.65)}{\pgfxy(41.31,0.37)}{\pgfxy(40.00,0.00)}\pgfstroke
\pgfmoveto{\pgfxy(80.00,0.00)}\pgfcurveto{\pgfxy(78.69,0.37)}{\pgfxy(77.84,1.65)}{\pgfxy(78.00,3.00)}\pgfcurveto{\pgfxy(78.13,4.09)}{\pgfxy(78.94,5.00)}{\pgfxy(80.00,5.00)}\pgfcurveto{\pgfxy(81.06,5.00)}{\pgfxy(81.87,4.09)}{\pgfxy(82.00,3.00)}\pgfcurveto{\pgfxy(82.16,1.65)}{\pgfxy(81.31,0.37)}{\pgfxy(80.00,0.00)}\pgfstroke
\end{pgfpicture}%
$$
It easy to see that $\cros\sigma=5$ and $\nest\sigma=10$. Indeed, 5 crossings are $(1,6)$, $(2,3)$, $(3,6)$, $(5,6)$, and $(10,9)$ and 10 nestings are $(1,2)$, $(1,3)$, $(1,4)$, $(1,5)$ $(1,8)$, $(3,4)$, $(3,5)$, $(6,8)$, $(9,7)$, and $(10,7)$.

\begin{defn}
For $\sigma \in \SS_n$, let $\sigma(0)=0$ and $\sigma(n+1)=n+1$.
The corresponding numbers of peaks, valleys, double ascents, and double descents of permutation $\sigma$ are  denoted by
$$
\peak \sigma, \quad \valley \sigma, \quad \da \sigma, \quad \dd \sigma.
$$
Moreover, a  \emph{double ascent}  $\sigma(i)$ of  $\sigma$ ($i \in [n]$)
 is said to be a \emph{foremaximum} if
$\sigma(i)$ is  a \emph{left-to-right maximum} of  $\sigma$, i.e., $\sigma(j)<\sigma(i)$ for all $1\leq i<j$.
Let $\fmax\sigma$ denote  the number of foremaxima of $\sigma$.
\end{defn}

For instance, $\fmax(4\,2\,1\,5\,7\,3\,6\,8)=2$.
Note that   $\peak\sigma=\valley\sigma$ for any $\sigma\in \SS_n$.
Recall the following result in \cite[Theorem 5]{SZ12}.
%%%%%%%%%%%%
\begin{lem}\label{lem:b}
There is a bijection $\Phi$ on $\SS_n$ such that  for all $\sigma\in \SS_n$ we have
%The two septuple statistics $(\nest , \cros, \defi, \cda , \cdd, \cvalley, \fix)$ and
%$(\RESS, \LES,\\ \des, \da - \fmax, \dd, \valley, \fmax)$
\begin{multline*}
(\nest , \cros, \defi, \cda , \cdd, \cvalley, \fix)\sigma\\
 \qquad=(\RESS, \LES, \des, \da - \fmax, \dd, \valley, \fmax)\Phi(\sigma).
\end{multline*}
\end{lem}
%\begin{rmk}
%Let $\wex\sigma=n-\defi\sigma$ be the number of weak excedances of $\sigma\in \SS_n$.
%\begin{itemize}
%\item Our $\Phi$ transforms
%$(\defi, \nest , \cros)$ to  $( \des, \RESS, \LES)$.
%\item Corteel gave a similar bijection $\Phi_{C}$ on $\SS_n$ which transforms
%$$( \ndes, \ptoht,\pthto)\quad\textrm{to}\quad (\wex,\nest,\cros).$$
%\item
%Steingr\'{i}msson-Williams gave another one   which transforms
% $$(\ndes,\ptoht,\pthto)\quad\textrm{to}\quad (n+1-\wex,\nest,\cros).$$
%\end{itemize}
%\end{rmk}
 Consider the  enumerative  polynomial
\begin{align}
B_n(p,q,t,u,v,w,y)=
\sum_{\sigma\in \SS_n} p^{\nest \sigma} q^{\cros \sigma} t^{\defi\sigma} u^{\cda\sigma} v^{\cdd\sigma} w^{\cvalley\sigma}y^{\fix\sigma}.\label{eq:defB}
%&= \sum_{\sigma\in \SS_n} p^{\ress \sigma} q^{\les \sigma} t^{\des\sigma} u^{\da\sigma - \fmax\sigma} v^{\dd\sigma}w^{\valley\sigma}y^{\fmax\sigma}.
 %\label{eq:q-Euler_D}
\end{align}
By \cite[(34)]{SZ12} the ordinary generating function of $B_{n}(p,q,t,u,v,w,y)$ has
the following continued fraction expansion:
\begin{equation}\label{eq:A_D}
 \sum_{n\ge 0} B_n(p,q,t,u,v,w,y)z^n=
\dfrac{1}{
1 - b_0 z - \dfrac{a_0 c_1 z^2}{
1 - b_1 z - \dfrac{a_1 c_2 z^2}{
\cdots}}},
\end{equation}
where $a_h = tw[h+1]_{p,q}$, $b_h=yp^h+(qu+tv)[h]_{p,q}$, and $c_h=[h]_{p,q}$.

Let $\SS_{n}^*$ be the subset of $\SS_n$ consisting of permutations
of which each double ascent is also a foremaxima, and
\begin{align}
\SS_{n,k,j}   &= \set{ \sigma\in \SS_{n} : {\cvalley\sigma}=k, \fix\sigma=j, \cda\sigma=0},\label{eq:Snkj}\\
\SS_{n,k,j}^* &= \set{ \sigma\in \SS_{n}^* : {\valley\sigma}=k, \da\sigma=j }.
\end{align}

The following  result  is derived by combining   Corollary~7 and Theorem~8
in \cite{SZ12}.
\begin{lem}\label{lem:B}
We have
\begin{align*}
B_n(p,q,t,u,v,w,y)
=  \sum_{j=0}^{n} y^{j}\sum_{k=0}^{\floor{(n-j)/2}} b_{n,k,j}(p,q) (tw)^k (qu+tv)^{n-j-2k},
\end{align*}
where the coefficient $b_{n,k,j}(p,q)$ is given by
\begin{align*}
b_{n,k,j}(p,q)=\sum_{\sigma\in \SS_{n,k,j}} p^{\nest \sigma} q^{\cros \sigma}=\sum_{\sigma\in \SS_{n,k,j}^*} p^{\231 \sigma} q^{\312 \sigma}.
\end{align*}
\end{lem}

%The following lemma is proved in \cite[Eq. (39)]{SZ12} using continued fractions of their  generating functions.
%\begin{lem}\label{lem:equiv}
%The triple statistics $(\ress, \les, \des)$ and $(\res, \les, \des)$ are equidistributed on $\SS_n$.
%\end{lem}

%\begin{rmk} Using a modified version of Foata-Strehl action Br{\"a}nd{\'e}n~\cite[Proposition~5.2]{Bra08}
%proved the symmetry of $\tilde{A}_{n}(p,q,t)$.
%{\bf I think that we should be able to prove this  without using continued fraction.}
%\end{rmk}

A permutation  $\sigma\in\SS_n$ is a  {\em coderangement} if $ \fmax\sigma=0$. Let
$$\D^*_n = \set{\sigma\in\SS_n: \fmax\sigma=0}.
$$
For example, we have
$
 \D^*_4 = \set{2143, 3142, 3241, 4123, 4132, 4213, 4231, 4312, 4321}.
 $
Thus,  $\SS_{n,k,0}$ is the subset of derangements in $\D_{n}$ with exactly $k$ cyclic valleys, and without
double excedance,  and $\SS^*_{n,k,0}$ is the subset of coderangements in $\D^*_{n}$ with exactly $k$ valleys and without double ascents.

\subsection{Proof of  Theorem~\ref{thm:inv-Eulerian}}
Let
$$
\tilde{A}_{n}(p,q,t)=\sum_{\pi\in \SS_{n}}p^{(13\mhyphen2)\pi}q^{(2\mhyphen31)\pi}t^{\des \pi}.
$$
Since the triple statistics $(\ress, \les, \des)$ and $(\res, \les, \des)$
are equidistributed on $\SS_n$, see \cite[Eq. (39)]{SZ12},
 Lemma~\ref{lem:b} infers that
\begin{align*}
\tilde{A}_n(p,q,t) =
\sum_{\sigma\in \SS_n} p^{\nest \sigma} q^{\cros \sigma} t^{\defi\sigma}.
\end{align*}
As $\inv\sigma^{-1}=\inv\sigma$ and $\drop\sigma=\exc\sigma^{-1}$ for any $\sigma\in \SS_n$, we have
$$
A_n(q,t):=\sum_{\sigma\in \SS_n}q^{\inv\sigma}t^{\exc \sigma}=\sum_{\sigma\in \SS_n}q^{\inv\sigma}t^{\drop\sigma}.
$$
Invoking the known result  (see \cite[Eq. (40)]{SZ10})
\begin{align}\label{eq:inv}
\inv=\drop+\cros+2\nest,
\end{align}
  we  have $\inv-\drop=\cros+2\nest$ and
$$
A_n(q,t/q)=\tilde{A}_n(q^2,q,t).
$$
The result follows then from Lemma~\ref{lem:a} with $\gamma_{n,k}(q) := a_{n,k}(q^2,q)$.\qed
%%%%%%%%%%%%%%%%%%%
\subsection{Proof of  Theorem~\ref{Der-A}}
Let
\begin{align}
D_n(q,t):=\sum_{\sigma\in \D_n}q^{\inv\sigma}t^{\exc \sigma}=\sum_{\sigma\in \D_n}q^{\inv\sigma}t^{\drop\sigma}.
\end{align}
By \eqref{eq:defB} we have $D_n(q,t)=B_n(q^2,q,qt,1,1,1,0)$. Hence,
 setting $y=0$, $p=q^2$ and $u=v=w=1$ in  Lemma~\ref{lem:B}
 and then replacing $t$ by $qt$
 we obtain
 $$
D_n(q,t)=\sum_{\sigma\in \D_{n}} q^{\inv\sigma}t^{\drop \sigma}=
\sum_{k=0}^{\lfloor n/2\rfloor}b_{n,k,0}(q^2,q)q^{n-k}t^k(1+t)^{n-2k},
$$
where
\begin{align*}
b_{n,k,0}(q^2,q)q^{n-k}=
\sum_{\sigma\in \SS_{n,k,0}} q^{2\nest \sigma+\cros \sigma+n-\cvalley\sigma}.
\end{align*}
Since $n-\cvalley \sigma =n-\cpeak \sigma=\drop\sigma$ for any $\sigma\in \D_n$ and
$$\SS_{n,k,0}=\{\sigma\in \D_n: \cvalley(\sigma)=k\}=\{\sigma\in \D_n: \exc(\sigma)=k\}=\DE_{n,k},
$$
 the result follows  from \eqref{eq:inv}.
\qed

%%%%%%%
\section{Proofs of Theorems~\ref{thm:DerangeR}, \ref{thm:fexc} and \ref{thm:analogue-as}}\label{sec:proof2}
%%%%%%%
%%%%%%%%%%%%%%%%%%%%%%%%%%%%%%%
\subsection{Preliminaries}\label{sec:cf}
Recall that an integer $i$ is called \emph{fixed point} (resp. \emph{A-excedance}, \emph{A-weak excedance})
of $\sigma\in \Z_r\wr \SS_n$
if $\sigma(i)=i$ (resp. $\sigma(i) >_c i$, $\sigma(i) \ge_c i $). Let
\begin{align*}
\fixa \sigma &= \#\set{i \in [n]: i = \sigma(i) },
\quad \exca \sigma = \#\set{i\in [n]: i<_c \sigma(i)},\\
%\fixb \sigma &= \#\set{i\in [n]: i = -\sigma_i },\\
\wexa \sigma &= \#\set{i\in [n]: i\le_c \sigma(i)} = \exca \sigma + \fixa \sigma.
\end{align*}
Also, we define two statistics about colors:
\begin{align*}
\wexc \sigma &= \#\set{i\in [n]: i\le \abs{\sigma(i)} \text{ and } \cnum(\sigma(i))>0 },\\
%\nwexa \sigma &= \#\set{i\in [n]:0<\sigma_i <i},\\
%\negba \sigma &= \#\set{i\in [n]:(\sigma_i<0) \wedge (i\le -\sigma_i)},\\
%\negbb \sigma &= \#\set{i\in [n]:(\sigma_i<0) \wedge (-\sigma_i<i)},\\
%\negb \sigma &= \#\set{i\in [n]:\sigma(i)<0}, %\\ %= \negba \sigma + \negbb \sigma,\\
\csum \sigma &= \sum_{i \in [n]} \cnum(\sigma(i)).
%\crosb \sigma &= \#\set{(i,j) \in [n] \times [n]:
%{
%( j < i \le \sigma(j) < \sigma(i) )
%\vee ( j > i >   \sigma(j) > \sigma(i) )
%\atop
%\vee (i \leq -\sigma(j) <\sigma(i))
%}
%\nestb \sigma &= \#\set{(i,j) \in [n] \times [n]: (j<i\le \sigma_i<\sigma_j) \vee ( j>i>\sigma_i>\sigma_j) \vee (i \leq  \sigma_i< -\sigma_j)}.%,\\
\end{align*}
We generalize the crossings of permutations in $B_n$ (see \cite{CKJ13})  to colored permutations.
\begin{defn}[Crossings] For a permutation
$\sigma = \sigma(1) \dots \sigma(n) =
 {\pi_1 \dots \pi_n \choose \z_1 \dots \z_n} \in \Z_r \wr \SS_n$ a crossing is a pair $(i,j)$ such that
\begin{itemize}
\item $\z_i=\z_j=0$ and $i < j \le \pi_i < \pi_j $; or
\hspace{\fill}
\adjustbox{valign=t}{$
\begin{pgfpicture}{-2.00mm}{-6.10mm}{51.00mm}{5.50mm}
\pgfsetxvec{\pgfpoint{0.70mm}{0mm}}
\pgfsetyvec{\pgfpoint{0mm}{0.70mm}}
\color[rgb]{0,0,0}\pgfsetlinewidth{0.30mm}\pgfsetdash{}{0mm}
\pgfmoveto{\pgfxy(0.00,0.00)}\pgflineto{\pgfxy(70.00,0.00)}\pgfstroke
\pgfcircle[fill]{\pgfxy(10.00,0.00)}{0.70mm}
\pgfcircle[stroke]{\pgfxy(10.00,0.00)}{0.70mm}
\pgfcircle[fill]{\pgfxy(30.00,0.00)}{0.70mm}
\pgfcircle[stroke]{\pgfxy(30.00,0.00)}{0.70mm}
\pgfcircle[fill]{\pgfxy(40.00,0.00)}{0.70mm}
\pgfcircle[stroke]{\pgfxy(40.00,0.00)}{0.70mm}
\pgfcircle[fill]{\pgfxy(60.00,0.00)}{0.70mm}
\pgfcircle[stroke]{\pgfxy(60.00,0.00)}{0.70mm}
\pgfputat{\pgfxy(10.00,-5.00)}{\pgfbox[bottom,left]{\fontsize{7.97}{9.56}\selectfont \makebox[0pt]{$i$}}}
\pgfputat{\pgfxy(30.00,-5.00)}{\pgfbox[bottom,left]{\fontsize{7.97}{9.56}\selectfont \makebox[0pt]{$j$}}}
\pgfputat{\pgfxy(40.00,-5.00)}{\pgfbox[bottom,left]{\fontsize{7.97}{9.56}\selectfont \makebox[0pt]{$\pi_i$}}}
\pgfputat{\pgfxy(60.00,-5.00)}{\pgfbox[bottom,left]{\fontsize{7.97}{9.56}\selectfont \makebox[0pt]{$\pi_j$}}}
\pgfmoveto{\pgfxy(10.00,0.00)}\pgfcurveto{\pgfxy(14.33,3.25)}{\pgfxy(19.59,5.00)}{\pgfxy(25.00,5.00)}\pgfcurveto{\pgfxy(30.41,5.00)}{\pgfxy(35.67,3.25)}{\pgfxy(40.00,0.00)}\pgfstroke
\pgfmoveto{\pgfxy(30.00,0.00)}\pgfcurveto{\pgfxy(34.33,3.25)}{\pgfxy(39.59,5.00)}{\pgfxy(45.00,5.00)}\pgfcurveto{\pgfxy(50.41,5.00)}{\pgfxy(55.67,3.25)}{\pgfxy(60.00,0.00)}\pgfstroke
\end{pgfpicture}%
$}
\item $\z_i=\z_j=0$ and $\pi_i < \pi_j < i < j$; or
\hspace{\fill}
\adjustbox{valign=t}{$
\begin{pgfpicture}{-2.00mm}{-5.50mm}{51.00mm}{6.30mm}
\pgfsetxvec{\pgfpoint{0.70mm}{0mm}}
\pgfsetyvec{\pgfpoint{0mm}{0.70mm}}
\color[rgb]{0,0,0}\pgfsetlinewidth{0.30mm}\pgfsetdash{}{0mm}
\pgfmoveto{\pgfxy(0.00,0.00)}\pgflineto{\pgfxy(70.00,0.00)}\pgfstroke
\pgfcircle[fill]{\pgfxy(10.00,0.00)}{0.70mm}
\pgfcircle[stroke]{\pgfxy(10.00,0.00)}{0.70mm}
\pgfcircle[fill]{\pgfxy(30.00,0.00)}{0.70mm}
\pgfcircle[stroke]{\pgfxy(30.00,0.00)}{0.70mm}
\pgfcircle[fill]{\pgfxy(40.00,0.00)}{0.70mm}
\pgfcircle[stroke]{\pgfxy(40.00,0.00)}{0.70mm}
\pgfcircle[fill]{\pgfxy(60.00,0.00)}{0.70mm}
\pgfcircle[stroke]{\pgfxy(60.00,0.00)}{0.70mm}
\pgfputat{\pgfxy(40.00,3.00)}{\pgfbox[bottom,left]{\fontsize{7.97}{9.56}\selectfont \makebox[0pt]{$i$}}}
\pgfputat{\pgfxy(60.00,3.00)}{\pgfbox[bottom,left]{\fontsize{7.97}{9.56}\selectfont \makebox[0pt]{$j$}}}
\pgfputat{\pgfxy(10.00,3.00)}{\pgfbox[bottom,left]{\fontsize{7.97}{9.56}\selectfont \makebox[0pt]{$\pi_i$}}}
\pgfputat{\pgfxy(30.00,3.00)}{\pgfbox[bottom,left]{\fontsize{7.97}{9.56}\selectfont \makebox[0pt]{$\pi_j$}}}
\pgfmoveto{\pgfxy(10.00,0.00)}\pgfcurveto{\pgfxy(14.33,-3.25)}{\pgfxy(19.59,-5.00)}{\pgfxy(25.00,-5.00)}\pgfcurveto{\pgfxy(30.41,-5.00)}{\pgfxy(35.67,-3.25)}{\pgfxy(40.00,0.00)}\pgfstroke
\pgfmoveto{\pgfxy(30.00,0.00)}\pgfcurveto{\pgfxy(34.33,-3.25)}{\pgfxy(39.59,-5.00)}{\pgfxy(45.00,-5.00)}\pgfcurveto{\pgfxy(50.41,-5.00)}{\pgfxy(55.67,-3.25)}{\pgfxy(60.00,0.00)}\pgfstroke
\end{pgfpicture}%
$}
\item $\z_i>0$, $\z_j=0$, and $j \leq \pi_i <\pi_j$; or
\hspace{\fill}
\adjustbox{valign=t}{$
\begin{pgfpicture}{-5.93mm}{-6.10mm}{51.00mm}{9.09mm}
\pgfsetxvec{\pgfpoint{0.70mm}{0mm}}
\pgfsetyvec{\pgfpoint{0mm}{0.70mm}}
\color[rgb]{0,0,0}\pgfsetlinewidth{0.30mm}\pgfsetdash{}{0mm}
\pgfmoveto{\pgfxy(0.00,0.00)}\pgflineto{\pgfxy(70.00,0.00)}\pgfstroke
\pgfcircle[fill]{\pgfxy(30.00,0.00)}{0.70mm}
\pgfcircle[stroke]{\pgfxy(30.00,0.00)}{0.70mm}
\pgfcircle[fill]{\pgfxy(40.00,0.00)}{0.70mm}
\pgfcircle[stroke]{\pgfxy(40.00,0.00)}{0.70mm}
\pgfcircle[fill]{\pgfxy(60.00,0.00)}{0.70mm}
\pgfcircle[stroke]{\pgfxy(60.00,0.00)}{0.70mm}
\pgfputat{\pgfxy(10.00,-5.00)}{\pgfbox[bottom,left]{\fontsize{7.97}{9.56}\selectfont \makebox[0pt]{$i$}}}
\pgfputat{\pgfxy(30.00,-5.00)}{\pgfbox[bottom,left]{\fontsize{7.97}{9.56}\selectfont \makebox[0pt]{$j$}}}
\pgfputat{\pgfxy(40.00,-5.00)}{\pgfbox[bottom,left]{\fontsize{7.97}{9.56}\selectfont \makebox[0pt]{$\pi_i$}}}
\pgfputat{\pgfxy(60.00,-5.00)}{\pgfbox[bottom,left]{\fontsize{7.97}{9.56}\selectfont \makebox[0pt]{$\pi_j$}}}
\pgfmoveto{\pgfxy(30.00,0.00)}\pgfcurveto{\pgfxy(34.33,3.25)}{\pgfxy(39.59,5.00)}{\pgfxy(45.00,5.00)}\pgfcurveto{\pgfxy(50.41,5.00)}{\pgfxy(55.67,3.25)}{\pgfxy(60.00,0.00)}\pgfstroke
\pgfmoveto{\pgfxy(40.00,0.00)}\pgfcurveto{\pgfxy(29.87,7.35)}{\pgfxy(17.48,10.89)}{\pgfxy(5.00,10.00)}\pgfcurveto{\pgfxy(0.94,9.71)}{\pgfxy(-3.35,8.62)}{\pgfxy(-5.00,5.00)}\pgfcurveto{\pgfxy(-6.84,0.97)}{\pgfxy(-4.39,-3.71)}{\pgfxy(0.00,-5.00)}\pgfcurveto{\pgfxy(4.13,-6.21)}{\pgfxy(8.49,-4.03)}{\pgfxy(10.00,0.00)}\pgfstroke
\pgfcircle[fill]{\pgfxy(10.00,0.00)}{0.70mm}
\pgfcircle[stroke]{\pgfxy(10.00,0.00)}{0.70mm}
\end{pgfpicture}%
$}
\item $\z_i>0$, $\z_j=0$, and $\pi_j < i < j$; or
\hspace{\fill}
\adjustbox{valign=t}{$
\begin{pgfpicture}{-5.86mm}{-9.06mm}{51.00mm}{6.30mm}
\pgfsetxvec{\pgfpoint{0.70mm}{0mm}}
\pgfsetyvec{\pgfpoint{0mm}{0.70mm}}
\color[rgb]{0,0,0}\pgfsetlinewidth{0.30mm}\pgfsetdash{}{0mm}
\pgfmoveto{\pgfxy(0.00,0.00)}\pgflineto{\pgfxy(70.00,0.00)}\pgfstroke
\pgfcircle[fill]{\pgfxy(30.00,0.00)}{0.70mm}
\pgfcircle[stroke]{\pgfxy(30.00,0.00)}{0.70mm}
\pgfcircle[fill]{\pgfxy(40.00,0.00)}{0.70mm}
\pgfcircle[stroke]{\pgfxy(40.00,0.00)}{0.70mm}
\pgfcircle[fill]{\pgfxy(60.00,0.00)}{0.70mm}
\pgfcircle[stroke]{\pgfxy(60.00,0.00)}{0.70mm}
\pgfputat{\pgfxy(40.00,3.00)}{\pgfbox[bottom,left]{\fontsize{7.97}{9.56}\selectfont \makebox[0pt]{$i$}}}
\pgfputat{\pgfxy(60.00,3.00)}{\pgfbox[bottom,left]{\fontsize{7.97}{9.56}\selectfont \makebox[0pt]{$j$}}}
\pgfputat{\pgfxy(10.00,3.00)}{\pgfbox[bottom,left]{\fontsize{7.97}{9.56}\selectfont \makebox[0pt]{$\pi_i$}}}
\pgfputat{\pgfxy(30.00,3.00)}{\pgfbox[bottom,left]{\fontsize{7.97}{9.56}\selectfont \makebox[0pt]{$\pi_j$}}}
\pgfmoveto{\pgfxy(30.00,0.00)}\pgfcurveto{\pgfxy(34.33,-3.25)}{\pgfxy(39.59,-5.00)}{\pgfxy(45.00,-5.00)}\pgfcurveto{\pgfxy(50.41,-5.00)}{\pgfxy(55.67,-3.25)}{\pgfxy(60.00,0.00)}\pgfstroke
\pgfmoveto{\pgfxy(40.00,0.00)}\pgfcurveto{\pgfxy(29.78,-7.16)}{\pgfxy(17.46,-10.68)}{\pgfxy(5.00,-10.00)}\pgfcurveto{\pgfxy(0.88,-9.78)}{\pgfxy(-3.48,-8.69)}{\pgfxy(-5.00,-5.00)}\pgfcurveto{\pgfxy(-6.98,-0.21)}{\pgfxy(-3.02,4.99)}{\pgfxy(2.50,5.00)}\pgfcurveto{\pgfxy(5.78,5.00)}{\pgfxy(8.74,3.03)}{\pgfxy(10.00,0.00)}\pgfstroke
\pgfcircle[fill]{\pgfxy(10.00,0.00)}{0.70mm}
\pgfcircle[stroke]{\pgfxy(10.00,0.00)}{0.70mm}
\end{pgfpicture}%
$}
\item $\z_i>0$, $\z_j>0$, $i < j$, and $\pi_j <\pi_i$.
\hspace{\fill}
\adjustbox{valign=t}{$
\begin{pgfpicture}{-9.01mm}{-7.67mm}{51.00mm}{10.05mm}
\pgfsetxvec{\pgfpoint{0.70mm}{0mm}}
\pgfsetyvec{\pgfpoint{0mm}{0.70mm}}
\color[rgb]{0,0,0}\pgfsetlinewidth{0.30mm}\pgfsetdash{}{0mm}
\pgfmoveto{\pgfxy(0.00,0.00)}\pgflineto{\pgfxy(70.00,0.00)}\pgfstroke
\pgfcircle[fill]{\pgfxy(10.00,0.00)}{0.70mm}
\pgfcircle[stroke]{\pgfxy(10.00,0.00)}{0.70mm}
\pgfcircle[fill]{\pgfxy(30.00,0.00)}{0.70mm}
\pgfcircle[stroke]{\pgfxy(30.00,0.00)}{0.70mm}
\pgfcircle[fill]{\pgfxy(40.00,0.00)}{0.70mm}
\pgfcircle[stroke]{\pgfxy(40.00,0.00)}{0.70mm}
\pgfcircle[fill]{\pgfxy(60.00,0.00)}{0.70mm}
\pgfcircle[stroke]{\pgfxy(60.00,0.00)}{0.70mm}
\pgfputat{\pgfxy(10.00,-5.00)}{\pgfbox[bottom,left]{\fontsize{7.97}{9.56}\selectfont \makebox[0pt]{$i$}}}
\pgfputat{\pgfxy(30.00,-5.00)}{\pgfbox[bottom,left]{\fontsize{7.97}{9.56}\selectfont \makebox[0pt]{$j$}}}
\pgfputat{\pgfxy(40.00,-5.00)}{\pgfbox[bottom,left]{\fontsize{7.97}{9.56}\selectfont \makebox[0pt]{$\pi_i$}}}
\pgfputat{\pgfxy(60.00,-5.00)}{\pgfbox[bottom,left]{\fontsize{7.97}{9.56}\selectfont \makebox[0pt]{$\pi_j$}}}
\pgfmoveto{\pgfxy(10.00,0.00)}\pgfcurveto{\pgfxy(8.36,-2.69)}{\pgfxy(5.62,-4.52)}{\pgfxy(2.50,-5.00)}\pgfcurveto{\pgfxy(-1.41,-5.60)}{\pgfxy(-5.21,-3.45)}{\pgfxy(-5.00,0.00)}\pgfcurveto{\pgfxy(-4.84,2.52)}{\pgfxy(-2.40,4.00)}{\pgfxy(0.00,5.00)}\pgfcurveto{\pgfxy(19.62,13.19)}{\pgfxy(42.00,11.32)}{\pgfxy(60.00,0.00)}\pgfstroke
\pgfmoveto{\pgfxy(30.00,0.00)}\pgfcurveto{\pgfxy(22.37,-4.50)}{\pgfxy(13.83,-7.23)}{\pgfxy(5.00,-8.00)}\pgfcurveto{\pgfxy(-2.12,-8.62)}{\pgfxy(-9.55,-6.26)}{\pgfxy(-10.00,0.00)}\pgfcurveto{\pgfxy(-10.37,5.15)}{\pgfxy(-5.34,8.53)}{\pgfxy(0.00,10.00)}\pgfcurveto{\pgfxy(14.16,13.90)}{\pgfxy(29.34,10.10)}{\pgfxy(40.00,0.00)}\pgfstroke
\end{pgfpicture}%
$}
\end{itemize}
Let $\cros \sigma$ be the number of crossings in $\sigma \in \Z_r \wr \SS_n$.
%$$
%\crosb \sigma = \#\set{(i,j) \in [n] \times [n]:
%{
%( j < i \le \sigma(j) < \sigma(i) )
%\vee ( j > i >   \sigma(j) > \sigma(i) )
%\atop
%\vee (i \leq -\sigma(j) <\sigma(i))
%}
%}.
%$$
\end{defn}
\begin{lem}\label{thm:cfb}
We have
\begin{align}\label{CF-B}
\sum_{n\ge 0} &\sum_{\sigma \in \Z_r \wr \SS_n} q^{\cros\sigma} t^{\wexa\sigma} w^{\wexc\sigma} x^{\fixa\sigma} y^{\csum \sigma} z^n=
\dfrac{1}{
1-b_0 z - \dfrac{a_0 c_1 z^2}{
1-b_1 z - \dfrac{a_1 c_2 z^2}{
%1-b_2 z - \dfrac{a_2 c_3 z^2}{
\cdots}}},
\end{align}
where the coefficients $a_h$, $b_h$ ($h\geq 0$) and $c_h$ ($h\geq 1$) are given by
\begin{align*}
a_h &= (t + w y[r-1]_{y} q^h) (1 + y [r-1]_{y} q^{h+1}),\\
b_h &= (1 + y[r-1]_{y} q^h) [h]_q + t( x + q[h]_q) + wy[r-1]_{y} q^h [h+1]_q,\\
c_h &= [h]_q^2.
\end{align*}
\end{lem}
\rmk Corteel et al \cite{CKJ13} have proved \eqref{CF-B} in the case of $r=2$ and $x=1$. Actually we shall prove a refined version of \eqref{CF-B}. See \eqref{strong:thm} in Section~\ref{sec:proof3}.

%%%%%%%%%%%%%%%%%%%%%%%%%%%%%%%%%%
%\subsection{A lemma}
\begin{lem} We have
\begin{align}\label{cfD}
\sum_{n\geq 0} D_n^{(r)}(t) z^n
=\cfrac{1}{1-b_0z-\cfrac{\lambda_1 z^2}{1-b_1z-\cfrac{\lambda_2 z^2}{\cdots}}},
\end{align}
where $b_h=t[r-1]_t + h (1+t) [r]_t$ ($h\geq 0$) and $\lambda_h=t [r]_t^2 h^2$ ($h\geq 1$)
 and
\begin{align}
\sum_{n\geq 0}d_{n}^{(r)}(t) z^{n}
=\cfrac{1}{1-b'_0z-\cfrac{\lambda'_1 z^2}{1-b'_1z-\cfrac{\lambda'_2 z^2}{\cdots}}},
\label{cfd}
\end{align}
where $b'_h=(r-1)t + rh(1+t)$ ($h\geq 0$) and $\lambda'_h=(r h)^2 t $ ($h\geq 1$).
\end{lem}
\begin{proof}
We can derive \eqref{cfD} and \eqref{cfd} from Lemma~\ref{thm:cfb} as
\begin{align}
\fexc &= r(\wexa-\fixa)+\csum,\\
\exc  &= (\wexa-\fixa)+\wexc.
\end{align}
Letting $q=1$, $t \leftarrow t^r $, $w=1$, $x=0$ and $y \leftarrow t$ in \eqref{CF-B},
we obtain \eqref{cfD}.
Also, letting $q=1$, $t \leftarrow  t$, $w \leftarrow t$, $x=0$ and $y = 1$ in \eqref{CF-B}, we obtain \eqref{cfd}.
\end{proof}

\begin{rmk} It is also possible to derive \eqref{cfD} and \eqref{cfd} from the known generating functions.
\begin{enumerate}[(i)]
\item
Let
$B_{n}^{(r)}(t,x)=\sum_{\sigma\in \Z_r \wr \B_n}t^{\fexc \sigma}x^{\fix \sigma}$. Since
$D_n^{(r)}(t) = B_{n}^{(r)}(t,0)$,
 from the  known generating function formula
(see \cite[(2.23)]{Lin14} and \cite[(1.13)]{FH11} or \cite[(1.9)]{FH09} for $r=2$)
\begin{align}\label{DB}
\sum_{n\geq 0}D^{(r)}_{n}(t)\frac{z^{n}}{n!}=\frac{1-t}{e^{t^{r}z}-te^z},
\end{align}
we derive that
\begin{align}\label{gf:Br}
\sum_{n\geq 0}\frac{D^{(r)}_{n}(t)}{([r]_t)^n}\frac{z^{n}}{n!}
=e^{-\frac{t^r}{[r]_t}z}    \frac{1-t}{1-te^{(1-t)z}}.
\end{align}
Also, it is known (see \cite{CTZ09, Cho09}) that
\begin{align}
\sum_{n\geq 0}\frac{d_{n}^{(r)}(t)}{r^n}\frac{z^{n}}{n!}
=e^{-tz/r} \frac{1-t}{1-te^{(1-t)z}}.\label{eq:ctz09}
\end{align}
Applying the addition formula of Rogers-Stieltjes~(see \cite[Chap. 5]{GJ83} and  \cite{Zen93}) to
\eqref{gf:Br} and \eqref{eq:ctz09} we also obtain the
continued fraction expansions \eqref{cfD} and \eqref{cfd}.

\item
Note that $\frac{1-t}{1-te^{(1-t)z}}$ is the exponential generatying function of Eulerian polynomials for symmetric groups.

\item
In \eqref{CF-B} if we set $r=2$ and do the following substitution
$$
q\leftarrow 1,\quad w\leftarrow t,\quad x\leftarrow t^{-1},\text{ and } \quad y\leftarrow 1,
$$
then we obtain $a_h= 4t$, $b_h=(2h+1)(t+1)$ and $c_h=h^2$. These are the same coefficients obtained when we specialize \eqref{eq:A_D} as follows:
$$
\text{$p=q=t=u=v=y=1$, $w \leftarrow \frac{4t}{(1+t)^2}$, and $z \leftarrow (1+t)z$}.
$$
Combining Lemmas~\ref{lem:B} and \ref{thm:cfb}, we have
\begin{align}\label{eq:pet}
\sum_{\sigma\in \B_n }t^{\exc \sigma}
&= \sum_{k=0}^{\lfloor n/2 \rfloor} \left[ \sum_{j=0}^{n-2k} 2^{(n-2k-j)} b_{n,k,j}(1,1)  \right] (4t)^k (1+t)^{n-2k},
\end{align}
where $n\ge 2k+j$ and $b_{n,k,j}(1,1)$ is the cardinality of the set $\SS_{n,k,j}$, which is defined by \eqref{eq:Snkj}.
Note that this result \label{eq:pet} is comparable with the $\gamma$-positivity result of Chow~\cite[Theorem~4.7]{Cho08}.
%, see also \cite[Corollary 5.6 (b)]{DPS09}.

\end{enumerate}
\end{rmk}

%%%%%%%%%%%%%%%%%%%%%%%
\subsection{Proof of  Theorem~\ref{thm:DerangeR}}
Let
\begin{align}\label{eq:Bn}
B_n(1,1,1,0,1,w,y) = \sum_{\sigma \in \SS_n \atop \cda(\sigma)=0}  w^{\exc \sigma}y^{\fix\sigma}=\sum_{i,j}\gamma_{n,i,j}w^j y^i,
\end{align}
where $\gamma_{n,i,j}=|\{\sigma\in \SS_n: \cda\sigma=0, \; \fix\sigma=i,\; \exc\sigma=j\}|$. Now, specializing  the formula \eqref{eq:A_D}  yields
\begin{align}\label{cf4}
\sum_{n\ge 0} B_n(1,1,1,0,1,w,y) z^n
=
\cfrac{1}{1 - b_0 z -
\cfrac{a_0 c_1 z^2}{1 - b_1 z -
\cfrac{a_1 c_2 z^2}{\cdots}}},
\end{align}
where $a_h = w(h+1)$ ($h\ge0$), $b_h=y+h$ ($h\ge0$), and $c_h=h$ ($h\ge1$).
Comparing the right-hand sides of \eqref{cfD} and \eqref{cf4}, we derive
\begin{align}\label{eqDB}
D_{n}^{(r)}(t)=(1+t)^{n}([r]_t)^{n} B_n(1,1,1,0,1,w,y)
\end{align}
with
$$
w= \frac{t}{(1+t)^2}\quad{and}\quad y = \frac{t[r-1]_t}{(1+t)[r]_t}.
$$
Since $\cvalley(\sigma) = \exc(\sigma)$ for any $\sigma\in \SS_n$,
 it follows from \eqref{eq:Bn} and \eqref{eqDB} that
\begin{align}
D_{n}^{(r)}(t) = \sum_{1\leq i+2j\leq n}
\gamma_{n,i,j} \;t^{i+j} (1+t)^{n-i-2j} ([r-1]_t)^i([r]_t)^{n-i}.
\label{eq:interpret2}
\end{align}
In the same vein,
comparing the right-hand sides of \eqref{cfd} and \eqref{cf4}, we derive
\begin{align}\label{eqdB}
d_{n}^{(r)}(t)=(1+t)^{n} r^{n} B_n(1,1,1,0,1,w,y)
\end{align}
with
$$
w= \frac{t}{r(1+t)^2}\quad{and}\quad y = \frac{t(r-1)}{r(1+t)}.
$$
Therefore,  by \eqref{eq:Bn} and \eqref{eqdB}  we have
\begin{align}\label{eq:interpret2d}
d_{n}^{(r)}(t) = \sum_{1\leq i+2j\leq n}
\gamma_{n,i,j} \; t^{i+j} (1+t)^{n-i-2j} (r-1)^i r^{n-i}.
\end{align}
This completes the proof of Theorem~\ref{thm:DerangeR}.
\qed

\subsection{Proof of Theorem~\ref{thm:fexc}}
Since  any primitive $r$th root $\omega = e^{2 \pi \imath / r}$ of unity  satisfies
$$\sum_{i=0}^{r-1} (\omega^k)^i =
\begin{cases}
0 &\text{if $r \nmid k$,}\\
r &\text{if $r \mid k$,}
\end{cases}
$$
  for integer $j=0, \ldots, r-1$, we have
\begin{align}
\frac{1}{r} \sum_{i=0}^{r-1} \frac{\omega^{ij}}{X-\omega^i Y} =
\begin{cases}
\displaystyle \frac{X^{r-1}}{X^r-Y^r} &\text{if $j=0$,}\\
\displaystyle \frac{X^{j-1}Y^{r-j}}{X^r-Y^r} &\text{if $1\le j < r$.}\\
\end{cases}
\end{align}
%Thus, if $f(t) = \sum_{k\ge 0} f_k t^k$, for all $0 \le j < r$
%\begin{align}
%\sum_{i=0}^{r-1} \omega^{ij} f(\omega^i t)
%&= \sum_{k\geq 0} f_k \left( \sum_{i=0}^{r-1} w^{ij} (w^i t)^{k} \right) \notag\\
%&= \sum_{k\geq 0} \left( \sum_{i=0}^{r-1} (w^{k+j})^i \right) f_k t^{k}
%= r \left( \sum_{k\geq \lceil j/r \rceil} f_{rk-j} t^{rk-j} \right). \label{eq:sum}
%\end{align}
%Recall that $D_{n}^{(r)}(t) = \sum_{k\ge 0} D_{n,k}^{(r)} t^k$ and
It follows that for all $0 \le j < r$,
\begin{align}
\sum_{i=0}^{r-1} \omega^{ij} D_{n}^{(r)}(\omega^i t)
&= \sum_{k\geq 0} D_{n,k}^{(r)} \left( \sum_{i=0}^{r-1} w^{ij} (w^i t)^{k} \right) \notag\\
&= \sum_{k\geq 0} D_{n,k}^{(r)} \left( \sum_{i=0}^{r-1} (w^{k+j})^i \right) t^{k}
= r \, \sum_{k\geq \lceil j/r \rceil} D_{n,rk-j}^{(r)} t^{rk-j}. \label{eq:sum}
\end{align}
%In view of \eqref{DB}, letting $x=0$ in \eqref{gf:Br} yields
%\begin{align}\label{eq:fh09}
%\sum_{n\geq 0}D_{n}^{(r)}(t)\frac{x^{n}}{n!}
%=\frac{1-t}{e^{t^{r}x}-te^x}.
%\end{align}
From \eqref{eq:sum} and \eqref{DB}, for all $0 \le j < r$, we derive that
\begin{align}
%\sum_{n\geq 0}
%\sum_{k\geq 0} D_{n,rk}^{(r)} t^{rk}
%\frac{z^{n}}{n!}
%&=\sum_{n\geq 0}
%\frac{1}{r}\left[D_{n}^{(r)}(\omega^0 t)+\dots + D_{n}^{(r)}(\omega^{r-1}t)\right]
%\frac{z^{n}}{n!}\notag \\
%&=\frac{e^{(r-1)t^r z}-t^re^{(r-1)z}}{e^{rt^r z}-t^re^{rz}},\label{eq:even}\\
\sum_{n\geq 0}
\sum_{k\geq \lceil j/r \rceil} D_{n,rk-j}^{(r)} t^{rk-j}
\frac{x^{n}}{n!}
&=
\sum_{n\geq 0}
\frac{1}{r} \left(\sum_{i=0}^{r-1} \omega^{ij}D_{n}^{(r)}(\omega^{i} t)\right)
\frac{x^{n}}{n!}\nonumber\\
&=
\frac{1}{r} \sum_{i=0}^{r-1} \frac{\omega^{ij}- \omega^{i(j+1)} t}{e^{t^r x} - \omega^i t e^{x}}
\notag \\
&=
\begin{cases}
\displaystyle \frac{e^{(r-1)t^r x}-t^re^{(r-1)x}}{e^{rt^r x}-t^re^{rx}} &\text{if $j=0$,}\\%\label{eq:even}\\
\displaystyle \frac{ t^{r-j} e^{(j-1) t^r x + (r-j)x} ( 1- e^{(t^r-1) x}) }{e^{rt^r x}-t^re^{rx}} &\text{if $1\le j < r$.}
\end{cases}
\label{eq:odd}
\end{align}
%where $D_{n,k}^{(r)}$ is the coefficient of $t^k$ in $D_{n}^{(r)}(t)$.
Summing  \eqref{eq:odd} multiplied by $t^j$ over all $0 \le j < r$ yields
\begin{align}
\sum_{n\geq 0}
\sum_{k\geq 0}
D_{n,k}^{(r)} t^{ r \lceil k/r \rceil}
\frac{x^{n}}{n!}
&=\frac{(1-t^r)e^{(r-1) t^r x}}{e^{rt^r x}-t^re^{rx}}. \label{eq:dn}
\end{align}
Comparing  \eqref{eq:dn}  with \eqref{eq:ctz09}  we see that
\begin{align}
d_{n}^{(r)}(t) = \sum_{\sigma\in \D_n^{(r)}}t^{\lceil \fexc(\sigma)/r \rceil}.
\label{eq:eqfexc}
\end{align}
Since $\lceil \fexc(\sigma)/r \rceil=k$ if and only if $\fexc(\sigma)=rk-j$ for $j=0, \ldots, r-1$,
we obtain \eqref{eq:fexc}.
\qed
%%%%%%%%%

 %%%%%%%%%
\subsection{Proof of Theorem~\ref{thm:analogue-as}}
By Theorem~\ref{thm:DerangeB}, there are  $\gamma_{n,i}^{(2)}$ permutations
in $\SS_n$ with  $i$ weak excedances with $0\leq i\leq k$, for any such permutation, there are
${n-i \choose k-i}$ ways to choose  $k-i$ drops among
$n-i$ drops and $2^{k-i}$ ways to  color the chosen drops by either \emph{tilde} or \emph{bar}. Summing over $i$ we see that the number of permutations in $\Z_n$
 whose total number of weak excedances and colored drops is $k$ is given by
 $\sum_{i=0}^k \gamma_{n,i}^{(2)} 2^{k-i} {n-i \choose k-i}$, which is the formula\eqref{eq:bb}.
Now, writing  $(1+t)^{2n-2k}= (1+2t+t^2)^{n-k}=\sum_{j\geq 0}{n-k\choose j} (2t)^j(1+t^2)^{n-k-j}$ and
 substituting this into
\eqref{derangeB} we get
\begin{align*}
D_n^{(2)}(t)&=\sum_{k=1}^{n}\sum_{j\geq 0} \gamma_{n,k}^{(2)}\, 2^j{n-k\choose j} t^{k+j}(1+t^2)^{n-k-j}\\
&=\sum_{k=1}^{n}\sum_{i\geq 0} \gamma_{n,i}^{(2)}\, 2^{k-i}{n-k\choose k-i} t^{k}(1+t^2)^{n-k},
\end{align*}
which is \eqref{eq:SZ} in view of \eqref{eq:bb}.
Finally, we derive  \eqref{eq:AS} from
Theorem~\ref{thm:fexc}.  \qed

\section{A continued fraction expansion}\label{sec:proof3}

%%%%%%%%%%%%%%%%%%%%%
%\subsection{Preliminaries on combinatorics of continued fractions}\label{subsec:example}
\subsection{$r$-colored Laguerre histories}\label{subsec:example}
A \emph{Motzkin path} of length $n$ is a sequence of points
$\gamma=(\gamma_0, \gamma_1, \ldots, \gamma_n)$  in
the plan $\N\times \N$ with $\gamma_i=(x_i, y_i)$ ($0\leq i\leq n$) such that
%\begin{itemize}
%\item[(i)]
$\gamma_0=(0,0)$ and $\gamma_n=(n, 0)$;
and each step $(\gamma_{i-1}, \gamma_i)$ stisfies
$\gamma_i-\gamma_{i-1}=(1, 1), (1,0)$ or $(1, -1)$, and is called East, North-East, and South-East, respectively.
%\end{itemize}
The \emph{height} $h_i$ of the $i$th step $(\gamma_{i-1}, \gamma_i)$ is the ordinate of $\gamma_{i-1}$.
We can depict  a Motzkin path $\gamma$ by drawing a line connecting  $\gamma_i$
and $\gamma_{i+1}$ as follows:
$$
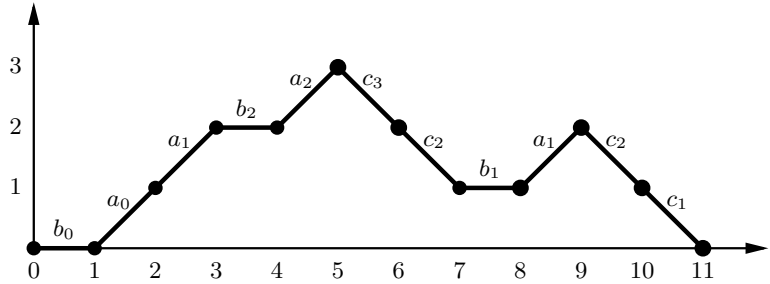
\centering
\begin{pgfpicture}{-5.20mm}{-6.69mm}{98.00mm}{34.00mm}
\pgfsetxvec{\pgfpoint{0.80mm}{0mm}}
\pgfsetyvec{\pgfpoint{0mm}{0.80mm}}
\color[rgb]{0,0,0}\pgfsetlinewidth{0.30mm}\pgfsetdash{}{0mm}
\pgfmoveto{\pgfxy(0.00,0.00)}\pgflineto{\pgfxy(0.00,40.00)}\pgfstroke
\pgfmoveto{\pgfxy(0.00,40.00)}\pgflineto{\pgfxy(-0.70,37.20)}\pgflineto{\pgfxy(0.70,37.20)}\pgflineto{\pgfxy(0.00,40.00)}\pgfclosepath\pgffill
\pgfmoveto{\pgfxy(0.00,40.00)}\pgflineto{\pgfxy(-0.70,37.20)}\pgflineto{\pgfxy(0.70,37.20)}\pgflineto{\pgfxy(0.00,40.00)}\pgfclosepath\pgfstroke
\pgfmoveto{\pgfxy(0.00,0.00)}\pgflineto{\pgfxy(120.00,0.00)}\pgfstroke
\pgfmoveto{\pgfxy(120.00,0.00)}\pgflineto{\pgfxy(117.20,0.70)}\pgflineto{\pgfxy(117.20,-0.70)}\pgflineto{\pgfxy(120.00,0.00)}\pgfclosepath\pgffill
\pgfmoveto{\pgfxy(120.00,0.00)}\pgflineto{\pgfxy(117.20,0.70)}\pgflineto{\pgfxy(117.20,-0.70)}\pgflineto{\pgfxy(120.00,0.00)}\pgfclosepath\pgfstroke
\pgfcircle[fill]{\pgfxy(10.00,0.00)}{0.80mm}
\pgfcircle[stroke]{\pgfxy(10.00,0.00)}{0.80mm}
\pgfcircle[fill]{\pgfxy(0.00,0.00)}{0.80mm}
\pgfcircle[stroke]{\pgfxy(0.00,0.00)}{0.80mm}
\pgfcircle[fill]{\pgfxy(20.00,10.00)}{0.80mm}
\pgfcircle[stroke]{\pgfxy(20.00,10.00)}{0.80mm}
\pgfcircle[fill]{\pgfxy(30.00,20.00)}{0.80mm}
\pgfcircle[stroke]{\pgfxy(30.00,20.00)}{0.80mm}
\pgfcircle[fill]{\pgfxy(40.00,20.00)}{0.80mm}
\pgfcircle[stroke]{\pgfxy(40.00,20.00)}{0.80mm}
\pgfcircle[fill]{\pgfxy(70.00,10.00)}{0.80mm}
\pgfcircle[stroke]{\pgfxy(70.00,10.00)}{0.80mm}
\pgfputat{\pgfxy(10.00,-5.00)}{\pgfbox[bottom,left]{\fontsize{9.10}{10.93}\selectfont \makebox[0pt]{1}}}
\pgfputat{\pgfxy(20.00,-5.00)}{\pgfbox[bottom,left]{\fontsize{9.10}{10.93}\selectfont \makebox[0pt]{2}}}
\pgfputat{\pgfxy(30.00,-5.00)}{\pgfbox[bottom,left]{\fontsize{9.10}{10.93}\selectfont \makebox[0pt]{3}}}
\pgfputat{\pgfxy(40.00,-5.00)}{\pgfbox[bottom,left]{\fontsize{9.10}{10.93}\selectfont \makebox[0pt]{4}}}
\pgfputat{\pgfxy(50.00,-5.00)}{\pgfbox[bottom,left]{\fontsize{9.10}{10.93}\selectfont \makebox[0pt]{5}}}
\pgfputat{\pgfxy(60.00,-5.00)}{\pgfbox[bottom,left]{\fontsize{9.10}{10.93}\selectfont \makebox[0pt]{6}}}
\pgfputat{\pgfxy(70.00,-5.00)}{\pgfbox[bottom,left]{\fontsize{9.10}{10.93}\selectfont \makebox[0pt]{7}}}
\pgfputat{\pgfxy(80.00,-5.00)}{\pgfbox[bottom,left]{\fontsize{9.10}{10.93}\selectfont \makebox[0pt]{8}}}
\pgfputat{\pgfxy(90.00,-5.00)}{\pgfbox[bottom,left]{\fontsize{9.10}{10.93}\selectfont \makebox[0pt]{9}}}
\pgfputat{\pgfxy(0.00,-5.00)}{\pgfbox[bottom,left]{\fontsize{9.10}{10.93}\selectfont \makebox[0pt]{0}}}
\pgfputat{\pgfxy(-2.00,9.00)}{\pgfbox[bottom,left]{\fontsize{9.10}{10.93}\selectfont \makebox[0pt][r]{1}}}
\pgfputat{\pgfxy(-2.00,19.00)}{\pgfbox[bottom,left]{\fontsize{9.10}{10.93}\selectfont \makebox[0pt][r]{2}}}
\pgfsetlinewidth{0.60mm}\pgfmoveto{\pgfxy(10.00,0.00)}\pgflineto{\pgfxy(20.00,10.00)}\pgfstroke
\pgfmoveto{\pgfxy(0.00,0.00)}\pgflineto{\pgfxy(10.00,0.00)}\pgfstroke
\pgfmoveto{\pgfxy(20.00,10.00)}\pgflineto{\pgfxy(30.00,20.00)}\pgfstroke
\pgfmoveto{\pgfxy(40.00,20.00)}\pgflineto{\pgfxy(50.00,30.00)}\pgfstroke
\pgfmoveto{\pgfxy(30.00,20.00)}\pgflineto{\pgfxy(40.00,20.00)}\pgfstroke
\pgfcircle[fill]{\pgfxy(50.00,30.00)}{0.80mm}
\pgfcircle[stroke]{\pgfxy(50.00,30.00)}{0.80mm}
\pgfmoveto{\pgfxy(50.00,30.00)}\pgflineto{\pgfxy(60.00,20.00)}\pgfstroke
\pgfcircle[fill]{\pgfxy(60.00,20.00)}{0.80mm}
\pgfcircle[stroke]{\pgfxy(60.00,20.00)}{0.80mm}
\pgfmoveto{\pgfxy(60.00,20.00)}\pgflineto{\pgfxy(70.00,10.00)}\pgfstroke
\pgfmoveto{\pgfxy(80.00,10.00)}\pgflineto{\pgfxy(90.00,20.00)}\pgfstroke
\pgfmoveto{\pgfxy(70.00,10.00)}\pgflineto{\pgfxy(80.00,10.00)}\pgfstroke
\pgfmoveto{\pgfxy(90.00,20.00)}\pgflineto{\pgfxy(100.00,10.00)}\pgfstroke
\pgfmoveto{\pgfxy(100.00,10.00)}\pgflineto{\pgfxy(110.00,0.00)}\pgfstroke
\pgfcircle[fill]{\pgfxy(90.00,20.00)}{0.80mm}
\pgfcircle[stroke]{\pgfxy(90.00,20.00)}{0.80mm}
\pgfcircle[fill]{\pgfxy(100.00,10.00)}{0.80mm}
\pgfcircle[stroke]{\pgfxy(100.00,10.00)}{0.80mm}
\pgfcircle[fill]{\pgfxy(80.00,10.00)}{0.80mm}
\pgfcircle[stroke]{\pgfxy(80.00,10.00)}{0.80mm}
\pgfcircle[fill]{\pgfxy(110.00,0.00)}{0.80mm}
\pgfcircle[stroke]{\pgfxy(110.00,0.00)}{0.80mm}
\pgfputat{\pgfxy(110.00,-5.00)}{\pgfbox[bottom,left]{\fontsize{9.10}{10.93}\selectfont \makebox[0pt]{11}}}
\pgfputat{\pgfxy(100.00,-5.00)}{\pgfbox[bottom,left]{\fontsize{9.10}{10.93}\selectfont \makebox[0pt]{10}}}
\pgfputat{\pgfxy(-2.00,29.00)}{\pgfbox[bottom,left]{\fontsize{9.10}{10.93}\selectfont \makebox[0pt][r]{3}}}
\pgfputat{\pgfxy(5.00,2.00)}{\pgfbox[bottom,left]{\fontsize{9.10}{10.93}\selectfont \makebox[0pt]{$b_0$}}}
\pgfputat{\pgfxy(75.00,12.00)}{\pgfbox[bottom,left]{\fontsize{9.10}{10.93}\selectfont \makebox[0pt]{$b_1$}}}
\pgfputat{\pgfxy(35.00,22.00)}{\pgfbox[bottom,left]{\fontsize{9.10}{10.93}\selectfont \makebox[0pt]{$b_2$}}}
\pgfputat{\pgfxy(54.00,27.00)}{\pgfbox[bottom,left]{\fontsize{9.10}{10.93}\selectfont $c_3$}}
\pgfputat{\pgfxy(64.00,17.00)}{\pgfbox[bottom,left]{\fontsize{9.10}{10.93}\selectfont $c_2$}}
\pgfputat{\pgfxy(94.00,17.00)}{\pgfbox[bottom,left]{\fontsize{9.10}{10.93}\selectfont $c_2$}}
\pgfputat{\pgfxy(104.00,7.00)}{\pgfbox[bottom,left]{\fontsize{9.10}{10.93}\selectfont $c_1$}}
\pgfputat{\pgfxy(16.00,7.00)}{\pgfbox[bottom,left]{\fontsize{9.10}{10.93}\selectfont \makebox[0pt][r]{$a_0$}}}
\pgfputat{\pgfxy(26.00,17.00)}{\pgfbox[bottom,left]{\fontsize{9.10}{10.93}\selectfont \makebox[0pt][r]{$a_1$}}}
\pgfputat{\pgfxy(46.00,27.00)}{\pgfbox[bottom,left]{\fontsize{9.10}{10.93}\selectfont \makebox[0pt][r]{$a_2$}}}
\pgfputat{\pgfxy(86.00,17.00)}{\pgfbox[bottom,left]{\fontsize{9.10}{10.93}\selectfont \makebox[0pt][r]{$a_1$}}}
\end{pgfpicture}%
$$
%$$
%\includegraphics[scale=0.6]{PATH.pdf}
%$$
Given a Motzkin path $\gamma$, we weight each step $(x_{i+1}, y_{i+1})-(x_i, y_i)$
of height $i$ by $a_i$ (resp. $b_i$ and $c_i$)  and define the weight of $\gamma$
as the product of its step weights.  For example, the weight of the above
Motzkin path  is $\w(\gamma)=a_0{a_1}^2 a_2 b_0 b_1 b_2 c_1{c_2}^2 c_3$.
Denote by $\M_{n}$ the set of Motzkin paths of length $n\geq 1$.
It is folklore (see \cite{GJ83}) that
\begin{align}\label{eq:fla}
1+\sum_{n\geq 1}\sum_{\gamma\in \M_{n}}\w(\gamma) z^{n}=\dfrac{1}
{1-b_{0}z-\dfrac{a_0c_1z^2}
{1-b_{1}z-\dfrac{a_1c_2z^2}
{\cdots}}}.
\end{align}
%where $\lambda_h:=a_{h-1}c_h$.

Our starting point is the  following $r$-dilatation of Laguerre's continued fraction:
\begin{align}
\sum_{n=0}^{\infty} (n! r^n) z^{n}&=\dfrac{1}{
1 - b_0z - \cfrac{a_0c_1z^2}{
1 - b_1z - \cfrac{a_1c_2 z^2}{
\cdots}}},
\label{eq:r-euler}
\end{align}
where
$
a_h = r^2,  b_h = (2h+1)r ~ (h\geq 0), c_h = h^2 ~ (h\geq 1).
$
Indeed, this formula is  related to the moment sequence of the simple orthogonal Laguerre polynomials when $r=1$.
Clearly $r^nn!$ is the cardinality of $\Z_r\wr \SS_n$.
In order to make the counterpart of $\Z_r\wr \SS_n$ in terms of Motzkin paths, we need to
generalize the notion of \emph{Laguerre history} \cite{dMV94}.
\begin{defn}[$r$-colored Laguerre history]
An \emph{$r$-colored Laguerre history of length $n$} is a couple $h=(\gamma, \xi)$, where
$\gamma$ is a  Motzkin path of length $n$
and $\xi = (({p_1},{q_1}), \dots, ({p_n},{q_n}))$ is an integer-pair sequence (as labels of steps)
satisfying
\begin{enumerate}[(i)]
\item if the $i$th step is North-East, then  $(p_i,q_i)\in [-(r-1), 0]^2$;
\item if the $i$th step is East, then
$(p_i, q_i)\in [1, h_i]\times [-(r-1), 0]\cup  [-(r-1), 0]\times [1, h_i+1]$;
\item if the $i$th step is South-East, then
$(p_i,q_i)\in [1, h_i]^2$;
\end{enumerate}
where $h_i$ is the  height of the $i$th step  of $\gamma$.
\end{defn}
Counting the number of possible pairs $(p,q)$ for each step of height $h$,
the number of $r$-colored Laguerre histories associated to a Motzkin path $\gamma$ is equal to  the weight $\w(\gamma)$.
Hence, according to \eqref{eq:r-euler}, the number of $r$-colored Laguerre histories of length $n$ is $n! r^n$.

The following is an example of $3$-colored Laguerre history of length 11, where
 the label on the $i$th step stands for $({p_i},{q_i})$.
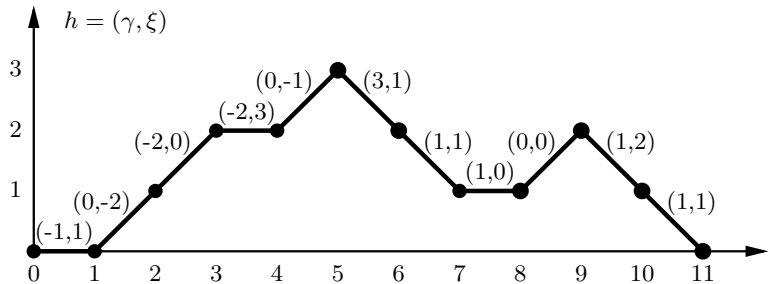
$$
\centering
\begin{pgfpicture}{-5.20mm}{-6.69mm}{98.00mm}{34.11mm}
\pgfsetxvec{\pgfpoint{0.80mm}{0mm}}
\pgfsetyvec{\pgfpoint{0mm}{0.80mm}}
\color[rgb]{0,0,0}\pgfsetlinewidth{0.30mm}\pgfsetdash{}{0mm}
\pgfmoveto{\pgfxy(0.00,0.00)}\pgflineto{\pgfxy(0.00,40.00)}\pgfstroke
\pgfmoveto{\pgfxy(0.00,40.00)}\pgflineto{\pgfxy(-0.70,37.20)}\pgflineto{\pgfxy(0.70,37.20)}\pgflineto{\pgfxy(0.00,40.00)}\pgfclosepath\pgffill
\pgfmoveto{\pgfxy(0.00,40.00)}\pgflineto{\pgfxy(-0.70,37.20)}\pgflineto{\pgfxy(0.70,37.20)}\pgflineto{\pgfxy(0.00,40.00)}\pgfclosepath\pgfstroke
\pgfmoveto{\pgfxy(0.00,0.00)}\pgflineto{\pgfxy(120.00,0.00)}\pgfstroke
\pgfmoveto{\pgfxy(120.00,0.00)}\pgflineto{\pgfxy(117.20,0.70)}\pgflineto{\pgfxy(117.20,-0.70)}\pgflineto{\pgfxy(120.00,0.00)}\pgfclosepath\pgffill
\pgfmoveto{\pgfxy(120.00,0.00)}\pgflineto{\pgfxy(117.20,0.70)}\pgflineto{\pgfxy(117.20,-0.70)}\pgflineto{\pgfxy(120.00,0.00)}\pgfclosepath\pgfstroke
\pgfcircle[fill]{\pgfxy(10.00,0.00)}{0.80mm}
\pgfcircle[stroke]{\pgfxy(10.00,0.00)}{0.80mm}
\pgfcircle[fill]{\pgfxy(0.00,0.00)}{0.80mm}
\pgfcircle[stroke]{\pgfxy(0.00,0.00)}{0.80mm}
\pgfcircle[fill]{\pgfxy(20.00,10.00)}{0.80mm}
\pgfcircle[stroke]{\pgfxy(20.00,10.00)}{0.80mm}
\pgfcircle[fill]{\pgfxy(30.00,20.00)}{0.80mm}
\pgfcircle[stroke]{\pgfxy(30.00,20.00)}{0.80mm}
\pgfcircle[fill]{\pgfxy(40.00,20.00)}{0.80mm}
\pgfcircle[stroke]{\pgfxy(40.00,20.00)}{0.80mm}
\pgfcircle[fill]{\pgfxy(70.00,10.00)}{0.80mm}
\pgfcircle[stroke]{\pgfxy(70.00,10.00)}{0.80mm}
\pgfputat{\pgfxy(10.00,-5.00)}{\pgfbox[bottom,left]{\fontsize{9.10}{10.93}\selectfont \makebox[0pt]{1}}}
\pgfputat{\pgfxy(20.00,-5.00)}{\pgfbox[bottom,left]{\fontsize{9.10}{10.93}\selectfont \makebox[0pt]{2}}}
\pgfputat{\pgfxy(30.00,-5.00)}{\pgfbox[bottom,left]{\fontsize{9.10}{10.93}\selectfont \makebox[0pt]{3}}}
\pgfputat{\pgfxy(40.00,-5.00)}{\pgfbox[bottom,left]{\fontsize{9.10}{10.93}\selectfont \makebox[0pt]{4}}}
\pgfputat{\pgfxy(50.00,-5.00)}{\pgfbox[bottom,left]{\fontsize{9.10}{10.93}\selectfont \makebox[0pt]{5}}}
\pgfputat{\pgfxy(60.00,-5.00)}{\pgfbox[bottom,left]{\fontsize{9.10}{10.93}\selectfont \makebox[0pt]{6}}}
\pgfputat{\pgfxy(70.00,-5.00)}{\pgfbox[bottom,left]{\fontsize{9.10}{10.93}\selectfont \makebox[0pt]{7}}}
\pgfputat{\pgfxy(80.00,-5.00)}{\pgfbox[bottom,left]{\fontsize{9.10}{10.93}\selectfont \makebox[0pt]{8}}}
\pgfputat{\pgfxy(90.00,-5.00)}{\pgfbox[bottom,left]{\fontsize{9.10}{10.93}\selectfont \makebox[0pt]{9}}}
\pgfputat{\pgfxy(0.00,-5.00)}{\pgfbox[bottom,left]{\fontsize{9.10}{10.93}\selectfont \makebox[0pt]{0}}}
\pgfputat{\pgfxy(-2.00,9.00)}{\pgfbox[bottom,left]{\fontsize{9.10}{10.93}\selectfont \makebox[0pt][r]{1}}}
\pgfputat{\pgfxy(-2.00,19.00)}{\pgfbox[bottom,left]{\fontsize{9.10}{10.93}\selectfont \makebox[0pt][r]{2}}}
\pgfsetlinewidth{0.60mm}\pgfmoveto{\pgfxy(10.00,0.00)}\pgflineto{\pgfxy(20.00,10.00)}\pgfstroke
\pgfmoveto{\pgfxy(0.00,0.00)}\pgflineto{\pgfxy(10.00,0.00)}\pgfstroke
\pgfmoveto{\pgfxy(20.00,10.00)}\pgflineto{\pgfxy(30.00,20.00)}\pgfstroke
\pgfmoveto{\pgfxy(40.00,20.00)}\pgflineto{\pgfxy(50.00,30.00)}\pgfstroke
\pgfmoveto{\pgfxy(30.00,20.00)}\pgflineto{\pgfxy(40.00,20.00)}\pgfstroke
\pgfcircle[fill]{\pgfxy(50.00,30.00)}{0.80mm}
\pgfcircle[stroke]{\pgfxy(50.00,30.00)}{0.80mm}
\pgfmoveto{\pgfxy(50.00,30.00)}\pgflineto{\pgfxy(60.00,20.00)}\pgfstroke
\pgfcircle[fill]{\pgfxy(60.00,20.00)}{0.80mm}
\pgfcircle[stroke]{\pgfxy(60.00,20.00)}{0.80mm}
\pgfmoveto{\pgfxy(60.00,20.00)}\pgflineto{\pgfxy(70.00,10.00)}\pgfstroke
\pgfmoveto{\pgfxy(80.00,10.00)}\pgflineto{\pgfxy(90.00,20.00)}\pgfstroke
\pgfmoveto{\pgfxy(70.00,10.00)}\pgflineto{\pgfxy(80.00,10.00)}\pgfstroke
\pgfmoveto{\pgfxy(90.00,20.00)}\pgflineto{\pgfxy(100.00,10.00)}\pgfstroke
\pgfmoveto{\pgfxy(100.00,10.00)}\pgflineto{\pgfxy(110.00,0.00)}\pgfstroke
\pgfcircle[fill]{\pgfxy(90.00,20.00)}{0.80mm}
\pgfcircle[stroke]{\pgfxy(90.00,20.00)}{0.80mm}
\pgfcircle[fill]{\pgfxy(100.00,10.00)}{0.80mm}
\pgfcircle[stroke]{\pgfxy(100.00,10.00)}{0.80mm}
\pgfcircle[fill]{\pgfxy(80.00,10.00)}{0.80mm}
\pgfcircle[stroke]{\pgfxy(80.00,10.00)}{0.80mm}
\pgfcircle[fill]{\pgfxy(110.00,0.00)}{0.80mm}
\pgfcircle[stroke]{\pgfxy(110.00,0.00)}{0.80mm}
\pgfputat{\pgfxy(110.00,-5.00)}{\pgfbox[bottom,left]{\fontsize{9.10}{10.93}\selectfont \makebox[0pt]{11}}}
\pgfputat{\pgfxy(100.00,-5.00)}{\pgfbox[bottom,left]{\fontsize{9.10}{10.93}\selectfont \makebox[0pt]{10}}}
\pgfputat{\pgfxy(-2.00,29.00)}{\pgfbox[bottom,left]{\fontsize{9.10}{10.93}\selectfont \makebox[0pt][r]{3}}}
\pgfputat{\pgfxy(5.00,37.00)}{\pgfbox[bottom,left]{\fontsize{9.10}{10.93}\selectfont $h = (\gamma, \xi)$}}
\pgfputat{\pgfxy(5.00,2.00)}{\pgfbox[bottom,left]{\fontsize{9.10}{10.93}\selectfont \makebox[0pt]{(-1,1)}}}
\pgfputat{\pgfxy(75.00,12.00)}{\pgfbox[bottom,left]{\fontsize{9.10}{10.93}\selectfont \makebox[0pt]{(1,0)}}}
\pgfputat{\pgfxy(35.00,22.00)}{\pgfbox[bottom,left]{\fontsize{9.10}{10.93}\selectfont \makebox[0pt]{(-2,3)}}}
\pgfputat{\pgfxy(54.00,27.00)}{\pgfbox[bottom,left]{\fontsize{9.10}{10.93}\selectfont (3,1)}}
\pgfputat{\pgfxy(64.00,17.00)}{\pgfbox[bottom,left]{\fontsize{9.10}{10.93}\selectfont (1,1)}}
\pgfputat{\pgfxy(94.00,17.00)}{\pgfbox[bottom,left]{\fontsize{9.10}{10.93}\selectfont (1,2)}}
\pgfputat{\pgfxy(104.00,7.00)}{\pgfbox[bottom,left]{\fontsize{9.10}{10.93}\selectfont (1,1)}}
\pgfputat{\pgfxy(16.00,7.00)}{\pgfbox[bottom,left]{\fontsize{9.10}{10.93}\selectfont \makebox[0pt][r]{(0,-2)}}}
\pgfputat{\pgfxy(26.00,17.00)}{\pgfbox[bottom,left]{\fontsize{9.10}{10.93}\selectfont \makebox[0pt][r]{(-2,0)}}}
\pgfputat{\pgfxy(46.00,27.00)}{\pgfbox[bottom,left]{\fontsize{9.10}{10.93}\selectfont \makebox[0pt][r]{(0,-1)}}}
\pgfputat{\pgfxy(86.00,17.00)}{\pgfbox[bottom,left]{\fontsize{9.10}{10.93}\selectfont \makebox[0pt][r]{(0,0)}}}
\end{pgfpicture}%
$$

When $r=1$, there are two well-known bijections between $\SS_n$ and the Laguerre histories of length $n$
due to  Fran\c con-Viennot and  Foata-Zeilberger,  which correspond
to the two   interpretations of Eulerian polynomials in $\SS_n$:
one using linear statistics and the other  using cyclic statistics.
The connection between these two bijections is explained in \cite{CSZ97} (see also \cite{dMV94, SZ12}).

The \emph{pignose diagram} representation for a permutation was
 introduced in \cite{dMV94, CKJ13}. This is a  useful device  to illustrate the crossings of a permutation by the intersecting  points of two arcs.
Actually, we can associate a pignose diagram to any $r$-colored permutation
$\sigma (= \pi^{\z}) ={\pi_1 \pi_2 \dots \pi_n \choose \z_1 \z_2 \dots \z_n} \in \Z_r\wr \SS_n$
 (see \S~\ref{subsec:wreath})  as follows:
A pair of two vertices positioned side by side enclosed by an ellipse is called a \emph{pignose}.
We arrange $n$ pignoses labeled with $1,2,\dots,n$ from left to right on an  horizontal line. For each $i\in [n]$, we connect the \emph{left} vertex of $i$th pignose and the \emph{right} vertex of $\pi_i$th pignose with an arc:
\begin{itemize}
\item if $i\le \pi_i$ and $\z_i = 0$, draw an arc above the horizontal line;
\item if $\pi_i <i$ and $\z_i = 0$, draw an arc below the horizontal line;
\item if $\z_i>0$, draw an arc starting from the left vertex of $i$th pignose below the horizontal line to the right vertex of $\pi_i$th pignose above the horizontal line  clockwise (like a spiral) and label this spiral arc with $\z_i$;
\item optionally, if $i\le \pi_i$, use a blue color in drawing arcs.
\end{itemize}
$$
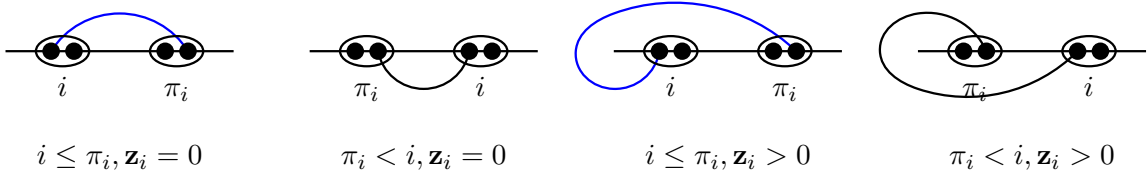
\centering
\begin{pgfpicture}{-2.00mm}{-7.86mm}{152.00mm}{18.43mm}
\pgfsetxvec{\pgfpoint{1.00mm}{0mm}}
\pgfsetyvec{\pgfpoint{0mm}{1.00mm}}
\color[rgb]{0,0,0}\pgfsetlinewidth{0.30mm}\pgfsetdash{}{0mm}
\color[rgb]{0,0,1}\pgfmoveto{\pgfxy(6.00,10.00)}\pgfcurveto{\pgfxy(7.93,13.11)}{\pgfxy(11.34,15.00)}{\pgfxy(15.00,15.00)}\pgfcurveto{\pgfxy(18.66,15.00)}{\pgfxy(22.07,13.11)}{\pgfxy(24.00,10.00)}\pgfstroke
\color[rgb]{0,0,0}\pgfmoveto{\pgfxy(61.00,10.00)}\pgfcurveto{\pgfxy(60.47,7.10)}{\pgfxy(57.94,5.00)}{\pgfxy(55.00,5.00)}\pgfcurveto{\pgfxy(52.06,5.00)}{\pgfxy(49.53,7.10)}{\pgfxy(49.00,10.00)}\pgfstroke
\color[rgb]{0,0,1}\pgfmoveto{\pgfxy(86.00,10.00)}\pgfcurveto{\pgfxy(85.54,7.06)}{\pgfxy(82.97,4.92)}{\pgfxy(80.00,5.00)}\pgfcurveto{\pgfxy(77.15,5.07)}{\pgfxy(74.78,7.27)}{\pgfxy(75.00,10.00)}\pgfcurveto{\pgfxy(75.21,12.52)}{\pgfxy(77.55,14.10)}{\pgfxy(80.00,15.00)}\pgfcurveto{\pgfxy(88.30,18.06)}{\pgfxy(97.61,16.12)}{\pgfxy(104.00,10.00)}\pgfstroke
\color[rgb]{0,0,0}\pgfmoveto{\pgfxy(129.00,10.00)}\pgfcurveto{\pgfxy(127.47,13.51)}{\pgfxy(123.78,15.56)}{\pgfxy(120.00,15.00)}\pgfcurveto{\pgfxy(117.25,14.59)}{\pgfxy(114.88,12.64)}{\pgfxy(115.00,10.00)}\pgfcurveto{\pgfxy(115.11,7.46)}{\pgfxy(117.49,5.84)}{\pgfxy(120.00,5.00)}\pgfcurveto{\pgfxy(127.38,2.52)}{\pgfxy(135.53,4.46)}{\pgfxy(141.00,10.00)}\pgfstroke
\pgfputat{\pgfxy(15.00,-5.00)}{\pgfbox[bottom,left]{\fontsize{11.38}{13.66}\selectfont \makebox[0pt]{$i \le \pi_{i}, \z_i=0$}}}
\pgfellipse[stroke]{\pgfxy(62.50,10.00)}{\pgfxy(3.50,0.00)}{\pgfxy(0.00,2.00)}
\pgfcircle[fill]{\pgfxy(61.00,10.00)}{1.00mm}
\pgfcircle[stroke]{\pgfxy(61.00,10.00)}{1.00mm}
\pgfcircle[fill]{\pgfxy(64.00,10.00)}{1.00mm}
\pgfcircle[stroke]{\pgfxy(64.00,10.00)}{1.00mm}
\pgfmoveto{\pgfxy(40.00,10.00)}\pgflineto{\pgfxy(70.00,10.00)}\pgfstroke
\pgfellipse[stroke]{\pgfxy(47.50,10.00)}{\pgfxy(3.50,0.00)}{\pgfxy(0.00,2.00)}
\pgfcircle[fill]{\pgfxy(46.00,10.00)}{1.00mm}
\pgfcircle[stroke]{\pgfxy(46.00,10.00)}{1.00mm}
\pgfcircle[fill]{\pgfxy(49.00,10.00)}{1.00mm}
\pgfcircle[stroke]{\pgfxy(49.00,10.00)}{1.00mm}
\pgfellipse[stroke]{\pgfxy(102.50,10.00)}{\pgfxy(3.50,0.00)}{\pgfxy(0.00,2.00)}
\pgfcircle[fill]{\pgfxy(101.00,10.00)}{1.00mm}
\pgfcircle[stroke]{\pgfxy(101.00,10.00)}{1.00mm}
\pgfcircle[fill]{\pgfxy(104.00,10.00)}{1.00mm}
\pgfcircle[stroke]{\pgfxy(104.00,10.00)}{1.00mm}
\pgfmoveto{\pgfxy(80.00,10.00)}\pgflineto{\pgfxy(110.00,10.00)}\pgfstroke
\pgfellipse[stroke]{\pgfxy(87.50,10.00)}{\pgfxy(3.50,0.00)}{\pgfxy(0.00,2.00)}
\pgfcircle[fill]{\pgfxy(86.00,10.00)}{1.00mm}
\pgfcircle[stroke]{\pgfxy(86.00,10.00)}{1.00mm}
\pgfcircle[fill]{\pgfxy(89.00,10.00)}{1.00mm}
\pgfcircle[stroke]{\pgfxy(89.00,10.00)}{1.00mm}
\pgfellipse[stroke]{\pgfxy(142.50,10.00)}{\pgfxy(3.50,0.00)}{\pgfxy(0.00,2.00)}
\pgfcircle[fill]{\pgfxy(141.00,10.00)}{1.00mm}
\pgfcircle[stroke]{\pgfxy(141.00,10.00)}{1.00mm}
\pgfcircle[fill]{\pgfxy(144.00,10.00)}{1.00mm}
\pgfcircle[stroke]{\pgfxy(144.00,10.00)}{1.00mm}
\pgfmoveto{\pgfxy(120.00,10.00)}\pgflineto{\pgfxy(150.00,10.00)}\pgfstroke
\pgfellipse[stroke]{\pgfxy(127.50,10.00)}{\pgfxy(3.50,0.00)}{\pgfxy(0.00,2.00)}
\pgfcircle[fill]{\pgfxy(126.00,10.00)}{1.00mm}
\pgfcircle[stroke]{\pgfxy(126.00,10.00)}{1.00mm}
\pgfcircle[fill]{\pgfxy(129.00,10.00)}{1.00mm}
\pgfcircle[stroke]{\pgfxy(129.00,10.00)}{1.00mm}
\pgfellipse[stroke]{\pgfxy(22.50,10.00)}{\pgfxy(3.50,0.00)}{\pgfxy(0.00,2.00)}
\pgfcircle[fill]{\pgfxy(21.00,10.00)}{1.00mm}
\pgfcircle[stroke]{\pgfxy(21.00,10.00)}{1.00mm}
\pgfcircle[fill]{\pgfxy(24.00,10.00)}{1.00mm}
\pgfcircle[stroke]{\pgfxy(24.00,10.00)}{1.00mm}
\pgfmoveto{\pgfxy(0.00,10.00)}\pgflineto{\pgfxy(30.00,10.00)}\pgfstroke
\pgfellipse[stroke]{\pgfxy(7.50,10.00)}{\pgfxy(3.50,0.00)}{\pgfxy(0.00,2.00)}
\pgfcircle[fill]{\pgfxy(6.00,10.00)}{1.00mm}
\pgfcircle[stroke]{\pgfxy(6.00,10.00)}{1.00mm}
\pgfcircle[fill]{\pgfxy(9.00,10.00)}{1.00mm}
\pgfcircle[stroke]{\pgfxy(9.00,10.00)}{1.00mm}
\pgfputat{\pgfxy(7.50,4.00)}{\pgfbox[bottom,left]{\fontsize{11.38}{13.66}\selectfont \makebox[0pt]{$i$}}}
\pgfputat{\pgfxy(22.50,4.00)}{\pgfbox[bottom,left]{\fontsize{11.38}{13.66}\selectfont \makebox[0pt]{$\pi_i$}}}
\pgfputat{\pgfxy(62.50,4.00)}{\pgfbox[bottom,left]{\fontsize{11.38}{13.66}\selectfont \makebox[0pt]{$i$}}}
\pgfputat{\pgfxy(47.50,4.00)}{\pgfbox[bottom,left]{\fontsize{11.38}{13.66}\selectfont \makebox[0pt]{$\pi_i$}}}
\pgfputat{\pgfxy(87.50,4.00)}{\pgfbox[bottom,left]{\fontsize{11.38}{13.66}\selectfont \makebox[0pt]{$i$}}}
\pgfputat{\pgfxy(102.50,4.00)}{\pgfbox[bottom,left]{\fontsize{11.38}{13.66}\selectfont \makebox[0pt]{$\pi_i$}}}
\pgfputat{\pgfxy(142.50,4.00)}{\pgfbox[bottom,left]{\fontsize{11.38}{13.66}\selectfont \makebox[0pt]{$i$}}}
\pgfputat{\pgfxy(127.50,4.00)}{\pgfbox[bottom,left]{\fontsize{11.38}{13.66}\selectfont \makebox[0pt]{$\pi_i$}}}
\pgfputat{\pgfxy(55.00,-5.00)}{\pgfbox[bottom,left]{\fontsize{11.38}{13.66}\selectfont \makebox[0pt]{$\pi_{i} < i, \z_i=0$}}}
\pgfputat{\pgfxy(135.00,-5.00)}{\pgfbox[bottom,left]{\fontsize{11.38}{13.66}\selectfont \makebox[0pt]{$\pi_{i} < i, \z_i>0$}}}
\pgfputat{\pgfxy(95.00,-5.00)}{\pgfbox[bottom,left]{\fontsize{11.38}{13.66}\selectfont \makebox[0pt]{$i \le \pi_{i}, \z_i>0$}}}
\end{pgfpicture}%
$$
%In drawing these spiral arcs, do not cross the other spiral arcs below the horizontal line.
Note that $\cros \sigma$ is equal to the number of crossing points of two arcs in the pignose
diagram of $\sigma$.
For example, the pignose diagram of
$$\sigma = {4~7~2~5~1~6~3 \choose 0~1~0~1~2~0~0} \in \Z_3 \wr \SS_7$$
%$\sigma = 4~\bar{7}~2~\bar{5}~\bar{\bar{1}}~6~3 \in \Z_3 \wr \SS_7$
is depicted as follows:
$$
\centering
\begin{pgfpicture}{-30.00mm}{-3.27mm}{87.00mm}{24.90mm}
\pgfsetxvec{\pgfpoint{1.00mm}{0mm}}
\pgfsetyvec{\pgfpoint{0mm}{1.00mm}}
\color[rgb]{0,0,0}\pgfsetlinewidth{0.30mm}\pgfsetdash{}{0mm}
\pgfmoveto{\pgfxy(0.00,10.00)}\pgflineto{\pgfxy(85.00,10.00)}\pgfstroke
\pgfellipse[stroke]{\pgfxy(12.50,10.00)}{\pgfxy(3.50,0.00)}{\pgfxy(0.00,2.00)}
\color[rgb]{0,0,1}\pgfmoveto{\pgfxy(11.00,10.00)}\pgfcurveto{\pgfxy(14.69,13.92)}{\pgfxy(19.65,16.40)}{\pgfxy(25.00,17.00)}\pgfcurveto{\pgfxy(26.66,17.19)}{\pgfxy(28.34,17.19)}{\pgfxy(30.00,17.00)}\pgfcurveto{\pgfxy(35.35,16.40)}{\pgfxy(40.31,13.92)}{\pgfxy(44.00,10.00)}\pgfstroke
\pgfmoveto{\pgfxy(21.00,10.00)}\pgfcurveto{\pgfxy(17.80,7.24)}{\pgfxy(14.05,5.20)}{\pgfxy(10.00,4.00)}\pgfcurveto{\pgfxy(5.14,2.57)}{\pgfxy(0.03,2.39)}{\pgfxy(-5.00,3.00)}\pgfcurveto{\pgfxy(-9.72,3.57)}{\pgfxy(-14.54,5.58)}{\pgfxy(-15.00,10.00)}\pgfcurveto{\pgfxy(-15.48,14.62)}{\pgfxy(-10.82,17.59)}{\pgfxy(-6.00,19.00)}\pgfcurveto{\pgfxy(5.31,22.31)}{\pgfxy(17.23,21.99)}{\pgfxy(29.00,22.00)}\pgfcurveto{\pgfxy(37.73,22.01)}{\pgfxy(46.54,22.20)}{\pgfxy(55.00,20.00)}\pgfcurveto{\pgfxy(62.02,18.17)}{\pgfxy(68.52,14.75)}{\pgfxy(74.00,10.00)}\pgfstroke
\color[rgb]{0,0,0}\pgfmoveto{\pgfxy(31.00,10.00)}\pgfcurveto{\pgfxy(31.25,8.07)}{\pgfxy(29.92,6.30)}{\pgfxy(28.00,6.00)}\pgfcurveto{\pgfxy(27.67,5.95)}{\pgfxy(27.33,5.95)}{\pgfxy(27.00,6.00)}\pgfcurveto{\pgfxy(25.08,6.30)}{\pgfxy(23.75,8.07)}{\pgfxy(24.00,10.00)}\pgfstroke
\color[rgb]{0,0,1}\pgfmoveto{\pgfxy(41.00,10.00)}\pgfcurveto{\pgfxy(25.06,0.52)}{\pgfxy(6.29,-3.02)}{\pgfxy(-12.00,0.00)}\pgfcurveto{\pgfxy(-18.36,1.05)}{\pgfxy(-24.78,4.03)}{\pgfxy(-25.00,10.00)}\pgfcurveto{\pgfxy(-25.23,16.21)}{\pgfxy(-18.63,19.64)}{\pgfxy(-12.00,21.00)}\pgfcurveto{\pgfxy(10.60,25.64)}{\pgfxy(34.12,21.72)}{\pgfxy(54.00,10.00)}\pgfstroke
\color[rgb]{0,0,0}\pgfmoveto{\pgfxy(51.00,10.00)}\pgfcurveto{\pgfxy(49.58,8.10)}{\pgfxy(47.90,6.42)}{\pgfxy(46.00,5.00)}\pgfcurveto{\pgfxy(35.09,-3.18)}{\pgfxy(20.57,-1.57)}{\pgfxy(7.00,0.00)}\pgfcurveto{\pgfxy(-0.07,0.82)}{\pgfxy(-6.68,4.49)}{\pgfxy(-5.00,10.00)}\pgfcurveto{\pgfxy(-3.84,13.80)}{\pgfxy(0.77,14.82)}{\pgfxy(5.00,14.00)}\pgfcurveto{\pgfxy(8.27,13.37)}{\pgfxy(11.34,12.00)}{\pgfxy(14.00,10.00)}\pgfstroke
\color[rgb]{0,0,1}\pgfmoveto{\pgfxy(61.00,10.00)}\pgfcurveto{\pgfxy(59.64,11.26)}{\pgfxy(60.21,13.53)}{\pgfxy(62.00,14.00)}\pgfcurveto{\pgfxy(62.33,14.09)}{\pgfxy(62.67,14.09)}{\pgfxy(63.00,14.00)}\pgfcurveto{\pgfxy(64.79,13.53)}{\pgfxy(65.36,11.26)}{\pgfxy(64.00,10.00)}\pgfstroke
\color[rgb]{0,0,0}\pgfmoveto{\pgfxy(71.00,10.00)}\pgfcurveto{\pgfxy(67.11,4.95)}{\pgfxy(61.34,1.70)}{\pgfxy(55.00,1.00)}\pgfcurveto{\pgfxy(53.34,0.82)}{\pgfxy(51.66,0.82)}{\pgfxy(50.00,1.00)}\pgfcurveto{\pgfxy(43.66,1.70)}{\pgfxy(37.89,4.95)}{\pgfxy(34.00,10.00)}\pgfstroke
\pgfputat{\pgfxy(12.50,5.00)}{\pgfbox[bottom,left]{\fontsize{11.38}{13.66}\selectfont \makebox[0pt]{$1$}}}
\pgfputat{\pgfxy(22.50,5.00)}{\pgfbox[bottom,left]{\fontsize{11.38}{13.66}\selectfont \makebox[0pt]{$2$}}}
\pgfputat{\pgfxy(32.50,5.00)}{\pgfbox[bottom,left]{\fontsize{11.38}{13.66}\selectfont \makebox[0pt]{$3$}}}
\pgfputat{\pgfxy(42.50,5.00)}{\pgfbox[bottom,left]{\fontsize{11.38}{13.66}\selectfont \makebox[0pt]{$4$}}}
\pgfputat{\pgfxy(52.50,5.00)}{\pgfbox[bottom,left]{\fontsize{11.38}{13.66}\selectfont \makebox[0pt]{$5$}}}
\pgfputat{\pgfxy(62.50,5.00)}{\pgfbox[bottom,left]{\fontsize{11.38}{13.66}\selectfont \makebox[0pt]{$6$}}}
\pgfputat{\pgfxy(72.50,5.00)}{\pgfbox[bottom,left]{\fontsize{11.38}{13.66}\selectfont \makebox[0pt]{$7$}}}
\pgfcircle[fill]{\pgfxy(11.00,10.00)}{1.00mm}
\pgfcircle[stroke]{\pgfxy(11.00,10.00)}{1.00mm}
\pgfcircle[fill]{\pgfxy(14.00,10.00)}{1.00mm}
\pgfcircle[stroke]{\pgfxy(14.00,10.00)}{1.00mm}
\pgfellipse[stroke]{\pgfxy(22.50,10.00)}{\pgfxy(3.50,0.00)}{\pgfxy(0.00,2.00)}
\pgfcircle[fill]{\pgfxy(24.00,10.00)}{1.00mm}
\pgfcircle[stroke]{\pgfxy(24.00,10.00)}{1.00mm}
\pgfcircle[fill]{\pgfxy(21.00,10.00)}{1.00mm}
\pgfcircle[stroke]{\pgfxy(21.00,10.00)}{1.00mm}
\pgfellipse[stroke]{\pgfxy(32.50,10.00)}{\pgfxy(3.50,0.00)}{\pgfxy(0.00,2.00)}
\pgfcircle[fill]{\pgfxy(34.00,10.00)}{1.00mm}
\pgfcircle[stroke]{\pgfxy(34.00,10.00)}{1.00mm}
\pgfcircle[fill]{\pgfxy(31.00,10.00)}{1.00mm}
\pgfcircle[stroke]{\pgfxy(31.00,10.00)}{1.00mm}
\pgfellipse[stroke]{\pgfxy(42.50,10.00)}{\pgfxy(3.50,0.00)}{\pgfxy(0.00,2.00)}
\pgfcircle[fill]{\pgfxy(44.00,10.00)}{1.00mm}
\pgfcircle[stroke]{\pgfxy(44.00,10.00)}{1.00mm}
\pgfcircle[fill]{\pgfxy(41.00,10.00)}{1.00mm}
\pgfcircle[stroke]{\pgfxy(41.00,10.00)}{1.00mm}
\pgfellipse[stroke]{\pgfxy(52.50,10.00)}{\pgfxy(3.50,0.00)}{\pgfxy(0.00,2.00)}
\pgfcircle[fill]{\pgfxy(54.00,10.00)}{1.00mm}
\pgfcircle[stroke]{\pgfxy(54.00,10.00)}{1.00mm}
\pgfcircle[fill]{\pgfxy(51.00,10.00)}{1.00mm}
\pgfcircle[stroke]{\pgfxy(51.00,10.00)}{1.00mm}
\pgfellipse[stroke]{\pgfxy(62.50,10.00)}{\pgfxy(3.50,0.00)}{\pgfxy(0.00,2.00)}
\pgfcircle[fill]{\pgfxy(64.00,10.00)}{1.00mm}
\pgfcircle[stroke]{\pgfxy(64.00,10.00)}{1.00mm}
\pgfcircle[fill]{\pgfxy(61.00,10.00)}{1.00mm}
\pgfcircle[stroke]{\pgfxy(61.00,10.00)}{1.00mm}
\pgfellipse[stroke]{\pgfxy(72.50,10.00)}{\pgfxy(3.50,0.00)}{\pgfxy(0.00,2.00)}
\pgfcircle[fill]{\pgfxy(74.00,10.00)}{1.00mm}
\pgfcircle[stroke]{\pgfxy(74.00,10.00)}{1.00mm}
\pgfcircle[fill]{\pgfxy(71.00,10.00)}{1.00mm}
\pgfcircle[stroke]{\pgfxy(71.00,10.00)}{1.00mm}
\pgfputat{\pgfxy(-16.00,10.00)}{\pgfbox[bottom,left]{\fontsize{11.38}{13.66}\selectfont \makebox[0pt][r]{$1$}}}
\pgfputat{\pgfxy(-26.00,10.00)}{\pgfbox[bottom,left]{\fontsize{11.38}{13.66}\selectfont \makebox[0pt][r]{$1$}}}
\pgfputat{\pgfxy(-6.00,10.00)}{\pgfbox[bottom,left]{\fontsize{11.38}{13.66}\selectfont \makebox[0pt][r]{$2$}}}
\end{pgfpicture}%
$$
%$$
%\includegraphics[scale=0.8]{pignose.pdf}
%$$
and $\cros(\sigma) = 6$.

%%%%%%%%%%%%%%%%%%%%%%%%%%%%%%%%%%%%%%%%%%%%%%%%%%%%%
\subsection{Bijection $\Phi$}
We construct a bijection $\Phi$ between $\Z_r \wr \SS_n$  and the $r$-colored Laguerre histories of size $n$ by generalizing the Foata-Zeilberger bijection.
Given an $r$-colored permutation $\sigma = \sigma(1) \dots \sigma(n) \in \Z_r \wr \SS_n$,
for $k=1,\ldots, n$,
we build  successively the partial pignose diagrams of the restrictions of $\sigma$ on $[i]$ by drawing the pignoses labled with $1, \ldots, k$ along
with a  half-arc connecting with each  vertex in these pignoses. Denote by $P^{(k)}$ the partial pignose diagram on $[k]$. Let $P^{(0)}$ denote the $n$ pignoses on the horizontal line. Also $\gamma^{(0)} = ((0,0))$ and $\xi^{(0)} = \emptyset$.
% where $\gamma_0 = $.
For $1 \le k \le n$, assume that the mappings
$P^{(k-1)}\mapsto (\gamma^{(k-1)}, \xi^{(k-1)} )$, where
$$
\text{ $\gamma^{(k-1)}=(\gamma_0, \dots, \gamma_{k-1})$ and $\xi^{(k-1)} = (\xi_1, \dots, \xi_{k-1})$},
$$
are already defined from $\sigma$.  We show how to extend this to $k$.
First we draw a half-arc starting
 from the left vertex of the $k$th pignose in $P^{(k-1)}$ in one of the  three possible ways:
$$
%\input{leftvertex.TpX}
%<TpX v="5" TeXFormat="pgf" PdfTeXFormat="pgf" ArrowsSize="0.7" StarsSize="1" DefaultFontHeight="4" DefaultSymbolSize="30" ApproximationPrecision="0.01" PicScale="0.8" Border="2" BitmapRes="20000" HatchingStep="2" DottedSize="0.5" DashSize="1" LineWidth="0.3" TeXCenterFigure="1" TeXFigure="none">
%  <line x1="20" y1="20" x2="65" y2="20"/>
%  <star x="41" y="20"/>
%  <star x="44" y="20"/>
%  <ellipse x="42.5" y="20" dx="7" dy="4"/>
%  <curve lc="blue">41,20 50,28 60,30</curve>
%  <line x1="80" y1="20" x2="125" y2="20"/>
%  <star x="101" y="20"/>
%  <star x="104" y="20"/>
%  <ellipse x="102.5" y="20" dx="7" dy="4"/>
%  <curve lc="blue">101,20 95,15 75,20 80,25 95,29 120,30</curve>
%  <curve>92,16</curve>
%  <line x1="-35" y1="20" x2="10" y2="20"/>
%  <star x="-14" y="20"/>
%  <star x="-11" y="20"/>
%  <ellipse x="-12.5" y="20" dx="7" dy="4"/>
%  <curve>-14,20 -20,13 -30,10</curve>
%  <text x="-12.5" y="5" t="p_k&lt;k" tex="$\pi_k &lt; k$" h="4" halign="c"/>
%  <text x="42.5" y="5" t="p_k&lt;k, z_k =0" tex="$k\le \pi_k, \z_k = 0$" h="4" halign="c"/>
%  <text x="102.5" y="5" t="k&lt; p_k, z_k &gt;0" tex="$k\le \pi_k, \z_k &gt; 0$" h="4" halign="c"/>
%</TpX>
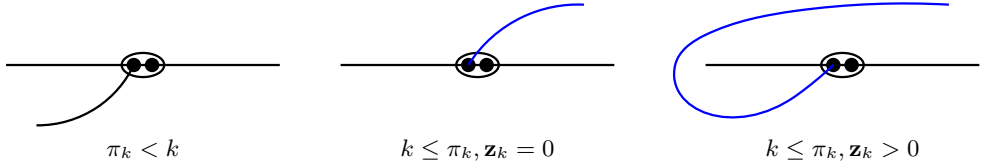
\centering
\begin{pgfpicture}{-30.00mm}{1.31mm}{102.00mm}{26.16mm}
\pgfsetxvec{\pgfpoint{0.80mm}{0mm}}
\pgfsetyvec{\pgfpoint{0mm}{0.80mm}}
\color[rgb]{0,0,0}\pgfsetlinewidth{0.30mm}\pgfsetdash{}{0mm}
\pgfmoveto{\pgfxy(20.00,20.00)}\pgflineto{\pgfxy(65.00,20.00)}\pgfstroke
\pgfcircle[fill]{\pgfxy(41.00,20.00)}{0.80mm}
\pgfcircle[stroke]{\pgfxy(41.00,20.00)}{0.80mm}
\pgfcircle[fill]{\pgfxy(44.00,20.00)}{0.80mm}
\pgfcircle[stroke]{\pgfxy(44.00,20.00)}{0.80mm}
\pgfellipse[stroke]{\pgfxy(42.50,20.00)}{\pgfxy(3.50,0.00)}{\pgfxy(0.00,2.00)}
\color[rgb]{0,0,1}\pgfmoveto{\pgfxy(41.00,20.00)}\pgfcurveto{\pgfxy(43.17,23.48)}{\pgfxy(46.29,26.25)}{\pgfxy(50.00,28.00)}\pgfcurveto{\pgfxy(53.12,29.47)}{\pgfxy(56.55,30.16)}{\pgfxy(60.00,30.00)}\pgfstroke
\color[rgb]{0,0,0}\pgfmoveto{\pgfxy(80.00,20.00)}\pgflineto{\pgfxy(125.00,20.00)}\pgfstroke
\pgfcircle[fill]{\pgfxy(101.00,20.00)}{0.80mm}
\pgfcircle[stroke]{\pgfxy(101.00,20.00)}{0.80mm}
\pgfcircle[fill]{\pgfxy(104.00,20.00)}{0.80mm}
\pgfcircle[stroke]{\pgfxy(104.00,20.00)}{0.80mm}
\pgfellipse[stroke]{\pgfxy(102.50,20.00)}{\pgfxy(3.50,0.00)}{\pgfxy(0.00,2.00)}
\color[rgb]{0,0,1}\pgfmoveto{\pgfxy(101.00,20.00)}\pgfcurveto{\pgfxy(99.12,18.20)}{\pgfxy(97.11,16.53)}{\pgfxy(95.00,15.00)}\pgfcurveto{\pgfxy(84.69,7.55)}{\pgfxy(73.23,12.52)}{\pgfxy(75.00,20.00)}\pgfcurveto{\pgfxy(75.57,22.41)}{\pgfxy(77.76,23.91)}{\pgfxy(80.00,25.00)}\pgfcurveto{\pgfxy(84.69,27.27)}{\pgfxy(89.84,28.30)}{\pgfxy(95.00,29.00)}\pgfcurveto{\pgfxy(103.28,30.13)}{\pgfxy(111.65,30.47)}{\pgfxy(120.00,30.00)}\pgfstroke
\color[rgb]{0,0,0}\pgfmoveto{\pgfxy(-35.00,20.00)}\pgflineto{\pgfxy(10.00,20.00)}\pgfstroke
\pgfcircle[fill]{\pgfxy(-14.00,20.00)}{0.80mm}
\pgfcircle[stroke]{\pgfxy(-14.00,20.00)}{0.80mm}
\pgfcircle[fill]{\pgfxy(-11.00,20.00)}{0.80mm}
\pgfcircle[stroke]{\pgfxy(-11.00,20.00)}{0.80mm}
\pgfellipse[stroke]{\pgfxy(-12.50,20.00)}{\pgfxy(3.50,0.00)}{\pgfxy(0.00,2.00)}
\pgfmoveto{\pgfxy(-14.00,20.00)}\pgfcurveto{\pgfxy(-15.34,17.17)}{\pgfxy(-17.41,14.75)}{\pgfxy(-20.00,13.00)}\pgfcurveto{\pgfxy(-22.95,11.01)}{\pgfxy(-26.44,9.96)}{\pgfxy(-30.00,10.00)}\pgfstroke
\pgfputat{\pgfxy(-12.50,5.00)}{\pgfbox[bottom,left]{\fontsize{9.10}{10.93}\selectfont \makebox[0pt]{$\pi_k < k$}}}
\pgfputat{\pgfxy(42.50,5.00)}{\pgfbox[bottom,left]{\fontsize{9.10}{10.93}\selectfont \makebox[0pt]{$k\le \pi_k, \z_k = 0$}}}
\pgfputat{\pgfxy(102.50,5.00)}{\pgfbox[bottom,left]{\fontsize{9.10}{10.93}\selectfont \makebox[0pt]{$k\le \pi_k, \z_k > 0$}}}
\end{pgfpicture}%
$$
For the first case let
\begin{align*}
p_k
= &\#\set{j: \z_k=\z_j=0, \pi_k< \pi_j < k <j } + \#\set{j: \z_k>0, \z_j=0, \pi_j<k<j}\\
  &+\#\set{j: \z_k>0, \z_j>0, \pi_j<\pi_k < k < j } + 1,
\end{align*}
which is positive,  for the second case let $p_k = 0$, and for the third case let
$p_k = -\z_k$, which  is negative.
If $p_k$ is positive, then $(p_k-1)$ means the number of new crossing points by connecting two vertices by a new arc into a partial pignose diagram, otherwise $-p_k$ means the color of new arc drawn into a partial pignose diagram.

Then we get the partial diagram $P^{(k)}$ by  drawing a half-arc connecting the right vertex in $k$th pignose in the above diagram, for the index $\ell$ such that $\pi_{\ell}= k$, as one of the followings.
$$
%\input{rightvertex.TpX}
%<TpX v="5" TeXFormat="pgf" PdfTeXFormat="pgf" ArrowsSize="0.7" StarsSize="1" DefaultFontHeight="4" DefaultSymbolSize="30" ApproximationPrecision="0.01" PicScale="0.8" Border="2" BitmapRes="20000" HatchingStep="2" DottedSize="0.5" DashSize="1" LineWidth="0.3" TeXCenterFigure="1" TeXFigure="none">
%  <line x1="-35" y1="20" x2="10" y2="20"/>
%  <star x="-14" y="20"/>
%  <star x="-11" y="20"/>
%  <ellipse x="-12.5" y="20" dx="7" dy="4"/>
%  <curve lc="blue">-11,20 -20,28 -30,30</curve>
%  <line x1="80" y1="20.0201" x2="125" y2="20"/>
%  <star x="101" y="20.0201"/>
%  <star x="104" y="20.0201"/>
%  <ellipse x="102.5" y="20.0201" dx="7" dy="4"/>
%  <curve>104,20 92,27 75,20 80,15.0201 95,11.0201 120,10.0201</curve>
%  <curve>92,16.5</curve>
%  <line x1="20" y1="20" x2="65" y2="20"/>
%  <star x="41" y="20"/>
%  <star x="44" y="20"/>
%  <ellipse x="42.5" y="20" dx="7" dy="4"/>
%  <curve>44,20 50,13 60,10</curve>
%  <text x="-12.5" y="5" t="l &lt; p_l" tex="$\ell \le \pi_{\ell}$" h="4" halign="c"/>
%  <text x="42.5" y="5" t="p_l &lt; l, z_l =0" tex="$\pi_{\ell} &lt; \ell, \z_{\ell} = 0$" h="4" halign="c"/>
%  <text x="102.5" y="5" t="p_l &lt; l, z_l &gt;0" tex="$\pi_{\ell} &lt; \ell, \z_{\ell} &gt; 0$" h="4" halign="c"/>
%</TpX>
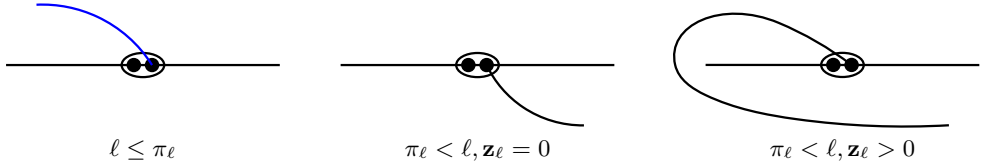
\centering
\begin{pgfpicture}{-30.00mm}{1.31mm}{102.00mm}{26.02mm}
\pgfsetxvec{\pgfpoint{0.80mm}{0mm}}
\pgfsetyvec{\pgfpoint{0mm}{0.80mm}}
\color[rgb]{0,0,0}\pgfsetlinewidth{0.30mm}\pgfsetdash{}{0mm}
\pgfmoveto{\pgfxy(-35.00,20.00)}\pgflineto{\pgfxy(10.00,20.00)}\pgfstroke
\pgfcircle[fill]{\pgfxy(-14.00,20.00)}{0.80mm}
\pgfcircle[stroke]{\pgfxy(-14.00,20.00)}{0.80mm}
\pgfcircle[fill]{\pgfxy(-11.00,20.00)}{0.80mm}
\pgfcircle[stroke]{\pgfxy(-11.00,20.00)}{0.80mm}
\pgfellipse[stroke]{\pgfxy(-12.50,20.00)}{\pgfxy(3.50,0.00)}{\pgfxy(0.00,2.00)}
\color[rgb]{0,0,1}\pgfmoveto{\pgfxy(-11.00,20.00)}\pgfcurveto{\pgfxy(-13.17,23.48)}{\pgfxy(-16.29,26.25)}{\pgfxy(-20.00,28.00)}\pgfcurveto{\pgfxy(-23.12,29.47)}{\pgfxy(-26.55,30.16)}{\pgfxy(-30.00,30.00)}\pgfstroke
\color[rgb]{0,0,0}\pgfmoveto{\pgfxy(80.00,20.02)}\pgflineto{\pgfxy(125.00,20.00)}\pgfstroke
\pgfcircle[fill]{\pgfxy(101.00,20.02)}{0.80mm}
\pgfcircle[stroke]{\pgfxy(101.00,20.02)}{0.80mm}
\pgfcircle[fill]{\pgfxy(104.00,20.02)}{0.80mm}
\pgfcircle[stroke]{\pgfxy(104.00,20.02)}{0.80mm}
\pgfellipse[stroke]{\pgfxy(102.50,20.02)}{\pgfxy(3.50,0.00)}{\pgfxy(0.00,2.00)}
\pgfmoveto{\pgfxy(104.00,20.00)}\pgfcurveto{\pgfxy(100.32,22.84)}{\pgfxy(96.28,25.19)}{\pgfxy(92.00,27.00)}\pgfcurveto{\pgfxy(82.68,30.94)}{\pgfxy(73.44,26.67)}{\pgfxy(75.00,20.00)}\pgfcurveto{\pgfxy(75.56,17.59)}{\pgfxy(77.75,16.10)}{\pgfxy(80.00,15.02)}\pgfcurveto{\pgfxy(84.70,12.76)}{\pgfxy(89.84,11.73)}{\pgfxy(95.00,11.02)}\pgfcurveto{\pgfxy(103.28,9.88)}{\pgfxy(111.65,9.54)}{\pgfxy(120.00,10.02)}\pgfstroke
\pgfmoveto{\pgfxy(20.00,20.00)}\pgflineto{\pgfxy(65.00,20.00)}\pgfstroke
\pgfcircle[fill]{\pgfxy(41.00,20.00)}{0.80mm}
\pgfcircle[stroke]{\pgfxy(41.00,20.00)}{0.80mm}
\pgfcircle[fill]{\pgfxy(44.00,20.00)}{0.80mm}
\pgfcircle[stroke]{\pgfxy(44.00,20.00)}{0.80mm}
\pgfellipse[stroke]{\pgfxy(42.50,20.00)}{\pgfxy(3.50,0.00)}{\pgfxy(0.00,2.00)}
\pgfmoveto{\pgfxy(44.00,20.00)}\pgfcurveto{\pgfxy(45.34,17.17)}{\pgfxy(47.41,14.75)}{\pgfxy(50.00,13.00)}\pgfcurveto{\pgfxy(52.95,11.01)}{\pgfxy(56.44,9.96)}{\pgfxy(60.00,10.00)}\pgfstroke
\pgfputat{\pgfxy(-12.50,5.00)}{\pgfbox[bottom,left]{\fontsize{9.10}{10.93}\selectfont \makebox[0pt]{$\ell \le \pi_{\ell}$}}}
\pgfputat{\pgfxy(42.50,5.00)}{\pgfbox[bottom,left]{\fontsize{9.10}{10.93}\selectfont \makebox[0pt]{$\pi_{\ell} < \ell, \z_{\ell} = 0$}}}
\pgfputat{\pgfxy(102.50,5.00)}{\pgfbox[bottom,left]{\fontsize{9.10}{10.93}\selectfont \makebox[0pt]{$\pi_{\ell} < \ell, \z_{\ell} > 0$}}}
\end{pgfpicture}%
$$
For the first case let
\begin{align*}
q_k
= &\#\set{j: \z_{\ell}=\z_j=0, {\ell} < j \le \pi_{\ell}< \pi_j } + \#\set{j: \z_{\ell}>0, \z_j=0, j\le \pi_{\ell} < \pi_j}\\
  &+\#\set{i: \z_i>0, \z_{\ell}>0, i<\ell \le \pi_{\ell}<\pi_i } + 1,
\end{align*}
which is positive,
for the second case let $q_k = 0$, and
for the third case let $q_k = -\z_{\ell}$, which is negative.
Simliarly, if $q_k$ is positive, $(q_k-1)$ means the number of new crossing points by connecting two vertices by a new arc into a partial pignose diagram, otherwise $-q_k$ means the color of new arc drawn into a partial pignose diagram.

Lastly, when $\gamma_{k-1} = (k-1, h_{k-1})$, we get $\gamma_k = (k, h_k)$ by defining
\begin{enumerate}[(i)]
\item $h_k = h_{k-1} - 1$ (South-East) if both of $p_i$ and $q_i$ are positive,
\item $h_k = h_{k-1}$ (East) if exactly one of $p_i$ and $q_i$ is positive, and
\item $h_k = h_{k-1} + 1$ (North-East) if none of $p_i$ and $q_i$ are positive.
\end{enumerate}
Actually, $h_k$ means half of the number of open half-arcs, which do not become complete-arcs, in the partial pignose diagram $P^{(k)}$.

By taking the above process inductively for $k=1, \dots, n$, we obtain the pignose diagram $P=P^{(n)}$ of $\sigma$.
Since there is no half-arcs in $P^{(n)}$, we have $h_n=0$ and $\gamma^{(n)}$ is a Motzkin path of length $n$.
As the $r$-colored Laguerre history corresponding to an $r$-colored permutation $\sigma$, let $\Phi(\sigma) = (\gamma^{(n)}, \xi^{(n)})$.

%(Formal definition)
%Given a $r$-colored permutation $\sigma\in \Z_r \wr \SS_n$,
%for an integer $1\le x \le 2n+1$,
%consider two indices sets $H_{x}$ and $\tilde{H}_{x}$ defined by
%$$
%H_x         = \set{j: j\le \frac{x}{2} \le \abs{\sigma_j}} \quad \text{and} \quad
%\tilde{H}_x = \set{j: \abs{\sigma_j} < \frac{x}{2} < j},
%$$
%and $h_x = \abs{H_x}$ and $\tilde{h}_x = \abs{\tilde{H}_x}$.
%It is trivial that $H_i$ is obtained from $H_{i-1}$ by inserting or removing
%Note that  $H_x$ (resp. $\tilde{H}_x$) depends only on blue (resp. black) arcs in the pignose diagram of $\sigma$ and $h_{2i+1} = \tilde{h}_{2i+1}$ for $0\le i \le n$.
%And then, we make a Motzkin path $\gamma=(\gamma_0, \dots, \gamma_n)$ by defining
%$$\gamma_i = (i, h_{2i+1}) \text{ for } 0\le i \le n$$
%and a integer-pair sequence $\xi = ((p_1,q_1), \dots, (p_n, q_n))$ by
%$$
%p_i=
%\begin{cases}
%\tilde{o}_i&\text{if $\tilde{h}_{2i-1}-\tilde{h}_{2i} = 1$,}\\
%-\cnum(\sigma(i))&\text{otherwise,}
%\end{cases}
%$$
%and
%$$
%q_i=
%\begin{cases}
%o_i&\text{if $h_{2i}-h_{2i+1}=1$,}\\
%-\cnum(\sigma(\abs{\sigma^{-1}(i)})) &\text{otherwise,}
%\end{cases}
%$$
%where $o_i$ is the order of the number, which belongs to $H_{2i}$, but not belongs to $H_{2i+1}$, among $H_{2i}$
%and $\tilde{o}_i$ is the order of the number, which belongs to $\tilde{H}_{2i-1}$, but not belongs to $\tilde{H}_{2i}$, among $\tilde{H}_{2i-1}.$
%

Note that for $r=1$, the above bijection  is  the same as  the Foata-Zeilberger bijection.
In Figure~\ref{fig:bijection}, we run the above algorithm for
the permutation $\sigma = 4~\bar{7}~2~\bar{5}~\bar{\bar{1}}~6~3 \in \Z_3 \wr \SS_7$ in the previous subsection.
% produces the partial
%pignose diagrams and partial Laguerre histories under the above algorithm.

\begin{figure}[p]
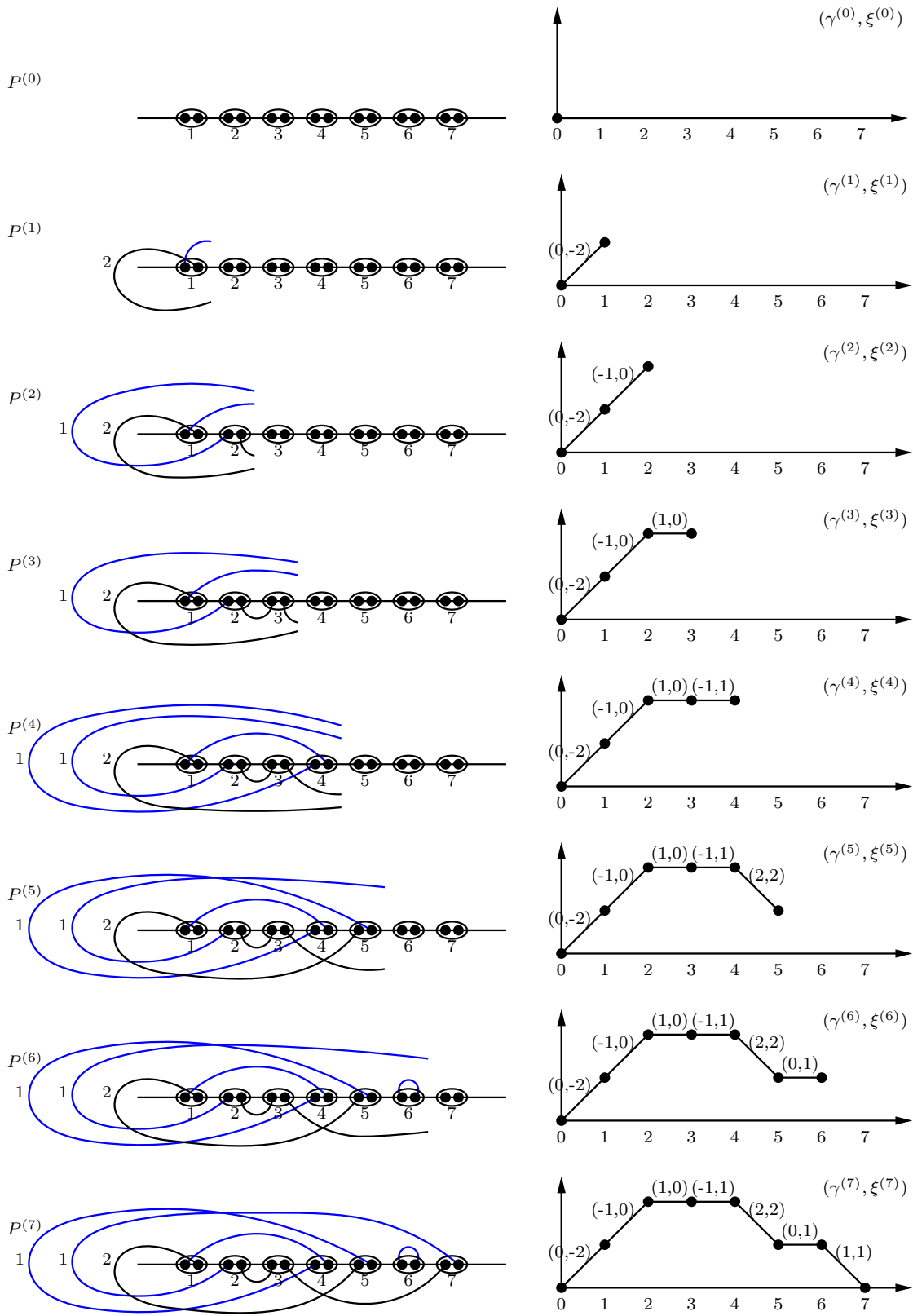

\begin{tabular}{rl}
$
\centering
\begin{pgfpicture}{-23.00mm}{0.90mm}{61.50mm}{16.10mm}
\pgfsetxvec{\pgfpoint{0.70mm}{0mm}}
\pgfsetyvec{\pgfpoint{0mm}{0.70mm}}
\color[rgb]{0,0,0}\pgfsetlinewidth{0.30mm}\pgfsetdash{}{0mm}
\pgfmoveto{\pgfxy(0.00,10.00)}\pgflineto{\pgfxy(85.00,10.00)}\pgfstroke
\pgfellipse[stroke]{\pgfxy(12.50,10.00)}{\pgfxy(3.50,0.00)}{\pgfxy(0.00,2.00)}
\pgfputat{\pgfxy(12.50,5.00)}{\pgfbox[bottom,left]{\fontsize{7.97}{9.56}\selectfont \makebox[0pt]{$1$}}}
\pgfputat{\pgfxy(22.50,5.00)}{\pgfbox[bottom,left]{\fontsize{7.97}{9.56}\selectfont \makebox[0pt]{$2$}}}
\pgfputat{\pgfxy(32.50,5.00)}{\pgfbox[bottom,left]{\fontsize{7.97}{9.56}\selectfont \makebox[0pt]{$3$}}}
\pgfputat{\pgfxy(42.50,5.00)}{\pgfbox[bottom,left]{\fontsize{7.97}{9.56}\selectfont \makebox[0pt]{$4$}}}
\pgfputat{\pgfxy(52.50,5.00)}{\pgfbox[bottom,left]{\fontsize{7.97}{9.56}\selectfont \makebox[0pt]{$5$}}}
\pgfputat{\pgfxy(62.50,5.00)}{\pgfbox[bottom,left]{\fontsize{7.97}{9.56}\selectfont \makebox[0pt]{$6$}}}
\pgfputat{\pgfxy(72.50,5.00)}{\pgfbox[bottom,left]{\fontsize{7.97}{9.56}\selectfont \makebox[0pt]{$7$}}}
\pgfcircle[fill]{\pgfxy(11.00,10.00)}{0.70mm}
\pgfcircle[stroke]{\pgfxy(11.00,10.00)}{0.70mm}
\pgfcircle[fill]{\pgfxy(14.00,10.00)}{0.70mm}
\pgfcircle[stroke]{\pgfxy(14.00,10.00)}{0.70mm}
\pgfellipse[stroke]{\pgfxy(22.50,10.00)}{\pgfxy(3.50,0.00)}{\pgfxy(0.00,2.00)}
\pgfcircle[fill]{\pgfxy(24.00,10.00)}{0.70mm}
\pgfcircle[stroke]{\pgfxy(24.00,10.00)}{0.70mm}
\pgfcircle[fill]{\pgfxy(21.00,10.00)}{0.70mm}
\pgfcircle[stroke]{\pgfxy(21.00,10.00)}{0.70mm}
\pgfellipse[stroke]{\pgfxy(32.50,10.00)}{\pgfxy(3.50,0.00)}{\pgfxy(0.00,2.00)}
\pgfcircle[fill]{\pgfxy(34.00,10.00)}{0.70mm}
\pgfcircle[stroke]{\pgfxy(34.00,10.00)}{0.70mm}
\pgfcircle[fill]{\pgfxy(31.00,10.00)}{0.70mm}
\pgfcircle[stroke]{\pgfxy(31.00,10.00)}{0.70mm}
\pgfellipse[stroke]{\pgfxy(42.50,10.00)}{\pgfxy(3.50,0.00)}{\pgfxy(0.00,2.00)}
\pgfcircle[fill]{\pgfxy(44.00,10.00)}{0.70mm}
\pgfcircle[stroke]{\pgfxy(44.00,10.00)}{0.70mm}
\pgfcircle[fill]{\pgfxy(41.00,10.00)}{0.70mm}
\pgfcircle[stroke]{\pgfxy(41.00,10.00)}{0.70mm}
\pgfellipse[stroke]{\pgfxy(52.50,10.00)}{\pgfxy(3.50,0.00)}{\pgfxy(0.00,2.00)}
\pgfcircle[fill]{\pgfxy(54.00,10.00)}{0.70mm}
\pgfcircle[stroke]{\pgfxy(54.00,10.00)}{0.70mm}
\pgfcircle[fill]{\pgfxy(51.00,10.00)}{0.70mm}
\pgfcircle[stroke]{\pgfxy(51.00,10.00)}{0.70mm}
\pgfellipse[stroke]{\pgfxy(62.50,10.00)}{\pgfxy(3.50,0.00)}{\pgfxy(0.00,2.00)}
\pgfcircle[fill]{\pgfxy(64.00,10.00)}{0.70mm}
\pgfcircle[stroke]{\pgfxy(64.00,10.00)}{0.70mm}
\pgfcircle[fill]{\pgfxy(61.00,10.00)}{0.70mm}
\pgfcircle[stroke]{\pgfxy(61.00,10.00)}{0.70mm}
\pgfellipse[stroke]{\pgfxy(72.50,10.00)}{\pgfxy(3.50,0.00)}{\pgfxy(0.00,2.00)}
\pgfcircle[fill]{\pgfxy(74.00,10.00)}{0.70mm}
\pgfcircle[stroke]{\pgfxy(74.00,10.00)}{0.70mm}
\pgfcircle[fill]{\pgfxy(71.00,10.00)}{0.70mm}
\pgfcircle[stroke]{\pgfxy(71.00,10.00)}{0.70mm}
\pgfputat{\pgfxy(-30.00,17.00)}{\pgfbox[bottom,left]{\fontsize{7.97}{9.56}\selectfont $P^{(0)}$}}
\end{pgfpicture}%
$ &
$
%\input{laguerre0.TpX}
%<TpX v="5" TeXFormat="pgf" PdfTeXFormat="pgf" ArrowsSize="0.7" StarsSize="1" DefaultFontHeight="4" DefaultSymbolSize="30" ApproximationPrecision="0.01" PicScale="0.7" Border="2" BitmapRes="20000" HatchingStep="2" DottedSize="0.5" DashSize="1" LineWidth="0.3" TeXCenterFigure="1" TeXFigure="none">
%  <line x1="0" y1="0" x2="80" y2="0" arr2="h44"/>
%  <star x="0" y="0"/>
%  <line x1="0" y1="0" x2="0" y2="25" arr2="h44"/>
%  <text x="70" y="22" t="(g0,x0)" tex="$(\gamma^{(0)}, \xi^{(0)})$" h="4" halign="c"/>
%  <text x="0" y="-5" t="0" h="4" halign="c"/>
%  <text x="70" y="-5" t="7" h="4" halign="c"/>
%  <text x="60" y="-5" t="6" h="4" halign="c"/>
%  <text x="50" y="-5" t="5" h="4" halign="c"/>
%  <text x="40" y="-5" t="4" h="4" halign="c"/>
%  <text x="30" y="-5" t="3" h="4" halign="c"/>
%  <text x="20" y="-5" t="2" h="4" halign="c"/>
%  <text x="10" y="-5" t="1" h="4" halign="c"/>
%</TpX>
\centering
\begin{pgfpicture}{-2.70mm}{-6.10mm}{58.00mm}{19.60mm}
\pgfsetxvec{\pgfpoint{0.70mm}{0mm}}
\pgfsetyvec{\pgfpoint{0mm}{0.70mm}}
\color[rgb]{0,0,0}\pgfsetlinewidth{0.30mm}\pgfsetdash{}{0mm}
\pgfmoveto{\pgfxy(0.00,0.00)}\pgflineto{\pgfxy(80.00,0.00)}\pgfstroke
\pgfmoveto{\pgfxy(80.00,0.00)}\pgflineto{\pgfxy(77.20,0.70)}\pgflineto{\pgfxy(77.20,-0.70)}\pgflineto{\pgfxy(80.00,0.00)}\pgfclosepath\pgffill
\pgfmoveto{\pgfxy(80.00,0.00)}\pgflineto{\pgfxy(77.20,0.70)}\pgflineto{\pgfxy(77.20,-0.70)}\pgflineto{\pgfxy(80.00,0.00)}\pgfclosepath\pgfstroke
\pgfcircle[fill]{\pgfxy(0.00,0.00)}{0.70mm}
\pgfcircle[stroke]{\pgfxy(0.00,0.00)}{0.70mm}
\pgfmoveto{\pgfxy(0.00,0.00)}\pgflineto{\pgfxy(0.00,25.00)}\pgfstroke
\pgfmoveto{\pgfxy(0.00,25.00)}\pgflineto{\pgfxy(-0.70,22.20)}\pgflineto{\pgfxy(0.70,22.20)}\pgflineto{\pgfxy(0.00,25.00)}\pgfclosepath\pgffill
\pgfmoveto{\pgfxy(0.00,25.00)}\pgflineto{\pgfxy(-0.70,22.20)}\pgflineto{\pgfxy(0.70,22.20)}\pgflineto{\pgfxy(0.00,25.00)}\pgfclosepath\pgfstroke
\pgfputat{\pgfxy(70.00,22.00)}{\pgfbox[bottom,left]{\fontsize{7.97}{9.56}\selectfont \makebox[0pt]{$(\gamma^{(0)}, \xi^{(0)})$}}}
\pgfputat{\pgfxy(0.00,-5.00)}{\pgfbox[bottom,left]{\fontsize{7.97}{9.56}\selectfont \makebox[0pt]{0}}}
\pgfputat{\pgfxy(70.00,-5.00)}{\pgfbox[bottom,left]{\fontsize{7.97}{9.56}\selectfont \makebox[0pt]{7}}}
\pgfputat{\pgfxy(60.00,-5.00)}{\pgfbox[bottom,left]{\fontsize{7.97}{9.56}\selectfont \makebox[0pt]{6}}}
\pgfputat{\pgfxy(50.00,-5.00)}{\pgfbox[bottom,left]{\fontsize{7.97}{9.56}\selectfont \makebox[0pt]{5}}}
\pgfputat{\pgfxy(40.00,-5.00)}{\pgfbox[bottom,left]{\fontsize{7.97}{9.56}\selectfont \makebox[0pt]{4}}}
\pgfputat{\pgfxy(30.00,-5.00)}{\pgfbox[bottom,left]{\fontsize{7.97}{9.56}\selectfont \makebox[0pt]{3}}}
\pgfputat{\pgfxy(20.00,-5.00)}{\pgfbox[bottom,left]{\fontsize{7.97}{9.56}\selectfont \makebox[0pt]{2}}}
\pgfputat{\pgfxy(10.00,-5.00)}{\pgfbox[bottom,left]{\fontsize{7.97}{9.56}\selectfont \makebox[0pt]{1}}}
\end{pgfpicture}%
$ \\
\end{tabular}
\caption{An illustration of the  bijection $\Phi$}
\label{fig:bijection}
\end{figure}

%%%%%%%%%%%%%%%%%%%%%%%%%%%%%%%%%%%%%%%%%%%%%%%%%%%%%
\subsection{Proof of Lemma~\ref{thm:cfb}}
Introduce some more statistics on an $r$-colored permutation  $\sigma\in \Z_r \wr \SS_n$:
\begin{align*}
\fix \sigma &= \#\set{i\in [n]: i = \pi_i \text{ and } \z_i=0},\\
\fixc \sigma &= \#\set{i\in [n]: i = \pi_i \text{ and } \z_i>0},\\
\wexa \sigma &= \#\set{i\in [n]: i \le \pi_i \text{ and } \z_i=0},\\
\wexc \sigma &= \#\set{i\in [n]: i \le \pi_i \text{ and } \z_i>0},\\
\dropa \sigma &= \#\set{i\in [n]: \pi_i <i \text{ and } \z_i=0},\\
\dropc \sigma &= \#\set{i\in [n]: \pi_i <i \text{ and } \z_i>0},\\
\csumw \sigma &= \sum_{1 \le i\le \pi_i \le n} \z_i,
\quad \csumd \sigma = \sum_{1 \le \pi_i < i \le n} \z_i.
\end{align*}
For the above example, we have
\begin{align*}
\fix \sigma&=1, &\wexa \sigma&=2, &\dropa \sigma&=2,\\
\fixc \sigma&=0, &\wexc \sigma&=2, &\dropc \sigma&=1,
\end{align*}
also  $\csumw \sigma =2$  and $\csumd \sigma =2$.

For $n\geq 1$ define the weight of a permutation $\sigma \in \Z_r \wr \SS_n$ by
\begin{align}
\w(\sigma):= q^{\cros\sigma} t^{\wexa\sigma} {\tilde t}^{\dropa\sigma} w^{\wexc\sigma} {\tilde w}^{\dropc\sigma} x^{\fixa\sigma} {\tilde x}^{\fixc\sigma} y^{\csumw \sigma} {\tilde y}^{\csumd \sigma}.
\label{eq:weight}
\end{align}
Then, we have the following refined version of Lemma~\ref{thm:cfb}:
\begin{align}\label{strong:thm}
1+\sum_{n\ge 1} \sum_{\sigma \in \Z_r \wr \SS_n}\w(\sigma) z^n=
\dfrac{1}{
1-b_0 z - \dfrac{a_0 c_1 z^2}{
1-b_1 z - \dfrac{a_1 c_2 z^2}{
%1-b_2 z - \dfrac{a_2 c_3 z^2}{
\cdots}}},
\end{align}
where
\begin{align}
\begin{cases}
a_h = (t + w y[r-1]_{y} q^h) ({\tilde t} + {\tilde w}{\tilde  y} [r-1]_{\tilde y} q^{h+1}),\quad c_h = {[h]_q}^2,\\
b_h = ({\tilde t} + {\tilde w}{\tilde  y}[r-1]_{\tilde y} q^h) [h]_q + t( x + q[h]_q) + wy[r-1]_{y} q^h ([h]_q + {\tilde x} q^h). \end{cases}
\label{eq:coeff}
\end{align}
Clearly, the formula \eqref{strong:thm} becomes Lemma~\ref{thm:cfb} for ${\tilde y}=y$ and ${\tilde t}={\tilde w}={\tilde x}=1$.

In order to interpret $\w(\sigma)$ in \eqref{eq:weight} as
 a weight of $r$-colored Laguerre history $\Phi(\sigma)$,
we define  the weight of $k$th step $(\gamma_{k-1},\gamma_{k})$ of height $h_{k-1}$ with its label $(p_k, q_k)$ of Laguerre history $(\gamma, \xi)$ by
\begin{align*}
\w((\gamma_{k-1},\gamma_{k}), (p_k, q_k)) =
\begin{cases}
\w_l(p_k,h_{k-1})\cdot \w_r(q_k,h_{k-1})   & \text{if $p_k > 0$,}\\
x~\w_l(p_k,h_{k-1})\cdot \w_r(q_k,h_{k-1}+1) & \text{if $p_k = 0$ and $q_k = 1$,}\\
\tilde{x}~\w_l(p_k,h_{k-1})\cdot \w_r(q_k,h_{k-1}+1) & \text{if $p_k <0$ and $q_k = h_{k-1}+1$,}\\
\w_l(p_k,h_{k-1})\cdot \w_r(q_k,h_{k-1}+1) & \text{otherwise,}\\
\end{cases}
\end{align*}
where
\begin{align*}
\w_l(z,h) &=
\begin{cases}
q^{z-1}&\text{if $z>0$,}\\
t&\text{if $z=0$,}\\
wy^{-z}q^h&\text{if $z<0$,}\\
\end{cases}
&
\w_r(z,h) &=
\begin{cases}
q^{z-1}&\text{if $z>0$,}\\
\tilde{t}&\text{if $z=0$,}\\
\tilde{w}\tilde{y}^{-z} q^h&\text{if $z<0$.}\\
\end{cases}
\end{align*}
We fine the weight of $r$-colored Laguerre history $(\gamma, \xi)$ of length $n$ as
$$
\w(\gamma, \xi) = \prod_{k=1}^{n} \w((\gamma_{k-1},\gamma_{k}), (p_k, q_k)).
$$
It is easy to check the followings.
\begin{enumerate}[(i)]
\item $k = \pi_k$, $\z_k = 0$ if and only if $(p_k,q_k) = (0,1)$. (for the statistic $\fix$)
\item $k = \pi_k$, $\z_k > 0$ if and only if $(p_k,q_k) = (-\z_k,h_{k-1}+1)$. (for the statistic $\fixc$)
\item $k > \pi_k$ if and only if $p_k>0$ and we make a full arc, by connecting two half-arcs, which meets $(p_k-1)$ open half-arcs. (for the statistic $\cros$).
\item $k \le \pi_k$, $\z_k = 0$ if and only if $p_k=0$. (for the statistic $\wexa$)
\item $k \le \pi_k$, $\z_k > 0$ if and only if $p_k=-\z_k$. (for two statistics $\wexc$ and $\csumw$)
\item $k = \pi_\ell \ge \ell$ if and only if $q_k>0$ and we make a full arc, by connecting two half-arcs, which meets $(q_k-1)$ open half-arcs. (for the statistic $\cros$)
\item $k = \pi_\ell < \ell$, $\z_{\ell} = 0$ if and only if $q_k=0$. (for the statistic $\dropa$)
\item $k = \pi_\ell < \ell$, $\z_{\ell} > 0$ if and only if $q_k=-\z_{\ell}$. (for two statistics $\dropc$ and $\csumd$)
\item if $\z_{\ell} > 0$, we draw a spiral arc, which meets $h$ open half-arcs. (for the statistic $\cros$)
\end{enumerate}
According to the above conditions, given an $r$-colored permutation $\sigma \in \Z_r \wr \SS_n$, we have $\w(\sigma) = \w(\gamma, \xi),$
where $\Phi(\sigma)=(\gamma, \xi)$.

Summing the weight of possible pairs $(p,q)$ for each step of height $h$,
the weight sum of $r$-colored Laguerre histories associated to a Motzkin path $\gamma$ is equal to the weight $\w(\gamma)$ with
\begin{align}
\begin{cases}
a_h = \sum_{p=1-r}^{0} \sum_{q=1-r}^{0} \w(h, (p, q)), \quad c_h = \sum_{p=1}^{h} \sum_{q=1}^{h} \w(h, (p, q)),\\
b_n = \sum_{p=1}^{h} \sum_{q=1-r}^{0} \w(h, (p, q)) + \sum_{p=1-r}^{0} \sum_{q=1}^{h+1} \w(h, (p, q)),
\end{cases}\label{eq:weightsum}
\end{align}
where
\begin{align*}
\w(h, (p, q)) =
\begin{cases}
\w_l(p,h)\cdot \w_r(q,h)   & \text{if $p > 0$,}\\
x~\w_l(p,h)\cdot \w_r(q,h+1) & \text{if $(p,q) = (0,0)$,}\\
\tilde{x}~\w_l(p,h)\cdot \w_r(q,h+1) & \text{if $(p,q) = (0,h+1)$,}\\
\w_l(p,h)\cdot \w_r(q,h+1) & \text{otherwise.}\\
\end{cases}
\end{align*}
We derive \eqref{eq:coeff} from \eqref{eq:weightsum}.
\qed

\medskip
\rmk
It is instructive to look at  the weight of each step
from the corresponding pignose diagram. For $x=\tilde{x}=1$, every vertex is connected to only one half-arc with its weight in one of following six ways:
$$
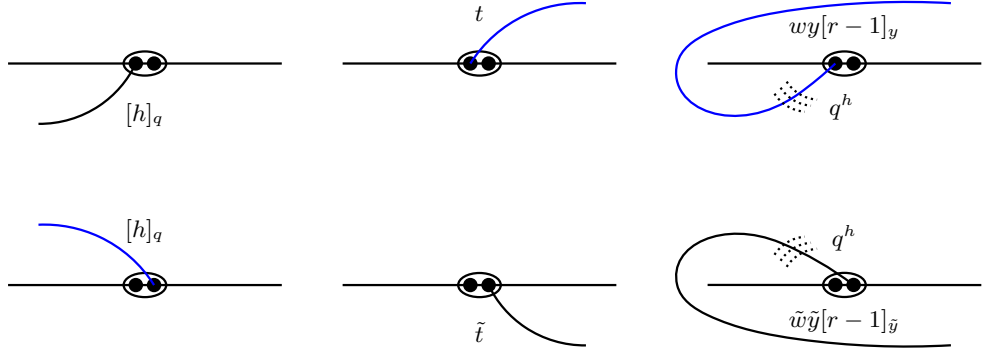
\centering
\begin{pgfpicture}{-30.00mm}{5.71mm}{102.00mm}{26.16mm}
\pgfsetxvec{\pgfpoint{0.80mm}{0mm}}
\pgfsetyvec{\pgfpoint{0mm}{0.80mm}}
\color[rgb]{0,0,0}\pgfsetlinewidth{0.30mm}\pgfsetdash{}{0mm}
\pgfmoveto{\pgfxy(20.00,20.00)}\pgflineto{\pgfxy(65.00,20.00)}\pgfstroke
\pgfcircle[fill]{\pgfxy(41.00,20.00)}{0.80mm}
\pgfcircle[stroke]{\pgfxy(41.00,20.00)}{0.80mm}
\pgfcircle[fill]{\pgfxy(44.00,20.00)}{0.80mm}
\pgfcircle[stroke]{\pgfxy(44.00,20.00)}{0.80mm}
\pgfellipse[stroke]{\pgfxy(42.50,20.00)}{\pgfxy(3.50,0.00)}{\pgfxy(0.00,2.00)}
\color[rgb]{0,0,1}\pgfmoveto{\pgfxy(41.00,20.00)}\pgfcurveto{\pgfxy(43.17,23.48)}{\pgfxy(46.29,26.25)}{\pgfxy(50.00,28.00)}\pgfcurveto{\pgfxy(53.12,29.47)}{\pgfxy(56.55,30.16)}{\pgfxy(60.00,30.00)}\pgfstroke
\color[rgb]{0,0,0}\pgfputat{\pgfxy(42.50,27.00)}{\pgfbox[bottom,left]{\fontsize{9.10}{10.93}\selectfont \makebox[0pt]{$t$}}}
\pgfmoveto{\pgfxy(-35.00,20.00)}\pgflineto{\pgfxy(10.00,20.00)}\pgfstroke
\pgfcircle[fill]{\pgfxy(-14.00,20.00)}{0.80mm}
\pgfcircle[stroke]{\pgfxy(-14.00,20.00)}{0.80mm}
\pgfcircle[fill]{\pgfxy(-11.00,20.00)}{0.80mm}
\pgfcircle[stroke]{\pgfxy(-11.00,20.00)}{0.80mm}
\pgfellipse[stroke]{\pgfxy(-12.50,20.00)}{\pgfxy(3.50,0.00)}{\pgfxy(0.00,2.00)}
\pgfmoveto{\pgfxy(-14.00,20.00)}\pgfcurveto{\pgfxy(-15.34,17.17)}{\pgfxy(-17.41,14.75)}{\pgfxy(-20.00,13.00)}\pgfcurveto{\pgfxy(-22.95,11.01)}{\pgfxy(-26.44,9.96)}{\pgfxy(-30.00,10.00)}\pgfstroke
\pgfputat{\pgfxy(-12.50,10.50)}{\pgfbox[bottom,left]{\fontsize{9.10}{10.93}\selectfont \makebox[0pt]{$[h]_q$}}}
\pgfmoveto{\pgfxy(80.00,20.00)}\pgflineto{\pgfxy(125.00,20.00)}\pgfstroke
\pgfcircle[fill]{\pgfxy(101.00,20.00)}{0.80mm}
\pgfcircle[stroke]{\pgfxy(101.00,20.00)}{0.80mm}
\pgfcircle[fill]{\pgfxy(104.00,20.00)}{0.80mm}
\pgfcircle[stroke]{\pgfxy(104.00,20.00)}{0.80mm}
\pgfellipse[stroke]{\pgfxy(102.50,20.00)}{\pgfxy(3.50,0.00)}{\pgfxy(0.00,2.00)}
\color[rgb]{0,0,1}\pgfmoveto{\pgfxy(101.00,20.00)}\pgfcurveto{\pgfxy(99.12,18.20)}{\pgfxy(97.11,16.53)}{\pgfxy(95.00,15.00)}\pgfcurveto{\pgfxy(84.69,7.55)}{\pgfxy(73.23,12.52)}{\pgfxy(75.00,20.00)}\pgfcurveto{\pgfxy(75.57,22.41)}{\pgfxy(77.76,23.91)}{\pgfxy(80.00,25.00)}\pgfcurveto{\pgfxy(84.69,27.27)}{\pgfxy(89.84,28.30)}{\pgfxy(95.00,29.00)}\pgfcurveto{\pgfxy(103.28,30.13)}{\pgfxy(111.65,30.47)}{\pgfxy(120.00,30.00)}\pgfstroke
\color[rgb]{0,0,0}\pgfsetdash{{0.30mm}{0.50mm}}{0mm}\pgfmoveto{\pgfxy(93.00,17.00)}\pgfcurveto{\pgfxy(93.51,16.19)}{\pgfxy(94.19,15.51)}{\pgfxy(95.00,15.00)}\pgfcurveto{\pgfxy(95.90,14.43)}{\pgfxy(96.93,14.08)}{\pgfxy(98.00,14.00)}\pgfstroke
\pgfmoveto{\pgfxy(92.00,16.00)}\pgfcurveto{\pgfxy(92.51,15.19)}{\pgfxy(93.19,14.51)}{\pgfxy(94.00,14.00)}\pgfcurveto{\pgfxy(94.90,13.43)}{\pgfxy(95.93,13.08)}{\pgfxy(97.00,13.00)}\pgfstroke
\pgfmoveto{\pgfxy(91.00,15.00)}\pgfcurveto{\pgfxy(91.51,14.19)}{\pgfxy(92.19,13.51)}{\pgfxy(93.00,13.00)}\pgfcurveto{\pgfxy(93.90,12.43)}{\pgfxy(94.93,12.08)}{\pgfxy(96.00,12.00)}\pgfstroke
\pgfputat{\pgfxy(100.00,11.50)}{\pgfbox[bottom,left]{\fontsize{9.10}{10.93}\selectfont $q^h$}}
\pgfputat{\pgfxy(102.50,25.00)}{\pgfbox[bottom,left]{\fontsize{9.10}{10.93}\selectfont \makebox[0pt]{$w y[r-1]_{y}$}}}
\end{pgfpicture}%
$$

$$
%\input{right.TpX}
%<TpX v="5" TeXFormat="pgf" PdfTeXFormat="pgf" ArrowsSize="0.7" StarsSize="1" DefaultFontHeight="4" DefaultSymbolSize="30" ApproximationPrecision="0.01" PicScale="0.8" Border="2" BitmapRes="20000" HatchingStep="2" DottedSize="0.5" DashSize="1" LineWidth="0.3" TeXCenterFigure="1" TeXFigure="none">
%  <line x1="20" y1="20" x2="65" y2="20"/>
%  <star x="41" y="20"/>
%  <star x="44" y="20"/>
%  <ellipse x="42.5" y="20" dx="7" dy="4"/>
%  <curve>44,20 50,13 60,10</curve>
%  <text x="42.5" y="10.5" t="t2" tex="$\tilde{t}$" h="4" halign="c"/>
%  <line x1="-35" y1="20" x2="10" y2="20"/>
%  <star x="-14" y="20"/>
%  <star x="-11" y="20"/>
%  <ellipse x="-12.5" y="20" dx="7" dy="4"/>
%  <curve lc="blue">-11,20 -20,28 -30,30</curve>
%  <text x="-12.5" y="28" t="[h]_q" tex="$[h]_q$" h="4" halign="c"/>
%  <line x1="80" y1="20.0201" x2="125" y2="20"/>
%  <star x="101" y="20.0201"/>
%  <star x="104" y="20.0201"/>
%  <ellipse x="102.5" y="20.0201" dx="7" dy="4"/>
%  <curve>104,20 92,27 75,20 80,15.0201 95,11.0201 120,10.0201</curve>
%  <curve li="dot">93,23.0201 95,25.0201 98,26.0201</curve>
%  <curve>92,16.5</curve>
%  <curve li="dot">92,24.0201 94,26.0201 97,27.0201</curve>
%  <curve li="dot">91,25.0201 93,27.0201 96,28.0201</curve>
%  <text x="102.5" y="26.5" t="q^h" tex="$q^h$" h="4" halign="c"/>
%  <text x="102.5" y="13" t="y2" tex="$\tilde{w} \tilde{y}[r-1]_{\tilde{y}}$" h="4" halign="c"/>
%</TpX>
\centering
\begin{pgfpicture}{-30.00mm}{5.71mm}{102.00mm}{26.91mm}
\pgfsetxvec{\pgfpoint{0.80mm}{0mm}}
\pgfsetyvec{\pgfpoint{0mm}{0.80mm}}
\color[rgb]{0,0,0}\pgfsetlinewidth{0.30mm}\pgfsetdash{}{0mm}
\pgfmoveto{\pgfxy(20.00,20.00)}\pgflineto{\pgfxy(65.00,20.00)}\pgfstroke
\pgfcircle[fill]{\pgfxy(41.00,20.00)}{0.80mm}
\pgfcircle[stroke]{\pgfxy(41.00,20.00)}{0.80mm}
\pgfcircle[fill]{\pgfxy(44.00,20.00)}{0.80mm}
\pgfcircle[stroke]{\pgfxy(44.00,20.00)}{0.80mm}
\pgfellipse[stroke]{\pgfxy(42.50,20.00)}{\pgfxy(3.50,0.00)}{\pgfxy(0.00,2.00)}
\pgfmoveto{\pgfxy(44.00,20.00)}\pgfcurveto{\pgfxy(45.34,17.17)}{\pgfxy(47.41,14.75)}{\pgfxy(50.00,13.00)}\pgfcurveto{\pgfxy(52.95,11.01)}{\pgfxy(56.44,9.96)}{\pgfxy(60.00,10.00)}\pgfstroke
\pgfputat{\pgfxy(42.50,10.50)}{\pgfbox[bottom,left]{\fontsize{9.10}{10.93}\selectfont \makebox[0pt]{$\tilde{t}$}}}
\pgfmoveto{\pgfxy(-35.00,20.00)}\pgflineto{\pgfxy(10.00,20.00)}\pgfstroke
\pgfcircle[fill]{\pgfxy(-14.00,20.00)}{0.80mm}
\pgfcircle[stroke]{\pgfxy(-14.00,20.00)}{0.80mm}
\pgfcircle[fill]{\pgfxy(-11.00,20.00)}{0.80mm}
\pgfcircle[stroke]{\pgfxy(-11.00,20.00)}{0.80mm}
\pgfellipse[stroke]{\pgfxy(-12.50,20.00)}{\pgfxy(3.50,0.00)}{\pgfxy(0.00,2.00)}
\color[rgb]{0,0,1}\pgfmoveto{\pgfxy(-11.00,20.00)}\pgfcurveto{\pgfxy(-13.17,23.48)}{\pgfxy(-16.29,26.25)}{\pgfxy(-20.00,28.00)}\pgfcurveto{\pgfxy(-23.12,29.47)}{\pgfxy(-26.55,30.16)}{\pgfxy(-30.00,30.00)}\pgfstroke
\color[rgb]{0,0,0}\pgfputat{\pgfxy(-12.50,28.00)}{\pgfbox[bottom,left]{\fontsize{9.10}{10.93}\selectfont \makebox[0pt]{$[h]_q$}}}
\pgfmoveto{\pgfxy(80.00,20.02)}\pgflineto{\pgfxy(125.00,20.00)}\pgfstroke
\pgfcircle[fill]{\pgfxy(101.00,20.02)}{0.80mm}
\pgfcircle[stroke]{\pgfxy(101.00,20.02)}{0.80mm}
\pgfcircle[fill]{\pgfxy(104.00,20.02)}{0.80mm}
\pgfcircle[stroke]{\pgfxy(104.00,20.02)}{0.80mm}
\pgfellipse[stroke]{\pgfxy(102.50,20.02)}{\pgfxy(3.50,0.00)}{\pgfxy(0.00,2.00)}
\pgfmoveto{\pgfxy(104.00,20.00)}\pgfcurveto{\pgfxy(100.32,22.84)}{\pgfxy(96.28,25.19)}{\pgfxy(92.00,27.00)}\pgfcurveto{\pgfxy(82.68,30.94)}{\pgfxy(73.44,26.67)}{\pgfxy(75.00,20.00)}\pgfcurveto{\pgfxy(75.56,17.59)}{\pgfxy(77.75,16.10)}{\pgfxy(80.00,15.02)}\pgfcurveto{\pgfxy(84.70,12.76)}{\pgfxy(89.84,11.73)}{\pgfxy(95.00,11.02)}\pgfcurveto{\pgfxy(103.28,9.88)}{\pgfxy(111.65,9.54)}{\pgfxy(120.00,10.02)}\pgfstroke
\pgfsetdash{{0.30mm}{0.50mm}}{0mm}\pgfmoveto{\pgfxy(93.00,23.02)}\pgfcurveto{\pgfxy(93.51,23.83)}{\pgfxy(94.19,24.51)}{\pgfxy(95.00,25.02)}\pgfcurveto{\pgfxy(95.90,25.59)}{\pgfxy(96.93,25.94)}{\pgfxy(98.00,26.02)}\pgfstroke
\pgfmoveto{\pgfxy(92.00,24.02)}\pgfcurveto{\pgfxy(92.51,24.83)}{\pgfxy(93.19,25.51)}{\pgfxy(94.00,26.02)}\pgfcurveto{\pgfxy(94.90,26.59)}{\pgfxy(95.93,26.94)}{\pgfxy(97.00,27.02)}\pgfstroke
\pgfmoveto{\pgfxy(91.00,25.02)}\pgfcurveto{\pgfxy(91.51,25.83)}{\pgfxy(92.19,26.51)}{\pgfxy(93.00,27.02)}\pgfcurveto{\pgfxy(93.90,27.59)}{\pgfxy(94.93,27.94)}{\pgfxy(96.00,28.02)}\pgfstroke
\pgfputat{\pgfxy(102.50,26.50)}{\pgfbox[bottom,left]{\fontsize{9.10}{10.93}\selectfont \makebox[0pt]{$q^h$}}}
\pgfputat{\pgfxy(102.50,13.00)}{\pgfbox[bottom,left]{\fontsize{9.10}{10.93}\selectfont \makebox[0pt]{$\tilde{w} \tilde{y}[r-1]_{\tilde{y}}$}}}
\end{pgfpicture}%
$$
Here the parameter $h$ stands for the number of open half-arcs, which do not become complete-arcs, above horizontal line in the partial pignose diagram.

\section*{Acknowledgements}
The research of the first author was supported by Basic Science Research Program through the National Research Foundation of Korea (NRF) funded by the Ministry of Science, ICT \& Future Planning (NRF-2012R1A1A1014154) and INHA UNIVERSITY Research Grant.

% ----------------------------------------------------------------
\end{document}